\documentclass[11pt, a4paper, USenglish,abstract=true]{scrartcl}
\usepackage{tcolorbox}
\usepackage{enumitem}
\usepackage[utf8]{inputenc}
\usepackage{dsfont}
\usepackage{amssymb,amsthm,amsmath}
\usepackage{amsfonts}
\usepackage{comment} 
\usepackage[usestackEOL]{stackengine} 
\let\oldmathchoice\mathchoice
\usepackage{mathtools} 
\usepackage{mathrsfs} 
\usepackage{scalerel}
\usepackage{breqn}
\let\newmathchoice\mathchoice
\usepackage{microtype}
\usepackage{graphicx}
\usepackage{float}
\usepackage[USenglish]{babel}
\usepackage[T1]{fontenc}
\usepackage{lmodern}
\usepackage{pdfpages}
\usepackage{caption}
\usepackage[titletoc]{appendix} 
\usepackage{subcaption}
\usepackage{physics} 
\usepackage{setspace} 
\usepackage{tikz}
\usepackage{pgfplots} 
\usepackage[textsize=tiny]{todonotes}
\usepackage{bm} 
\usepackage{marginnote}

\usepackage{titling} 
\usepackage[final]{showlabels} 
\usepackage{cancel}
\usetikzlibrary{patterns}
\usetikzlibrary{cd, babel}
\usepackage{csquotes}
\usepackage[%
  style=numeric, 
  sortcites,
  backend=biber,
  giveninits=true, 
  isbn= false,
  url=false,
  doi = false,
  eprint=false,
  date=year,
  noerroretextools,
  maxcitenames=99, 
    maxbibnames=99   
  ]{biblatex}

\usepackage[colorlinks=true,linkcolor=blue,citecolor=blue,hypertexnames=false]{hyperref}

\usepackage{cleveref}
\usepackage{zref-clever}
\zcsetup{nameinlink=false, cap, noabbrev}

\theoremstyle{plain}

\newtheorem{prop}{Proposition}[section]
\AddToHook{env/prop/begin}{%
   \zcsetup{countertype={prop=prop}}}
\zcRefTypeSetup{prop}{Name-sg=Proposition} 
\newtheorem*{prop*}{Proposition}
\newtheorem{lemma}[prop]{Lemma}
\AddToHook{env/lemma/begin}{%
   \zcsetup{countertype={prop=lemma}}}
\zcRefTypeSetup{lemma}{Name-sg=Lemma} 

\newtheorem{thm}[prop]{Theorem}
\AddToHook{env/thm/begin}{%
   \zcsetup{countertype={prop=thm}}}
\zcRefTypeSetup{thm}{Name-sg=Theorem}

\newtheorem{cor}[prop]{Corollary}
\AddToHook{env/cor/begin}{%
   \zcsetup{countertype={prop=cor}}}
\zcRefTypeSetup{cor}{Name-sg=Corollary}

\theoremstyle{definition}

\newtheorem{defi}[prop]{Definition}
\AddToHook{env/defi/begin}{%
   \zcsetup{countertype={prop=defi}}}
\zcRefTypeSetup{defi}{Name-sg=Definition}

\newtheorem{remark}[prop]{Remark}
\AddToHook{env/remark/begin}{%
   \zcsetup{countertype={prop=remark}}}
\zcRefTypeSetup{remark}{Name-sg=Remark} 

\theoremstyle{remark}

\newcommand{\nocontentsline}[3]{}
\newcommand{\tocless}[2]{\bgroup\let\addcontentsline=\nocontentsline#1{#2}\egroup}

\newcommand{\tc}[2]{\textcolor{#1}{#2}}
\let\temp\phi
\let\phi\varphi
\let\varphi\temp
\newcommand*\dif{\,\mathrm{d}}

\DeclareMathOperator{\cqq}{\coloneqq}
\DeclareMathOperator{\qqc}{\eqqcolon}
\DeclareMathOperator{\wto}{\rightharpoonup}

\newcommand{\vertiii}[1]{{\left\vert\kern-0.25ex\left\vert\kern-0.25ex\left\vert #1 
    \right\vert\kern-0.25ex\right\vert\kern-0.25ex\right\vert}}

\newcommand\N{\mathbb{N}}

\newcommand\R{\mathbb{R}}
\newcommand\Z{\mathbb{Z}}
\newcommand\C{\mathbb{C}}

\newcommand{\mc}[1]{\mathcal #1}

\newcommand{\mb}[1]{\mathbb #1}

\DeclareMathOperator{\PSL}{PSL}

\newcommand{\parpar}[2]{\partial_{#1} #2}
\newcommand{\frpar}[2]{\frac{\partial #1}{\partial #2}}
\newcommand{\eps}{\varepsilon}

\newcommand{\limk}{\lim _{k\to \infty}}
\DeclareMathOperator{\im}{im}

\DeclareMathOperator{\Div}{div}
\DeclareMathOperator{\loc}{loc}

\def\dashint{\let\mathchoice\oldmathchoice\,\ThisStyle{\ensurestackMath{%
            \stackinset{c}{.2\LMpt}{c}{.5\LMpt}{\SavedStyle-}{%
            \SavedStyle\phantom{\int}}}%
        \setbox0=\hbox{$\SavedStyle\int\,$}\kern-\wd0}\int%
        \let\mathchoice\newmathchoice}

	\newcommand{\mres}{\mathbin{\vrule height 1.6ex depth 0pt width
0.13ex\vrule height 0.13ex depth 0pt width 1.3ex}}

\DeclareMathOperator{\diam}{diam}
\DeclareMathOperator{\genus}{genus}
\newcommand*{\vol}{\mathrm{vol}}

\newcommand{\hodge}{{\star}}

\DeclareMathOperator\arsinh{arsinh}

\DeclareMathOperator\spt{spt}

\DeclareMathOperator\injrad{injrad}
\DeclareMathOperator\Log{Log}

\DeclareMathOperator\inv{inv}
\newcommand{\tilvec}[1]{\tilde{\vec{#1}}}

\newcommand{\extp}{\mathchoice{{\textstyle\bigwedge}}%
    {{\bigwedge}}%
    {{\textstyle\wedge}}%
    {{\scriptstyle\wedge}}}
\usepackage{geometry}
\usepackage{xfrac}
\setkomafont{disposition}{\normalcolor\bfseries}
 \DeclareNameAlias{default}{family-given}
 \DefineBibliographyStrings{english}{%
 page             = {},
 pages            = {},
 }  
 \DeclareFieldFormat[article]{citetitle}{#1}
 \DeclareFieldFormat[article]{title}{#1} 
\numberwithin{equation}{section}

\usepackage{bigdelim}

\definecolor{myblue}{HTML}{18B7CF}
\definecolor{mygreen}{HTML}{56A039}
\definecolor{myred}{HTML}{852B2B}

\makeatletter
\def\namedlabel#1#2{\begingroup
    #2%
    \def\@currentlabel{#2}%
    \phantomsection\label{#1}\endgroup
}

\usepackage[normalem]{ulem} 
\let\etoolboxforlistloop\forlistloop 
\usepackage{autonum}
\let\forlistloop\etoolboxforlistloop 

\author{\textsc{Christian Scharrer}\thanks{Institute for Applied Mathematics, University of Bonn, Endenicher Allee 60, 53115 Bonn, Germany. \texttt{scharrer@iam.uni-bonn.de}}\and \textsc{Manuel Schlierf}\thanks{Mathematics department, Salzburg University, Hellbrunner Str.\ 34, 5020 Salzburg, Austria. \texttt{math@manuelschlierf.info}}\and \textsc{Alexander West}\thanks{Institute for Applied Mathematics, University of Bonn, Endenicher Allee 60, 53115 Bonn, Germany. \texttt{west@iam.uni-bonn.de}}}
\date{\today}
\title{Bubble classification of immersions at the boundary of the moduli space with \texorpdfstring{$8\pi$}{8pi} Willmore energy}

\begin{document}
\maketitle
\begin{abstract}
We study the asymptotic bubbling behavior of sequences of weak genus-$p$ immersions with diverging conformal classes and limiting Willmore energy of $8\pi$. After applying suitable Möbius transformations, in a strong $W^{2,2}_{\loc}$-limit, we obtain two round spheres at the largest scale and $p+1$ catenoids at the smallest scales. Moreover, we apply this classification to sequences of isoperimetrically, conformally and normalized-total-mean-curvature constrained Willmore minimizers when the constraints approach the boundary of the domain where minimizers exist, respectively.
\end{abstract}

\section{Introduction}

Given a smooth immersion $\vec \Phi\colon\Sigma\to\R^n$ of a closed, oriented, and connected surface $\Sigma$ with genus $p$, we define its Willmore energy as well as the Dirichlet energy as
\begin{equation}
	\mc W(\vec \Phi)\cqq \int_{\Sigma} |\vec H|^2 \dif{\mu},\quad \mc E (\vec \Phi) \cqq \int_{\Sigma} |\vec {\mb I}|_{g}^2 \dif{\mu}.\label{Willmore energy and Dirichlet energy}
\end{equation} 
Here, $g= g_{ij}\, d x^i \otimes d x^j \cqq \vec \Phi^* g_{\R^n} \in \Gamma(T^* \Sigma \otimes T^*\Sigma)$ denotes the pullback metric of the standard inner product $g_{\R^n}$ in $\R^n$ via $\vec \Phi$, and $g^{-1} = g^{ij} \frpar{}{x^i}\otimes \frpar{}{x^j} \in \Gamma(T\Sigma \otimes T \Sigma)$ its inverse. The \emph{Riemannian measure} and \emph{mean curvature} induced by $\vec\Phi$ are $\dif{\mu} \cqq \sqrt{\det g} \,d x^1 \wedge d x^2$ and
\begin{equation}
	\vec H \cqq \frac{1}{2} \tr_g(\vec {\mb I}) = \frac{1}{2} g^{ij}\vec {\mb I} _{ij}, \label{scalar mean curvature}
\end{equation}
where $\vec {\mb I} \in \Gamma(T^* \Sigma \otimes T^*\Sigma\otimes \R^n)$ with
$\vec {\mb I}_{ij}\cqq (\partial_{x^i}\partial _{x^j} \vec \Phi)^\perp$ is the \emph{second fundamental form}. The orthogonal projection ${}^\perp$ is given by ${\vec v}^\perp \cqq  {\vec v} - g^{ij} \langle {\vec v}, \partial _{x^i}\vec \Phi\rangle \partial _{x^j} \vec \Phi $. If $n=3$, we also define the \emph{Gauss map} as well as the \emph{scalar mean curvature} as
\begin{equation}
	\vec n  \cqq \frac{\parpar{x^1}{\vec \Phi}\times \parpar{x^2}{\vec \Phi}}{|\parpar{x^1}{\vec \Phi}\times \parpar{x^2}{\vec \Phi}|},\quad H\cqq \langle \vec H, \vec n\rangle\label{Gauss map}
\end{equation}
in any positive chart $x$. As a consequence of the Gauss--Bonnet theorem, these energies are related through
\[\mc E(\vec \Phi) = 4\mc W(\vec \Phi) - 4\pi \chi(\Sigma) = 4\mc W(\vec \Phi) - 8\pi(1-p),\]
where $\chi(\Sigma) = 2(1-p)$ is the Euler-characteristic of $\Sigma$. We denote the set of smooth immersions of a genus-$p$ surface into Euclidean $n$-space with $\mc S_p^n$. As already observed by Blaschke \cite{BlaschkeThomsen}, both functionals $\mc W$ and $\mc E$ are invariant under conformal transformations of the ambient space:
\begin{equation}\label{eq:intro:conf_invariance}
	\mc W(\Xi\circ\vec\Phi)=\mc W(\vec \Phi)\quad \text{for all Möbius transformations $\Xi$ of $\R^n$ such that $\Xi \circ \vec \Phi \in \mc S_p^n$}.
\end{equation}

A classical problem in geometric analysis addresses the question whether the functional $\mc W$ admits a smooth minimizer. Equivalently, one studies whether the infimum 
\begin{equation}\label{eq:intro:WP}
	\beta_p^n\cqq \inf_{\vec \Phi\in\mc S_p^n}\mc W(\vec \Phi)
\end{equation} 
is attained. Regarding the spherical case, Willmore \cite{Willmore,Willmore1970} discovered that the energy now bearing his name satisfies $\mc W(\vec \Phi)\ge 4\pi$ on $\mc S_p^n$ with equality if and only if $p=0$ and $\vec\Phi$ is totally umbilic. Willmore's inequality was further refined by Li--Yau \cite{LiYau} who showed that
\begin{equation}\label{eq:intro:LY}
	\mc W(\vec \Phi)\ge 4\pi  \cdot \# (\vec\Phi^{-1}(y))\quad\text{for all $y\in\R^n$},
\end{equation}
where $\#$ denotes the counting measure. In particular, if $\mc W(\vec\Phi)<8\pi$, then $\vec\Phi$ is an embedding. 

In a pioneering work, Simon \cite{Simon} proved the existence of a $\mc W$-minimal torus in $\R^n$, solving the Willmore problem \eqref{eq:intro:WP} for $p=1$. While compactness for the direct method in the calculus of variations can be achieved in some weaker spaces, a main difficulty of the minimization problem \eqref{eq:intro:WP} is to prove that the genus of a minimizing sequence is preserved in the limit. In other words, after passing to a subsequence, the conformal classes induced by a minimizing sequence have to remain within a compact subset of the moduli space. Proven by Simon \cite{Simon}, this can be guaranteed through the condition 
\begin{equation}\label{eq:intro:Douglas_condition}
	\beta_p^n<\min\{8\pi,\omega_p^n\},
\end{equation}
where 
\begin{equation}\label{eq:intro:Douglas_condition:omega}
	\omega_p^n\cqq \inf \left \{4\pi + \sum_{i=1}^n (\beta^n_{p_i} - 4\pi), \, \sum_{i=1}^n p_i = p \text{ and } 0<p_i<p\right \}. \label{omega 3 p}
\end{equation}
The number on the right of \eqref{eq:intro:Douglas_condition} gives a sharp lower bound for the Willmore functional on the space of nodal surfaces that arise as Deligne--Mumford limits from sequences of genus-$p$ surfaces with degenerating conformal classes \cite{KuwertLi,RiviereDegeneratingImmersions}. Thus, denoting the conformal class induced by $\vec\Phi\in \mc S^n_p$ with $c(\vec\Phi)$ and letting $\delta>0$,
\begin{equation}\label{eq:intro:compact_moduli}
	 \bigl \{ c(\vec\Phi),\,\vec\Phi\in\mc S_p^n \text{ and }\mc W(\vec \Phi)<\min\{8\pi,\omega_p^n\} - \delta  \bigr\}
\end{equation}
is a bounded subset of the moduli space.

By resolving the \emph{Willmore conjecture}, Marques--Neves \cite{MarquesNeves} proved
\begin{equation}
	\beta_1^3=2\pi^2,\qquad \beta_p^3>2\pi^2\quad\text{for $p\ge2$}.
\end{equation}
Consequently, given that $2\pi^2>6\pi$, one infers  
\begin{equation}\label{eq:intro:MND}
	\omega_p^3>8\pi.
\end{equation}
Thus, for $n=3$, \eqref{eq:intro:Douglas_condition} becomes 
\begin{equation}\label{eq:intro:Douglas_condition:codim1}
	\beta_p^3<8\pi.
\end{equation}
Notice that unlike \eqref{eq:intro:Douglas_condition:omega}, the right hand side of \eqref{eq:intro:Douglas_condition:codim1} does not depend on $p$.

Recalling the Möbius invariance \eqref{eq:intro:conf_invariance}, it was believed that for any sequence $\vec\Phi_k$ in $\mc S_p^n$ with  
\begin{equation}\label{eq:intro:8pi-energy-limit}
	\lim_{k\to\infty}\mc W(\vec \Phi_k)= 8\pi
\end{equation}
and degenerating induced conformal classes, there exists a sequence of Möbius transformations $\Xi_k$ of the ambient space such that, after passing to a subsequence, the varifold limit of $\Xi_k\circ\vec\Phi_k$ is given by the union of two intersecting spheres. This was observed e.g.\ in numerical experiments for the stereographic projections of tori with constant mean curvature in $\mathbb S^3$ \cite{KilianSchmidtSchmitt}. Blowing up at suitable intersection points of the two spheres, one should see $p+1$ catenoids, cf.\ \zcref{Caption 1}. The fact that the nodal surface consists of only spherical components is very intuitive considering that in codimension one, the Douglas-type condition \eqref{eq:intro:MND} simplifies in a way that it does not depend on the genus anymore. The purpose of the present article is to provide a proof of this conjecture. Notice here that in view of \eqref{eq:intro:LY} and \eqref{eq:intro:compact_moduli}, the property \eqref{eq:intro:8pi-energy-limit} naturally arises in the study of embeddings with degenerating conformal classes.

Previously, to the best of the authors' knowledge, no results in that direction have been known. A major difficulty arises from conformal invariance. Without the usage of suitable Möbius transformations, one cannot prevent that a sequence satisfying \eqref{eq:intro:8pi-energy-limit} converges as varifolds to the unit sphere, e.g.\ see Type 2 of \zcref{fig:types}. From a purely ambient view, all information will be lost in the limit, contracted into a single concentration point. We will rely on a local parametric approach, which by a recent breakthrough of Bernard--Rivière \cite{BernardRiviere} and Laurain--Rivière \cite{LaurainRiviereEnergyQuantization} allows to identify blow-up limits near concentration points. Thereby, all the geometric information contracting to a point in the limit can be recovered through blow-ups of different scales $s_k$. Instead of $\mc W$-critical immersions, we will more generally consider weak immersions as developed for the variational framework of the Willmore problem by Kuwert--Li \cite{KuwertLi} and Rivière \cite{RiviereCrelle}. Fixing a reference metric $g_0$ on $\Sigma$, the space of \emph{weak immersions} from $\Sigma$ into Euclidean three space is defined as
\begin{equation}
	\mc E_\Sigma\cqq \{\vec\Phi\in W^{2,2}(\Sigma;\R^3),\,\text{$cg_0\le\vec\Phi ^* g_{\R^3}\le c^{-1}g_0$ for some $c=c(\vec\Phi)>0$}\}.
\end{equation}
Similarly, we define the space $\mathcal F_\Sigma$ of \emph{conformal, branched} weak immersions with \emph{$L^2$-bounded second fundamental form} (with respect to the reference metric $g_0$) to consist of all maps $\vec\Phi\colon\Sigma\to\R^3$ for which there exist finitely many singular points $a_1,\ldots,a_N\in\Sigma$ later distinguished as either \emph{ends} or \emph{branch points}, and $\alpha \in L^\infty_{\loc}(\Sigma\setminus\{a_1,\ldots,a_N\})$ with 
\begin{equation}
	\vec\Phi\in W^{2,2}_{\mathrm{loc}}(\Sigma\setminus\{a_1,\ldots,a_N\};\R^3),\quad \vec n \in W^{1,2}(\Sigma;\R^3),\quad\vec\Phi^*g_{\R^3}=e^{2\alpha}g_0.
\end{equation}
Notice that by \cite[Theorem 5.4]{RiviereLectureNotes}, it holds
\begin{equation}
    \|d\alpha \|_{L^{2,\infty}_{g_0}(\Sigma) } \leq C\label{L2 infinity estimate for fixed immersion}
\end{equation}
for a constant $C$ depending on $\mc E(\vec \Phi)$ and $g_0$ and where the Lorentz space $L^{2,\infty}$ is defined in \eqref{Lorentz space}.

A key observation to prove our main goal \zcref{thm:Asymptotic convergence} is the fact that under condition \eqref{eq:intro:8pi-energy-limit}, any non-injective limiting immersion attains equality in \eqref{eq:intro:LY}. Such immersions are characterized by the following theorem proven in Appendix \hyperref[sec:Appendix Equality in the Li--Yau inequality]{C}.
\begin{thm}\label{thm:equality-in-Li-Yau}
	Suppose $x_0\in \R^n$ and $\mu$ is an integral $2$-varifold in $\R^n$ with generalized mean curvature $\vec H\in L^2(\mu;\R^n)$ as defined in \eqref{divergence theorem} and $\Theta^2(\mu, \infty) = 0$. Then, 
	\begin{equation}
		\Theta^2(\mu,x_0) 
		=\frac{1}{4\pi}\int_{\R^n}|\vec H|^2\,\mathrm d\mu
	\end{equation}
	if and only if $\nu\cqq {I_{x_0}}_\#\mu$ (for $I_{x_0}(x)\cqq \frac{x-x_0}{|x-x_0|^2}$) is a stationary integral $2$-varifold in $\R^n$. Moreover, $\Theta^2(\mu,x_0)=\Theta^2(\nu,\infty)$. 
\end{thm}
Under the condition \eqref{eq:intro:8pi-energy-limit}, this allows to exclude true branch points in \zcref{prop: no density 2}. Consequently, all minimal surfaces that arise as blow-ups from our sequence $\vec\Phi_k$ with degenerating conformal classes have to be planes or catenoids, see \zcref{lem:one catenoid}. While crucial for our analysis in order to prove \zcref{thm:Asymptotic convergence}, \zcref{thm:equality-in-Li-Yau} is already interesting on its own. For instance, multi $m$-bubbles that arise as the inversion of integral stationary cones attain equality in \eqref{Li Yau inequality}, cf.\ \cite[Lemma 3.3]{RuppScharrerGlobalVarifoldRegularity}. Moreover, in view of \cite[Theorem 5.2 (2)]{Allard}, \zcref{thm:equality-in-Li-Yau} gives an alternative proof of the fact that any integral 2-varifold with Willmore energy of $4\pi$ and density 0 at infinity has to be a round sphere, e.g.\ see \cite[Proposition~5.1]{NovagaPozzetta}.

\subsection{Analysis at the boundary of the moduli space}

Denote the set of complex structures on $\Sigma$ with $\mc C$ and recall that each $\vec\Phi\in\mathcal E_\Sigma$ induces a member of $\mc C$, cf.\ 
\cite[Corollary~IV.5]{LanMartinoRiviereSurvey}. The family $\mc D$ of smooth diffeomorphisms on $\Sigma$ acts on $\mc C$ via 
\begin{equation}
    f^* c \cqq \{(f^{-1}(U_i),\, \omega^{-1}_i \circ f)\}\quad\text{for $f\in \mc D$ and $c = \{(U_i, \omega^{-1}_i)\} \in \mc C$}.
\end{equation}
We define the \emph{moduli space}
\begin{equation}
	\mc M_p \cqq \mc C/\mc D.\label{definition moduli space}
\end{equation}
Elements of $\mc M_p$ are called \emph{conformal classes}. For $p=0$, this is a single point by the uniformization theorem. For $p\geq 1$, each complex structure induces a unique conformal (with respect to the complex structure) so-called \emph{Poincar\'e metric} $h$ of constant sectional curvature equal to $0$ if $p=1$ and $-1$ if $p\geq 2$ such that $\vol_h(\Sigma)=1$ if $p=1$. The set of Poincar\'e metrics is denoted by $\mc M_{\text{Poin}}$. It follows that $\mc M_p$ can be identified with $\mc M_{\text{Poin}}/\mc D$ as sets, where $\mc D$ acts on $\mc M_{\text{Poin}}$ by $(h,f) \mapsto f^*h$. Moreover, $\mc M_p$ can be equipped with a topology, see \cite{TrombaTeichmullerTheory,Jost,Gromov}. For a characterization of degenerating conformal classes, see Sections \ref{subsec:Mumford-Deligne compactness} and \ref{subsec:diverging conformal classes torus}. 

We can now state a first (simplified) version of our main result, cf.\ \zcref{Caption 1} below for an illustration and \zcref{thm:Asymptotic convergence} for the extended version.

\begin{thm}
	\label{thm:Main theorem simplified}
	Suppose $p$ is a positive integer, $\Sigma$ is a closed, connected, oriented surface of genus $p$, $\vec \Phi_k\colon\Sigma\to \R^3$ is a sequence of weak immersions whose induced conformal classes are not contained in any compact subset of the moduli space, and
	\begin{equation}
		\lim_{k\to \infty}\mc W(\vec \Phi_k) = 8\pi. \label{energy threshold}
	\end{equation}
	Then, after passing to a subsequence, there exist Möbius transformations $\Xi_k$ of $\R^3$, $(s^i_k, y^i_k) \in (0,\infty) \times \R^3$ for $i\in \{1,\ldots, p+3\}$, as well as smooth maps $\varphi^i_k\colon\C\to\Sigma$ if $i\in\{1,2\}$ and $\varphi^i_k:\C\setminus\{0\} \to \Sigma$ if $i\in \{3,\ldots, p+3\}$ such that for each $k\in\N$,  $\im \varphi^1_k,\ldots,\im \varphi^i_k$ are pairwise disjoint and
	\begin{equation}
		(s^i_k)^{-1} (\Xi_k \circ \vec \Phi_k \circ \varphi^i_k - y^i_k) \to \vec \Psi^i \quad 
		\begin{array}{ll}
			\text{in $W^{2,2}_{\loc}(\C;\R^3)$} &\text{if $i \in \{1,2\}$},\\
			\text{in $W^{2,2}_{\loc}(\C \setminus \{0\};\R^3)$} &\text{if $i \in \{3,\ldots, p+3\}$}.
		\end{array}\label{convergence in main theorem simplified}
	\end{equation}
	Hereby, $\vec \Psi^i$ immerses a round sphere for $i\in \{1,2\}$ and a catenoid for $i \in \{3,\ldots, p+3\}$. Additionally, we have $(s^1_k, y^1_k) = (s^2_k,y^2_k)=(1,0)$ and $\limk s^i_k = 0$ for $i\in \{3,\ldots, p+3\}$. The varifolds induced by $\Xi_k\circ\vec\Phi_k$ converge to two intersecting round spheres.
\end{thm}

Notice that with \eqref{energy threshold}, we only impose a limiting energy condition rather than the more frequently used strict condition $\mc W(\vec\Phi_k)<8\pi$ that by \eqref{eq:intro:LY} guarantees embeddedness. In particular, we allow for sequences with $\mc W(\vec\Phi_k)>8\pi$. Rotating the family of $\lambda$-figure-eights in \cite[Section~6]{muellerspener2020} yields an example of a sequence of non-embedded axi-symmetric tori whose conformal classes degenerate and whose Willmore energies approach $8\pi$ from above. An explicit example for $p=1$ with Willmore energy strictly below $8\pi$ is given by the family of 2-lobed Delaunay tori, cf. \cite{HellerPedit}. Finally, we point out that the specific phenomenon of diverging conformal classes is known to be the driver behind the formation of singularities for the Willmore flow of axi-symmetric tori, see Theorem~1.2 and Proposition~4.2 of \cite{DallAcquaMuellerSchaetzleSpener2024}.

\subsection{Full bubble classification}
In this section, we will state our main result \zcref{thm:Main theorem simplified} in full detail. To keep things manageable, rigorous definitions will be given later in \zcref{sec:Compactness theorems in diverging conformal classes}. For a simplified version, we refer to \zcref{thm:Main theorem simplified}, which already captures the general idea. 

The bubbling analysis is based on decomposing the surface $\Sigma$ into \emph{thick} and \emph{thin parts} resulting from the degenerating Poincar\' e metrics via Deligne--Mumford compactification, see \zcref{subsec:Mumford-Deligne compactness}. 
The set of all thick parts is denoted with $V_{\mathrm{thick}}$. The thin parts are topological cylinders. These will be subdivided into slabs of energy accumulation and connecting neck regions. The set of all slabs with energy accumulation is $V_{\mathrm{thin}}$. Together, they make the vertices of a graph $V'=V_{\mathrm{thick}}\cup V_{\mathrm{thin}}$. Restricted to each vertex $v\in V'$, Green function estimates (\eqref{bounded conformal factors in nodal region 1} and \eqref{L2 weak estimate on neck regions}) from \cite{LaurainRiviereGreenFunction} give suitable control of the conformal factors in a carefully chosen atlas outside of a set  $R^v$ of finitely many ends and branch points. After suitable scaling by a factor $(s_k^v)^{-1}$, weak $W^{2,2}$-compactness thereby yields a branched limit $\vec \Psi^v$ of $\vec\Phi_k|_v$, see \zcref{thm:thick part convergence} and \zcref{prop:Thin part bubble neck decomposition}. 
Around each end/branch point $r\in R^v$, we apply the bubble-neck decomposition of \cite{BernardRiviere}, resulting in a set $V_{\mathrm{conc}}^{v,r}$ of points in $v$, near which no further energy concentrates. Naturally, e.g.\ by Hèlein's moving frame method (cf. \cite[Lemma 5.1.4]{Helein}, \cite[Theorem 5.5]{RiviereLectureNotes}), the conformal factors near any member of $V_{\mathrm{conc}}^{v,r}$ are suitably controlled, leading once again to weak $W^{2,2}$-compactness, see \eqref{concentration point convergence}. Lastly, we define the set of concentration points
\begin{equation}
	V_{\mathrm{conc}}\cqq \bigcup_{v\in V'}\bigcup_{r\in R^v}V_{\mathrm{conc}}^{v,r}.
\end{equation}

Our main result states that after composing $\vec \Phi_k$ with suitable Möbius transformations $\Xi_k$, the resulting graph 
\begin{equation}
	V=V'\cup V_{\mathrm{conc}}= V_{\mathrm{thick}}\cup V_{\mathrm{thin}}\cup V_{\mathrm{conc}}
\end{equation}
can be given the structure of a double tree, see \zcref{Caption 2} for an illustration, where each depicted edge points in the direction of the bubble which is immersed on a smaller scale in the sense of \eqref{scale equation in main theorem}.

\begin{thm}
	\label{thm:Asymptotic convergence}
	Suppose $p$ is a positive integer, $\Sigma$ is a closed, connected, oriented surface of genus $p$, $\vec \Phi_k\colon\Sigma\to \R^3$ is a sequence of weak immersions whose induced conformal classes are not contained in any compact subset of the moduli space, and
	\begin{equation}
		\lim_{k\to \infty}\mc W(\vec \Phi_k) = 8\pi. \label{energy threshold2}
	\end{equation}
	Then, after passing to a subsequence, there exist Möbius transformations $\Xi_k$ of $\R^3$ such that the oriented graph $G'=(V',E')$ associated to $\Xi_k\circ\vec\Phi_k$ as defined in Definitions \ref{def:Graph structure} and \ref{def:Graph structure for p equals 1} is given by $(T+T)/\sim$, where $T=(V_T,E_T)$ is a rooted tree with edges pointing away from the root and $p+1$ leaves given by the set $V_{\mathrm{thin}}$,
	\begin{equation}
		T+T=\{(v,i),\, v\in V_T,\,i\in\{1,2\}\}
	\end{equation}
	is the disjoint union of two copies of $T$, and $(c,1)\sim(c,2)$ identifies the leaves $c\in V_{\mathrm{thin}}\subset V_T$. There are exactly two concentration points $V_{\mathrm{conc}}=\{S_1,S_2\}$ and, denoting the root of $T$ with $\omega$, the graph $G=(V,E)$ is obtained from $G'$ by adding the two edges $(S_1,(w,1))$ and $(S_2,(\omega,2))$. For each vertex $v\in V$, there exist an immersion $\vec\Psi^v\in \mc F_{\mathbb S^2}$ of the Riemann sphere $\mathbb S^2=\hat{\mathbb C}$, finitely many points $Q^v \cup R^v\subset \hat \C$, sequences of scales $s_k^v>0$, points $y_k^v\in\R^3$, as well as charts $\varphi_k^v\in C^\infty(\mathbb C\setminus (Q^v\ \cup R^v),\Sigma)$ such that for each $k\in\N$, set of images $\{\operatorname{im}\varphi_k^v,\,v\in V\}$ is disjoint and   
	\begin{equation}
		(s_k^v)^{-1} (\Xi_k \circ \vec \Phi_k \circ \varphi^v_k - y^v_k) \to \vec \Psi^v
			\quad\text{in $W^{2,2}_{\loc}(\C \setminus (Q^v\ \cup R^v);\R^3)$}.\label{convergence in main theorem}
	\end{equation}
	Moreover, $\vec\Psi^v$ immerses a flat plane for $v\in V_{\mathrm{thick}}$, a catenoid for $v\in V_{\mathrm{thin}}$, and a round sphere for $v\in V_{\mathrm{conc}}$. The scales satisfy
	\begin{align}
		(s^{(v,1)}_k, y^{(v,1)}_k) &= (s^{(v,2)}_k, y^{(v,2)}_k) \qquad \quad \text{for $v\in V_T$},  \label{same scale and position for two bubbles}\\
		(s^{S_1}_k, y^{S_1}_k) &= (s^{S_2}_k, y^{S_2}_k)=(1,0)\label{same scale and position for two spheres}
	\end{align}
	and, if $e = (v,w)  \in E$ is an edge from $v$ to $w$, then
	\begin{equation}
		\limk \frac{s^{v}_k}{s^{w}_k} = \infty.\label{scale equation in main theorem}
	\end{equation}
	Finally, the varifolds induced by $\Xi_k\circ\vec\Phi_k$ converge to two intersecting round spheres.
\end{thm}

\begin{remark}
	If $\vec \Phi_k$ is a sequence of possibly constrained Willmore surfaces with uniformly bounded Lagrange multipliers, that is, $\vec\Phi_k\in \mc S_p^3$ solves
	\begin{equation}
		\Delta_{g_k}H_k+2H_k(H_k^2-K_k)-2\alpha_k H_k - \beta_k-\gamma_kK_k=0
	\end{equation}
	for some sequences of real numbers $\alpha_k,\beta_k,\gamma_k$ with
	\begin{equation}
		\limsup_{k\to\infty}(|\alpha_k|+|\beta_k|+|\gamma_k|)<\infty,
	\end{equation}
	then, by the $\varepsilon$-regularity \cite{BernardWheelerWheeler}, see also \cite[Theorem 1.1]{ScharrerWestEnergyQuantization}, the sequence \eqref{convergence in main theorem} converges in $C^l_{\loc}(\C\setminus (Q^v \cup R^v);\R^3)$ for each $l\in\N$.
\end{remark}

Note that the ambient space is the Euclidean three space. This particular codimension one setting was used twice in the proof of \zcref{thm:Asymptotic convergence}. Firstly, with \eqref{eq:intro:MND}, we applied the resolution of the Willmore conjecture \cite{MarquesNeves} which is unknown in higher codimensions. Secondly, we used the classification result \cite{SchoenUniqueness} that complete, embedded, minimal surfaces in $\R^3$ with two ends are catenoids or union of planes. This result fails in higher codimensions, considering the fact that holomorphic curves in $\C^2$ such as
\[\{(z,1/z), \, z\in \C \setminus \{0\}\}\subset \C^2 \cong \R^4\]
are minimal surfaces in $\R^4$.

The structure of the graph $G$ in \zcref{thm:Asymptotic convergence} results from the particular choice of Möbius transformations $\Xi_k$. The full classification of possible graphs independent of the chosen Möbius transformation is given in \zcref{sec:Inversions}.

Finally, we point out the following interesting open problem. Suppose $\tilde{\Sigma}$ is a nodal surface as described in \zcref{subsec:Mumford-Deligne compactness}. Define
\begin{equation}
	\beta(\tilde{\Sigma}) \cqq \inf \liminf _{k\to \infty} \mc W(\vec \Phi_k), \label{new minimization problem for fixed nodal surface}
\end{equation}
where the infimum is taken over all sequences $\vec \Phi_k$ in $\mc E_{\Sigma}$ that have $\tilde{\Sigma}$ as their nodal surface. Our main result implies that if $\tilde{\Sigma}$ does not have the specific double tree structure as described in \zcref{thm:Asymptotic convergence}, then $\beta(\tilde{\Sigma}) > 8\pi$. Is it possible to determine the value $\beta(\tilde\Sigma)$ and to do a similar bubbling analysis?

\begin{figure}
	\begin{minipage}[c]{.48\textwidth}  
		\centering  
		\input{catenoids_scale.tex}  
	\end{minipage}%
	\hfill 
	\begin{minipage}[c]{0.48\textwidth}  
		\centering  
		\input{tree_structure.tex}
	\end{minipage}
	\par
	\begin{minipage}[t]{.48\textwidth}
		\vspace{-2.7em}
		\caption{After composing with Möbius transformations, the immersions $\vec \Phi_k$ resemble two spheres connected by $p+1$ catenoids. The positions and scales of the catenoids are determined by the structure of the nodal surface.} 
		\label{Caption 1}  
	\end{minipage}%
	\hfill
	\begin{minipage}[t]{.48 \textwidth}  
		\caption{The bubble graph structure $G=(V,E)$ from \zcref{def:Graph structure} in the setting of \zcref{thm:Asymptotic convergence}. Depicted are the bubbles $\vec \Psi^v$ for $v\in V$ together with the edges $E$. } 
		\label{Caption 2}  
	\end{minipage}  
\end{figure}

\subsection{Applications to sequences of constrained Willmore minimizers} \label{subsec:Applications to minimizers}
For $\vec\Phi\in\mathcal E_\Sigma$, we define the \emph{area $\mc A$, algebraic volume $\mc V$, isoperimetric ratio $\mathcal I$, total mean curvature $\mc M$, and normalized total mean curvature $\mc T$} by
\begin{gather}
	\mc A(\vec \Phi) \cqq \int _{\Sigma} 1 \dif{\mu}, \quad \mc V(\vec \Phi) \cqq -\frac{1}{3} \int _{\Sigma} \langle \vec n, \vec \Phi\rangle \dif{ \mu},\quad \mc I(\vec \Phi) \cqq \frac{\mc A(\vec \Phi)}{\mc V(\vec \Phi)^{2/3}}, \\
	\mc M(\vec \Phi) \cqq \int _{\Sigma} H \dif{ \mu}, \quad \mc T(\vec\Phi)\cqq \frac{\mc M(\vec\Phi)}{\sqrt{\mc A(\vec\Phi)}}.\label{are, volume and total mean curvature}
\end{gather}
Given a sequence $\vec \Phi_k$ in $\mc E_\Sigma$ with uniformly bounded Willmore energy, we define the set of macroscopic bubbles $V_{\textrm{macro}}$ to consist of all $v\in V$, where the scale $s_k^v$ defined in \zcref{def:Graph structure} satisfies $\lim_{k\to\infty}s_k^v>0$. As a consequence of our bubble analysis, we recover the following continuity properties from \cite[Theorem~1]{ChenLi}.
\begin{thm}\label{thm:bubble convergence for extrinsic quantities}
	Suppose $\mc F \in \{\mc A, \mc V, \mc M\}$, $p$ is a positive integer, $\Sigma$ is a closed, connected, oriented surface of genus $p$, and $\vec\Phi_k\in\mc E_{\Sigma}$ satisfies
	\begin{equation}
		\sup_{k\in \N} \bigg [\mc W(\vec \Phi_k) + \mc A(\vec \Phi_k) \bigg ]< \infty.\label{bounded Willmore aaaaand area}
	\end{equation}
	Then, after passing to a subsequence,
	\begin{equation}
		\limk \mc F(\vec \Phi_k) = \sum_{v\in V_{\mathrm{macro}}} \mc F(\vec \Psi^v). \label{bubble convergence of area, volume, and total mean curvature}
	\end{equation}
\end{thm}
In a series of works \cite{Schygulla,KMR,MondinoScharrerInequality,ScharrerDelaunay, KusnerMcGrath}, it was shown that for all $\sigma \in (\sqrt[3]{36\pi}, \infty)$ there exists a smooth embedded minimizer $\vec \Phi^{\mc I}_{p, \sigma}$ of $\mc W$ among all $\vec\Phi\in \mc S_p^3$ with $\mc I(\vec\Phi)=\sigma$. By the invariance of $\mc W$ and $\mc I$ under dilations, we may assume $\mc A(\vec \Phi^{\mc I}_{p, \sigma})=1$.

\begin{cor}\label{cor:Bubbles in space iso problem}
	Suppose $p$ is a positive integer, $\sigma_k$ is a sequence of real numbers, and $\lim_{k\to\infty}\sigma_k=\infty$. Then, $\vec \Phi_{p,\sigma_k}^{\mathcal I}$ satisfies the hypotheses of \zcref{thm:Asymptotic convergence} and for the induced graph $G=(V,E)$ of \zcref{thm:Asymptotic convergence} there holds $V_{\mathrm{macro}} = \{S_1, S_2\}$, where $\vec \Psi^{S_1}, \vec \Psi^{S_2}$ immerse round spheres of opposite orientations whose images coincide.
\end{cor}
Kuwert--Li \cite{KuwertAsymptoticsSmallIso} proved a similar statement for the genus zero case.

In \cite{MasterThesis}, the first and third author showed that for all $\tau\in(0,\sqrt{8\pi})\setminus\{\sqrt{4\pi}\}$ there exists a smooth embedded minimizer $\vec \Phi^{\mc T}_{p, \tau}$ of $\mc W$ among all $\vec\Phi\in\mc E_\Sigma$ with $\mc T(\vec\Phi)=\tau$. By the invariance of $\mc W$ and $\mc T$ under dilations, we may assume $\mc A(\vec \Phi^{\mc T}_{p, \tau})=1$. 

\begin{cor}\label{cor:bubbles in space T problem}
	Suppose $p$ is a positive integer, $\tau_k$ is a sequence in $(0,\sqrt{8\pi})\setminus \{\sqrt{4\pi}\}$, and $\lim_{k\to\infty}\tau_k=\sqrt{8\pi}$. Then, $\vec \Phi_{p,\tau_k}^{\mathcal T}$ satisfies the hypotheses of \zcref{thm:Asymptotic convergence} and for the induced graph $G=(V,E)$ of \zcref{thm:Asymptotic convergence} there holds $V_{\mathrm{macro}} = \{S_1, S_2\}$, where $\vec \Psi^{S_1}, \vec \Psi^{S_2}$ immerse round spheres of equal radii and orientations whose images intersect in a single point. 
\end{cor}
A similar result can also be proved in the genus $0$ case by the same methods. As the spherical case does not fit well into the setting of this article, we will omit it here.

\subsection{Structure of the article}
\begin{enumerate}[label = $\bullet$]
	\item In \zcref{sec:Compactness theorems in diverging conformal classes}, we recall compactness theorems as well as $L^{2,\infty}$-bounds for the conformal factors of sequences of immersions with degenerating conformal classes. 
	Subsequently, we define the underlying graph structure $G=(V,E)$ of bubbles forming along the sequence.
	\item In \zcref{sec:Scales and positions}, we show that branch points of order $m$ are attached to ends of order $-m$. This leads to one of our central techniques: the notions of bubble descent and bubble ascent.
	
	\item In \zcref{sec:Proof main theorem}, we prove \zcref{thm:Asymptotic convergence}. A key observation is that due to \zcref{thm:equality-in-Li-Yau}, no true branch points exist, see \zcref{prop: no density 2}.
	\item In \zcref{sec:Inversions}, we classify the way in which Möbius transformations affect the bubble graph from \zcref{thm:Asymptotic convergence}.
	\item In \zcref{sec:Applications to minimizers of several problems}, we study several examples of constrained Willmore minimizers.
\end{enumerate}

\section{Compactness theorems in diverging conformal classes} \label{sec:Compactness theorems in diverging conformal classes}

We denote by $\mb D$ the two-dimensional, open unit disk. $B_r(x)$ denotes the open ball in $\R^n$ around $x\in \R^n$ with radius $r$ and if $x=0$, we simply write $B_r$.

Suppose $(X,\mu)$ is some measure space. We define the Lorentz space $L^{p,\infty}(X)$ for $p\in (1,\infty)$ as the space of all measurable functions $f$ such that the quasi-norm
\begin{equation}
    |f|_{L^{p,\infty}(X)} \cqq \left (\sup_{t>0}  t^p \mu(\{x,\,|f(x)|>t\})\right )^{\frac{1}{p}}\label{Lorentz space}
\end{equation}
is finite. $|\cdot|_{L^{p,\infty}(X)}$ is equivalent to a norm $\|\cdot \|_{L^{p,\infty}(X)}$ on $L^{p,\infty}(X)$ with respect to which $L^{p,\infty}(X)$ becomes a Banach space, see \cite[Exercise 1.4.3, 1.1.12]{Grafakos}.
\subsection{Almost logarithmic behavior of the conformal factor in neck regions}\label{subsec:Almost log bounds}
The following lemma describes the behavior of the conformal factor in neck regions\footnote{In \cite{LaurainRiviereEnergyQuantization}, neck regions are degenerating annuli with no energy concentration on dyadic annuli such that the immersion is extended \emph{throughout the interior disk of the annulus} and their nonextendable counterpart is called collar region. This distinction is important for the study of residues. However, since residues do not play a role in this work, we will simply refer to both as neck regions.}. Equation \eqref{weakened version of pointwise gradient estimate} is a weaker statement than the conclusion in \cite[Lemma V.3]{BernardRiviere}. However, we do not assume smallness of $\|\nabla \vec n\|_{L^{2,\infty}}$ in the neck region. This assumption would be fulfilled whenever some type of $\eps$-regularity, see \cite{RiviereAnalysisAspects,BernardWheelerWheeler}, is satisfied, for example, when working with constrained or unconstrained Willmore surfaces. In this case, the $\eps$-regularity is first applied locally in the neck region to get uniform bounds of the form $|\nabla \vec n(x)| \leq \eps/|x|$, which then implies the $L^{2,\infty}$-estimate. In our case however, we work with general weak immersions and cannot apply any $\eps$-regularity.
 
 \begin{lemma}[Almost logarithmic behavior of the conformal factor in neck regions] \label{lem:almost pointwise control conformal factor}
  Let $\Lambda>0$ and $\delta>0$. There exist $\eps=\eps(\Lambda,\delta)>0$, $\alpha_0 = \alpha_0(\Lambda, \delta) \in (0,1)$, and $Q = Q(\Lambda,\delta)>1$ depending only on $\Lambda$ and $\delta$ with the following property. Suppose $R>r>0$, $R/r > Q$, and $\vec \Phi\in \mc E_{B_R \setminus B_r}$ is a weak, conformal immersion with conformal factor $\lambda$ such that
 \[  \|\nabla \lambda\|_{L^{2,\infty}(B_{R}\setminus B_{r})}\leq \Lambda\]
  and
 \[\int _{B_{2s}\setminus B_{s}} |\nabla \vec n|^2 \dif x < \eps\quad \text{for all $s \in (r, R/2)$}.\]
 Then, there exist $m\in \Z\setminus \left \{0\right \}$, $ A \in \R$, and a constant $C(\Lambda,\delta)$ depending only on $\Lambda$ and $\delta$ such that
\begin{equation}
|\lambda(x) - \lambda(y)- (m-1) \log(|x|/|y|) | \leq   \delta |\log(|x|/|y|)| + C(\Lambda, \delta)\label{weakened version of pointwise gradient estimate}
\end{equation}
for all $x,y \in B_{\alpha_0 R }\setminus B_{\alpha_0^{-1}r}$. Additionally, $|m|\leq C(\Lambda, \delta)$.
 \end{lemma}
 The proof is moved to Appendix \hyperref[sec:Proof of Lemma]{A}.

 

\subsection{Mumford--Deligne compactness for higher genus} \label{subsec:Mumford-Deligne compactness}
Let us first consider the higher genus case, i.e., $p\geq 2$. The case $p=1$ will be a corollary of the more general analysis for $p\geq 2$. A reference for the statements of this section is \cite{Gromov, Imayoshi}, particularly \cite[Proposition 5.1]{Gromov}. We also follow the notation from \cite{LaurainRiviereGreenFunction}. Suppose that $(\Sigma, c_k)$ is a Riemann surface of genus $p\geq 2$. From the uniformization theorem, we know that $\Sigma$ admits a hyperbolic structure $h_k$, i.e., a (unique) metric $h_k$ satisfying $K_{h_k}\equiv -1$. 

By \cite[Lemma 4.1]{Gromov}, there are at most $3p-3$ simple closed geodesics $\gamma^i_{k}$ in $(\Sigma, h_k)$ of length $\ell(\gamma^i_k, h_k)$ less than $2\arsinh(1)$. Each of these geodesics is contained in precisely one connected component of ${\left \{x\in (\Sigma, h_k), \, \injrad(\Sigma, h_k, x) < \arsinh(1)\right \}}$. Such a connected component is called a \emph{thin part}. Here, $\injrad(\Sigma, h_k, x)$ is the injectivity radius of $(\Sigma, h_k)$ around $x$\footnote{For a Riemannian manifold $(\Sigma,h)$, the injectivity radius around $x\in \Sigma$ is defined to be 
 \[\injrad(\Sigma,h,x) \cqq \sup \left \{r>0,\, \exp_x \vert _{B(0,r)}\text{ is a diffeomorphism}\right \}, \] where $\exp$ denotes the exponential map.}. More precisely, any such component $U$ contains exactly one geodesic $\gamma^i_k$ of length $\ell(\gamma^i_k, h_k)<2\arsinh(1)$ and $U$ is isometric to
 \[\left \{z\in \mb H,\,  d_{\mb H}(z, e^{\ell(\gamma^i_k, h_k)} z) < 2\arsinh(1)\right \}/\langle z\mapsto e^{\ell(\gamma^i_k, h_k)}z\rangle.\]
Here, $\mb H$ is the upper half plane and $d_{\mb H}$ is the hyperbolic distance in $\mb H$.

A sequence $(\Sigma, c_k)$ diverges to the boundary of the moduli space if and only if there is at least one simple closed geodesic $\gamma^i_k$ such that
\[\ell(\gamma^i_k, h_k) \to 0\quad\text{as $k\to \infty$}.\]
After taking subsequences, let $\mc N$ be the number of simple closed geodesics whose lengths go to 0 as $k\to \infty$. We define $\tilde{\Sigma}$ to be the topological surface obtained by removing the geodesics $\gamma^i_k$ for $i\in \left \{1,\ldots, \mc N\right \}$ from $\Sigma$ and gluing two points $q^{i}_1$, $q^{i}_2$ to each of the newly obtained boundaries of $\Sigma \setminus \bigcup _{i=1}^{\mc N}\gamma^i_k$. We denote by $Q\cqq \left \{q^1_1,\ldots, q^{\mc N}_1, q^1_2,\ldots , q^{\mc N}_2\right \}$ the added points and by $\sigma^j\subset \tilde{\Sigma}$, $j\in \left \{1,\ldots, \mc M\right \}$ the connected components of $\tilde{\Sigma}$ and call the $\sigma^j$ the \emph{thick parts}. We denote by
\begin{equation}
V_{\mathrm{thick}} \cqq \{\sigma^j, \, j\in \{1,\ldots, \mc M\}\}\label{thick parts sigma j}
\end{equation}
the set of thick parts. As a consequence of Euler's formula, it holds
\begin{equation}
    p = \mc N+1 - \mc M + \sum_{j=1}^{\mc M} \genus(\sigma^j) \label{sum of genera},
\end{equation}
so in particular $\mc M \leq \mc N+1$. Furthermore, it follows from the Gauss--Bonnet theorem that 
\begin{equation}
  2\genus(\sigma^j) +   \#(Q \cap \sigma^j) = 2 - \chi(\sigma^j \setminus Q) \geq 3. \label{number of Q in v}
\end{equation}
 Furthermore, there are diffeomorphisms $\psi_k\colon\tilde{\Sigma} \setminus Q \to \Sigma \setminus \bigcup _{i=1}^{\mc N} \gamma^i_k$ such that $\tilde{h}_k \cqq \psi_k^* h_k$ converges in $C^\infty_{\loc}(\tilde{\Sigma} \setminus Q)$ to some complete, hyperbolic metric $\tilde{h}$. $\psi_k^{-1}$ extend to continuous maps from $(\Sigma, h_k)$ to $\tilde{\Sigma}/\sim$, where $q^i_1\sim q^i_2$ for $i\in\left \{1,\ldots, \mc N\right \}$ are identified, such that $\psi_k(\gamma^i_k) = \left \{q^i_1,q^i_2\right \}$. Around each point from $Q$, there is a punctured neighborhood in $(\tilde \Sigma, \tilde{h})$ isometric to a standard cusp $\left \{z\in \mb H, \, \Im z > \frac{1}{2}\right \}/\langle z\mapsto z+1\rangle $. In particular, the induced complex structure of $\tilde{h}$ extends uniquely to $\tilde{\Sigma}$. We denote the resulting compact Riemann surface by $(\bar{\Sigma}, \bar{c})$. We also equip $\bar{\Sigma}$ with a metric $\bar{h}$ of constant curvature on each connected component (so in particular, $\bar{h}$ need not be hyperbolic). $\tilde{\Sigma}$ is called the \emph{nodal surface} of the sequence $(\Sigma, c_k)$. 

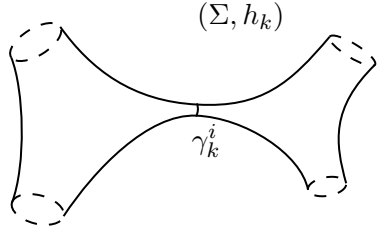
\begin{figure}[H] 
\centering 
\tikzset{every picture/.style={line width=0.75pt}} 

\begin{tikzpicture}[x=0.75pt,y=0.75pt,yscale=-1,xscale=1]

\draw    (118.76,137.26) .. controls (133.85,158.54) and (154.72,174.83) .. (187.32,176.36) ;
\draw    (94.25,152.56) .. controls (100.75,168.72) and (101.76,206.91) .. (94.64,225.76) ;
\draw  [dash pattern={on 4.5pt off 4.5pt}] (94.25,152.56) .. controls (92.27,149.39) and (96.15,143.39) .. (102.92,139.17) .. controls (109.68,134.94) and (116.78,134.09) .. (118.76,137.26) .. controls (120.74,140.44) and (116.86,146.43) .. (110.09,150.66) .. controls (103.32,154.88) and (96.23,155.73) .. (94.25,152.56) -- cycle ;
\draw  [dash pattern={on 4.5pt off 4.5pt}] (94.64,225.76) .. controls (95.64,221.92) and (102.23,220.31) .. (109.35,222.18) .. controls (116.48,224.04) and (121.44,228.66) .. (120.44,232.5) .. controls (119.43,236.34) and (112.84,237.94) .. (105.72,236.08) .. controls (98.59,234.22) and (93.63,229.6) .. (94.64,225.76) -- cycle ;
\draw    (120.44,232.5) .. controls (135.88,208.95) and (162.71,180.92) .. (186.64,181.94) ;
\draw    (187.32,176.36) .. controls (219.3,178.19) and (233.69,165.66) .. (253.25,142.24) ;
\draw    (186.64,181.94) .. controls (211.08,182.96) and (238.99,201.33) .. (242.45,219.23) ;
\draw  [dash pattern={on 4.5pt off 4.5pt}] (242.45,219.23) .. controls (242.02,216.68) and (245.92,213.88) .. (251.17,212.98) .. controls (256.43,212.08) and (261.04,213.42) .. (261.48,215.97) .. controls (261.91,218.53) and (258.01,221.32) .. (252.76,222.22) .. controls (247.5,223.12) and (242.89,221.78) .. (242.45,219.23) -- cycle ;
\draw  [dash pattern={on 4.5pt off 4.5pt}] (253.25,142.24) .. controls (254.27,140.6) and (260.16,142.42) .. (266.42,146.3) .. controls (272.68,150.19) and (276.92,154.67) .. (275.9,156.31) .. controls (274.88,157.95) and (268.99,156.13) .. (262.73,152.24) .. controls (256.48,148.35) and (252.23,143.88) .. (253.25,142.24) -- cycle ;
\draw    (261.48,215.97) .. controls (255.84,195.4) and (253.8,178.09) .. (275.9,156.31) ;
\draw [color={rgb, 255:red, 0; green, 0; blue, 0 }  ,draw opacity=1 ]   (186.59,182.4) .. controls (187.51,180.31) and (187.84,178.31) .. (187.01,175.81) ;

\draw (185.6,123.64) node [anchor=north west][inner sep=0.75pt]    {$( \Sigma , h_{k})$};
\draw (182.88,184.84) node [anchor=north west][inner sep=0.75pt]    {$\gamma^i_k $};

\end{tikzpicture}
\caption{In a neighborhood of a short simple closed geodesic $\gamma^i_k$, $(\Sigma, h_k)$ looks like a long thin cylinder.}
\end{figure}
\subsection{Analysis on the thick parts}\label{subsec:Analysis thick parts}
Laurain and Rivi\`ere \cite{LaurainRiviereGreenFunction} estimated the Green's function on surfaces with diverging conformal class and proved weak $L^2$-estimates for the gradient on a carefully chosen, finite atlas. They used this to prove a compactness result of immersions with $L^2$-bounded second fundamental form in this setting. We recall this statement\footnote{We slightly change the statement, namely we do not apply Möbius transformations to bound the image of the surface. We thus obtain surfaces possibly containing ends. The proof remains unchanged.}.
\begin{thm}[See {\cite[Theorem 0.3]{LaurainRiviereGreenFunction}}] \label{thm:thick part convergence}
Let $\Lambda > 0$. Suppose that $\Sigma$ is a closed surface of genus $p\geq 2$ and let $\vec \Phi_k\in \mc E_\Sigma$ be a sequence of weak, conformal immersions (conformal with respect to constant curvature metrics $h_k$ in the same conformal class as $\vec \Phi_k^* g_{\R^3}$) into $\R^3$ with
\begin{equation}
    \sup_{k\in \N} \mc E (\vec \Phi_k) < \Lambda. \label{bounded energies for phi k}
\end{equation}
Then, after passing to a subsequence, for any connected component $\sigma^j $ of the nodal surface $\tilde{\Sigma}$ associated to the sequence $(\Sigma, h_k)$, there exist a sequence $s^j_k \in \R_{>0}$, a finite set $R^j  \subset \sigma^j\setminus Q$, an arbitrary point $z^j \in \sigma^j \setminus (Q \cup R^j)$ and $y^j_k \cqq \vec \Phi_k \circ \psi_k(z^j)$ such that 
\begin{equation}
\vec \Psi^j_k \cqq (s^j_k)^{-1}  (\vec \Phi_k \circ \psi_k - y^j_k) \wto \vec \Psi^j \quad \text{in }W^{2,2}_{\loc}(\sigma^j \setminus (R^j\cup Q), \tilde{h};\R^3).\label{limit immersion on nodal part}
\end{equation}
Here, $Q$ is the set of punctures from Section \ref{subsec:Mumford-Deligne compactness}, $\vec \Psi^j$ is a conformal, branched weak immersion on $(\sigma^j,\overline{h})$ with branch points/ends around $Q\cup R^j$ and $\psi_k$ is the diffeomorphism satisfying $\psi_k^*h_k=\tilde{h}_k \to \tilde{h}$ in $C_{\loc}^\infty(\tilde{\Sigma}\setminus Q)$. 
\end{thm}
Furthermore, by \cite[(25)]{LaurainRiviereGreenFunction}, the conformal factor $u_k$ given by
\[(\vec \Phi_k \circ \psi_k) ^* g_{\R^3} = e^{2u_k} \tilde{h}_k\]
satisfies
\begin{equation}
\sup _{k} \|d u_k\|_{L_{\tilde{h}_k}^{2,\infty}(K)}< C_K \quad\text{for all $K\Subset \sigma^j \setminus Q$}\label{bounded conformal factors in nodal region 1}
\end{equation}
and
\begin{equation}
\sup _{k} \|u_k - \log s^j_k\|_{L^{\infty}(K)}<  C_K \quad\text{for all $K\Subset \sigma^j \setminus (R^j \cup Q)$}.\label{bounded conformal factors in nodal region 2}
\end{equation}
Here, $C_K<\infty$ depends on $K$ and the sequence of immersions $(\vec \Phi_k)_k$\footnote{$C_K$ needs to depend $(\vec \Phi_k)_k$ as can be seen by choosing a sequence of immersions where curvature concentrates on a tiny, but positive scale, which is not detected by the points $R^j$.}.
\begin{remark}\label{rem:Behavior around points R v}
Choosing local, conformal coordinates around the points in $(Q \cap \sigma^j) \cup R^j$, we deduce from \zcref{lem: Muller Sverak variant} and \eqref{L2 infinity estimate for fixed immersion} that $\vec \Psi^j$ has either branch points or ends there. Notice however that the metrics $\tilde{h}$ and $\overline{h}$ are not equivalent around the points in $Q$. We will study the behavior around these points in more detail in \zcref{rem:Behavior of Psi v around Q v}.

\end{remark}
\subsection{Analysis on the thin parts}
\label{subsec:ThinPart}
From Section \ref{subsec:Mumford-Deligne compactness}, we know that each thin part in $(\Sigma, h_k)$ contains precisely one short geodesic $\gamma^i_k$ for $i\in \left \{1,\ldots, \mc N\right \}$ of length $l_k \cqq \ell(\gamma^i_k, h_k)$ such that $l_k\to 0$ as $k\to \infty$. The thin part is isometric to
\[\left \{z \in \mb H, \, d_{\mb H}(z, e^{l_k}z) < 2\arsinh(1)\right \}/\langle z\mapsto e^{l_k} z\rangle. \]
We denote $\phi_k \cqq \arcsin(\sinh({l_k}/2))$. Using the explicit description of the distance $d_{\mb H}$ on $\mb H$ from \cite[Lemma 4.7]{Gromov}, this is the set 
\[\left \{z = r e^{i \phi} \in \mb H,\, r \in [1,e^{l_k}], \phi \in \left (\phi_k , \pi - \phi_k\right )\right \}/\langle z\mapsto e^{l_k}z\rangle.\]
Via the map $\frac{2\pi}{{l_k}} \Log( e^{-i\phi_k} z)$, this is isometric to the cylinder
\begin{equation}
P_{k} \cqq \left [0, \frac{2\pi}{{l_k}}\left (\pi - 2\phi_k\right )\right ] \times \mb S^1,\label{long thin cylinder}
\end{equation}
equipped with the metric 
\begin{equation}
g_k \cqq \left (\frac{{l_k}}{2\pi \sin\left (\frac{{l_k}t}{2\pi} + \phi_k\right )}\right )^2 (dt^2 + d\theta^2).\label{metric ds2 on the thin cylinder}
\end{equation}
Based on these isometries, we will always view $(P_k, g_k)$ as subsets of $(\Sigma, h_k)$. So the thin parts are conformally equivalent to long cylinders $ [0, L_k]\times \mb S^1$ for $\frac{2\pi}{l_k}(\pi -2\phi_k) = L_k \to \infty$ as $k\to \infty$. Finally, we denote $A_{{k}} = \mb D\setminus B_{e^{-{L_k}}}$ equipped with the usual euclidean metric and with $\chi_k\colon A_{{k}} \to P_{k}$ the conformal diffeomorphism
\begin{equation}
\chi_k(re^{i\theta}) \cqq\left (\ln(r) + L_k, \theta\right ). \label{conformal diffeo from Ak to Pk}
\end{equation} 
Suppose $\vec \Phi_k \in \mc E_{\Sigma}$ is a sequence of weak, conformal immersions with $ \sup_{k\in \N} \mc E(\vec \Phi_k) < \Lambda$. In \cite[Theorem~0.2]{LaurainRiviereGreenFunction}, it was shown that if $(\vec \Phi_k \circ \chi_k) ^* g_{\R^3} = e^{2 \lambda_k} g_{\R^2}$ on $A_{{k}}$, then it holds
\begin{equation}
 \sup _{k} \|\nabla \lambda_k\|_{L^{2,\infty}(A_{{k}})} < C(\Lambda).\label{L2 weak estimate on neck regions}
\end{equation}
We need to identify the bubbles and necks in this neck region. Following the full bubble-neck decomposition, originally done in \cite[Proposition III.1]{BernardRiviere}, see also \cite[Proposition A.1]{ScharrerWestEnergyQuantization} and  \cite[Lemma~3.1]{LaurainRiviereEnergyQuantization}, the following simplified version holds.
\begin{prop}[\cite{BernardRiviere,LaurainRiviereEnergyQuantization, ScharrerWestEnergyQuantization}]
\label{prop:Thin part bubble neck decomposition}Let $\Lambda>0$ and $0<\eps< \frac{8\pi}{3}$. Suppose $L_k \to \infty$ and that $\vec \Phi_k\colon A_k = \mb D\setminus B_{e^{-L_k}}\to \R^3$ is a sequence of weak, conformal immersions in $\mc E_{A_k}$ satisfying 
\begin{equation}
\sup_{k\in \N} \bigg [\|\nabla \vec n_k\|_{L^2(A_k)}+\|\nabla \lambda_k\|_{L^{2,\infty}(A_k)}\bigg] < \Lambda,\label{assumptions for bubble neck decomposition}
\end{equation}
where $\lambda_k$ denotes the conformal factor of $\vec \Phi_k$. After passing to a subsequence, the following hold: There exists $\alpha_0>0$ such that
\begin{equation}
\lim_{k\to \infty} \sup \left \{ r \in (0, \alpha_0),\, \int _{B_{\alpha_0}\setminus B_r \cup B_{\frac{e^{-L_k}}{r}}\setminus B_{\frac{e^{-L_k}}{\alpha_0}}} |\nabla \vec n_k|^2 \dif x \geq \eps\right \}=0.\label{alpha 0 choice for thin part}
\end{equation}
There exist $N\in \N_0$ and radii $\alpha_0 = a^0_k > b^0_k > a^1_k > b^1_k > \ldots > a^N_k > b^N_k = \alpha_0^{-1} e^{-L_k} > e^{-L_k} = a^{N+1}_k$ such that
\begin{equation}
\lim_{k\to \infty}\frac{a^l_k}{b^l_k} = \infty \quad \text{for $l\in \left \{0, \ldots, N\right \}$},\quad \lim _{k\to \infty} \frac{b^l_k}{a^{l+1}_k} < \infty \quad\text{for $l\in \left \{0, \ldots, N-1\right \}$.}\label{ratios of the ai and bi}
\end{equation}
Furthermore,
\begin{equation}
\int _{B_{2r}\setminus B_{r}} |\nabla \vec n_k|^2 \dif x < \eps \quad\text{for all $l\in \left \{0,\ldots, N\right \}$ and $r \in (b^l_k,  a^l_k/2)$}\label{no concentration in neck regions}
\end{equation}
and
\begin{equation}
\int _{B_{b^l_k}\setminus B_{a^{l+1}_k}} |\nabla \vec n_k|^2 \dif x > \eps\quad\text{for all $l \in \left \{0,\ldots, N-1\right \}$.} \label{at least eps energy in each bubble}
\end{equation}
Furthermore, for $l\in \{0,\ldots, N-1\}$, there are $s^l_k \in \R_{>0}$, a finite set $R^l \subset \C \setminus \left \{0\right \}$, an arbitrary point $z^l \in \C \setminus (\left \{0\right \} \cup R^l)$ and $y^l_k \cqq \vec \Phi_k(a^{l+1}_k z^l)$ such that
\begin{equation}
\vec \Psi^{l}_k  \cqq (s^l_k)^{-1} \left (\vec \Phi_k\left ( a^{l+1}_k \cdot \right ) - y^l_k\right ) \wto \vec \Psi^l \quad\text{in $W^{2,2}_{\loc}(\C \setminus (\left \{0\right \} \cup R^l);\R^3)$}.\label{bubble convergence in neck regions}
\end{equation}
Here, $\vec \Psi^l$ are conformal, branched weak immersions on $\hat \C$ with branch points/ends at $\left \{0, \infty\right \}\cup R^l$. If $\lambda_k$ denotes the conformal factor of $\vec \Phi_k$ in $A_k$, then
\begin{equation}
\sup_{k\in \N} \| - \log s^l_k + \log a^{l+1}_k + \lambda_k (a ^{l+1}_k \cdot ) \|_{L^\infty(K)} < C_K \quad \text{for all $K \Subset \C \setminus (\{0\} \cup R^l)$}. \label{conformal factor estimate thin part}
\end{equation}
Here, $C_K < \infty$ depends on $K$ and the sequence of immersions $(\vec \Phi_k)_k$.
\end{prop}
The regions $B_{a^l_k}\setminus B_{b^l_k}$ are called \emph{neck regions}. For each thin part isometric to $P_k = P^i_k$ and conformally equivalent to $A^i_k$ from \eqref{long thin cylinder} around a degenerating geodesic $\gamma^i_k$ with length $\ell(\gamma^i_k, h_k) \to 0$, we can consider the conformal reparametrization $\vec \Phi_k \circ \chi^i_k$ on $A^i_k$, whose conformal parameter $\lambda^i_k$ satisfies
\begin{equation}
\sup_i \sup_{k\in \N} \|\nabla \lambda^i_k\|_{L^{2,\infty}(A^i_k)} \qqc  \Lambda' < C(\Lambda) . \label{Lambda prime choice}
\end{equation} 
Let $\eps = \eps(\Lambda', 1/3)$ and $\alpha_0 = \alpha_0(\Lambda', 1/3)$ be the constants from \zcref{lem:almost pointwise control conformal factor}. We may choose $\alpha_0$ even smaller such that \eqref{alpha 0 choice for thin part} holds for all thin parts of which there are only finitely many. We apply \zcref{prop:Thin part bubble neck decomposition} with $\eps$ to find $N_i\in \N_0$ bubbles inside the thin part $A^i_k$. Without loss of generality, we may also assume that $\frac{b^l_k}{a^{l+1}_k} < \alpha_0^{-1}$ for all $l \in \{0,\ldots, N_i-1\}$ by \eqref{ratios of the ai and bi}. We let
\[V_{\mathrm{thin}} \cqq \left \{(i,l), \, i\in \left \{1,\ldots, \mc N\right \}, \, l \in \left \{0,\ldots, N_i-1\right \}\right \}\]
be the set of bubbles in the thin parts and define
\begin{equation}
V' \cqq V_{\mathrm{thick}} \cup V_{\mathrm{thin}},\label{thick and thin parts}
\end{equation}
where $V_{\mathrm{thick}}$ was defined in \eqref{thick parts sigma j}.

\begin{remark}[Behavior of $\vec \Psi^j$ around $Q^j$]\label{rem:Behavior of Psi v around Q v}
In \Cref{thm:thick part convergence}, the limiting immersions $\vec \Psi^j$ for $ \sigma^j\in V_{\mathrm{thick}}$ were considered. Using the parametrization from the thin part, we will study the behavior of $\vec \Psi^j$ around $q$.
%
%
%
 Without loss of generality, there is $i\in \left \{1,\ldots, \mc N\right \}$ with $q = q^i_1$\footnote{If needed, we reorder $q^i_1$ and $q^i_2$. This corresponds to using the parametrization $\hat{\chi^i_k}(re^{i\theta}) = \chi^i_k(e^{-L_k}(re^{i\theta})^{-1}) = (-\ln(r), -\theta)$ instead of \eqref{conformal diffeo from Ak to Pk}.}, so we can work with the map $\chi_k = \chi^i_k$ from \eqref{conformal diffeo from Ak to Pk}. We need to compare the two different parametrizations $\psi_k\colon\tilde{\Sigma}\to \Sigma$ from \zcref{subsec:Mumford-Deligne compactness} and $\chi_k$.
Fix $\alpha \in (0,1)$ and let $U_k^{\alpha} \cqq \psi_k^{-1} \circ \chi_k(B_{\alpha}\setminus B_{\alpha/2}) \subset \sigma^j$. There is $K$ depending on $\alpha$ such that
\begin{equation}
\bigcup _{k\geq K} U_k^{\alpha} \Subset \sigma^j\setminus Q. \label{Uk remains compact}
\end{equation}
We can choose $\alpha$ sufficiently small such that $U_k^{\alpha}$ remains away from the points in $R^j$.
It holds
\[((s^j_k)^{-1}\vec \Phi_k \circ \chi_k)^* g_{\R^3} = (\psi_k^{-1} \circ \chi_k)^* (e^{2u_k - 2\log s^j_k} \tilde{h}_k),\]
where $u_k$ was the conformal factor given by $(\vec \Phi_k \circ \psi_k)^* g_{\R^3} = e^{2u_k} \tilde{h}_k$ as in \zcref{thm:thick part convergence}. \eqref{Uk remains compact} and \eqref{bounded conformal factors in nodal region 2} imply that $((s^j_k)^{-1}\vec \Phi_k \circ \chi_k)^* g_{\R^3}$ is uniformly equivalent to the metric
\[(\psi_k^{-1} \circ \chi_k)^* \tilde{h}_k = \chi_k^* h_k\]
on $B_{\alpha}\setminus B_{\alpha/2}$. By the explicit description of $\chi_k$, see \eqref{metric ds2 on the thin cylinder} and \eqref{conformal diffeo from Ak to Pk}, it is clear that $\chi_k^* h_k$ is uniformly equivalent to $g_{\R^2}$ on $B_{\alpha}\setminus B_{\alpha/2}$. Owing to \eqref{L2 weak estimate on neck regions}, we can apply the same compactness argument as in \zcref{prop:Thin part bubble neck decomposition} to see that after passing to a subsequence, it holds
\begin{equation}
(s^j_k)^{-1} (\vec \Phi_k \circ \chi _k - y^j_k) \wto \tilde{\vec \Psi}^j \text{ in $W^{2,2}_{\loc}(B_{\alpha} \setminus \left \{0\right \};\R^3)$,} \label{weak convergence around point in Q}
\end{equation}
where $\tilde{\vec \Psi}^j$ is some weak, conformal immersion on $ B_{\alpha}\setminus \left \{0\right \}$ whose image coincides with the image of $\vec \Psi^j$ in a neighborhood of $q$ in $v=\sigma^j$. Thus, $\vec \Psi^j$ has an end/branch point of order $m$ at $q$, if and only if $\tilde{\vec \Psi}^j$ has an end/branch point of order $m$ at $0$. If $\tilde{\vec \Psi}^j$ has a branch point around $0$, it holds $\vec \Psi^j(q) = \tilde{\vec \Psi}^j(0)$ (which exists since $\tilde{\vec \Psi}^j$ is continuous at $0$).
\end{remark}
\subsection{Analysis around concentration points}\label{subsec:Analysis concentration points}
In the following, we will construct a tree which describes the relation of the bubbles that accumulate around the concentration points $R^j$ and $R^l$ from the previous two compactness theorems. If $v = \sigma^j \in V_{\mathrm{thick}}$, let $r\in R^j\subset \sigma^j$. The smooth convergence of the metrics $\tilde{h}_k \to \tilde h$ allows us to choose conformal, converging coordinates $\omega^r_k\colon \mb D\to (v, \tilde{h}_k)$ with $\omega^r_k(0)=r$. If $v = (i,l)\in V_{\mathrm{thin}}$, we choose $r\in R^l$ and use the parametrization $\chi^i_k$ from \zcref{subsec:ThinPart}. The full bubble-neck decomposition\footnote{While these results are formulated for bounded conformal classes, they are local in nature and all that we need is the $L^{2,\infty}$-bound for the gradient of the conformal factor in a neighborhood around the concentration points, which comes from \eqref{bounded conformal factors in nodal region 1} and \eqref{L2 weak estimate on neck regions}.} as in \cite[Proposition III.1]{BernardRiviere}, \cite[Proposition A.1]{ScharrerWestEnergyQuantization} with the choice of $\eps$ as given in \zcref{lem:almost pointwise control conformal factor} for $\delta = 1/3$ allows us to split the domain around $r$ into bubbles and necks: More precisely, we find $N_{v,r}\in \N$, pairs $(x^m_k, \rho^m_k) \in \mb D \times (0,\infty)$, $m\in \{1,\ldots, N_{v,r}\}$, such that $\limk (x^m_k, \rho^m_k) = (0,0)$, finite sets $R^m \subset \C$, $s^m_k \in \R_{>0}$, arbitrary points $z^m \in \mb D \setminus R^m$, and $y^m_k \cqq \vec \Theta^m_k(z^m)$, where 
\begin{equation}
\vec \Theta^m_k(z) \cqq \begin{cases}
\vec \Phi_k \circ \psi_k \circ \omega^r_k (x^m_k + \rho^m_k z), & v \in V_{\mathrm{thick}},\\
\vec \Phi_k \circ \chi^i_k ( a^{l+1}_k (r +x^m_k +  \rho^m_k z)), & v = (i,l) \in V_{\mathrm{thin}},
\end{cases}\label{Theta k parametrization in neck region}
\end{equation}
 such that
\begin{equation}
\vec \Psi^m_k \cqq (s^m_k)^{-1}(\vec \Theta^m_k-y^m_k) \wto \vec \Psi^m \quad\text{in $W^{2,2}_{\loc}(\C \setminus R^m;\R^3)$.}\label{concentration point convergence}
\end{equation}
Indeed, the bubbles are constructed in such a way that no concentration of energy happens away from the points $R^m$ in the sense of \cite[(A.17)]{ScharrerWestEnergyQuantization}, which implies \eqref{concentration point convergence} via the same compactness procedure as before.

The set $ \{B_{2\rho^m_k}(x^m_k), \, m \in \{1,\ldots, N_{v,r}\}\}$ has a directed graph structure given by the inclusion relation, such that \cite[(A.8)]{ScharrerWestEnergyQuantization} yields the set of children of a vertex. This graph structure is independent of $k$ after taking a subsequence and can be identified with a directed graph $T^{v,r}$. We denote the set of vertices of $T^{v,r}$ by
\begin{equation}
V^{v,r}_{\text{conc}} \cqq \{(v,r, m), \, m \in \{1,\ldots, N_{v,r}\}\}\label{concentration bubbles}
\end{equation}
and each vertex corresponds to the datum from \eqref{concentration point convergence}. \cite[(A.5)]{ScharrerWestEnergyQuantization} shows that $T^{v,r}$ may be viewed as a directed \emph{tree} with edges pointing away from the root (the set of children of a bubble is given by \cite[(A.8)]{ScharrerWestEnergyQuantization}). After relabelling, the root of $T^{v,r}$ is $(v,r,1)$. If $e=((v,r,m_1), (v,r,m_2))$ is an edge (i.e., $(v,r,m_1)$ is the parent of $(v,r,m_2)$), then it holds
\[\limk \frac{\rho^{m_1}_k}{\rho^{m_2}_k} = \infty\]
and $x^{m_2}_k \in B_{2\rho^{m_1}_k}(x^{m_1}_k)$ for all $k$ sufficiently large. After possibly decreasing $\alpha_0$, $\vec \Theta^{m_2}_k$ restricted to the corresponding neck region 
\begin{equation}
\Omega_k^{\alpha,\beta}(e) \cqq B_{\alpha  \rho^{m_1}_k(\rho^{m_2}_k)^{-1}}(0)\setminus B_{\beta^{-1}}(0),\quad 0<\alpha,\beta \leq \alpha_0,\label{neck region around concentration points}
\end{equation}
satisfies the assumptions of \zcref{lem:almost pointwise control conformal factor} by \cite[(A.14)]{ScharrerWestEnergyQuantization}. After passing to a subsequence, $R^{m_1}$ denotes the points around which the bubbles concentrate, i.e.,
\begin{equation}
(\rho^{m_1}_k)^{-1}(x^{m_2}_k - x^{m_1}_k) \to r^{m_2} \in R^{m_1} \subset \C.\label{xm2 converges to a point in R m2}
\end{equation}
This defines a bijection\footnote{This follows from the way the bubbles are constructed, see \cite{BernardRiviere,ScharrerWestEnergyQuantization}.} between $R^{m_1}$ and the children of $(v,r,m_1)$. To the edge $e$, we associate the pair
\begin{equation}
(q_1(e), q_2(e)) \cqq (r^{m_2}, \infty).\label{q1 q2 for edge in concentration tree}
\end{equation} 
\subsection{Diverging conformal classes for tori}\label{subsec:diverging conformal classes torus}
Suppose $p=1$, i.e., $\Sigma = \mb T^2$. Given a complex structure $c$ on $\mb T^2$, let $h$ be its corresponding Poincar\' e metric of unit volume. By the uniformization theorem, $(\mb T^2, h)$ is isometric to the Riemann surface
\[\C\bigg/\left (\frac{1}{\sqrt{\Im \omega}}\Z+\frac{\omega}{\sqrt{\Im \omega}} \Z\right ),\]
for some unique $\omega = \omega(c) $ lying in the fundamental domain
\begin{equation}
 \omega \in \left \{z \in \mb H, \, -\frac{1}{2} < \Re z < 0,\, |z|>1 \right \}  \cup  \left \{z \in \mb H, \, 0\leq\Re z \leq \frac{1}{2}, \, |z|\geq 1\right \} \label{fundamental domain}
\end{equation}
of $\mb H/\PSL_2(\Z)$. Also useful is the isometry of $(\mb T^2, h)$ to the cylinder 
\[ C_{l} \cqq \frac{1}{\sqrt{2\pi l}}(\mb S^1 \times [0,l]),\]
where $l =2\pi \Im \omega$ and the boundary components are identified via
\[\frac{1}{\sqrt{2\pi l }}(\theta,0) \sim \frac{1}{\sqrt{2\pi l }}(\theta + 2\pi \Re \omega,l).  \]
This cylinder admits the conformal chart 
\begin{equation}
\chi\colon \mb D\setminus B_{e^{-l}}  \to C_{l}, \quad (\theta, r)\mapsto \left (\frac{\cos(\theta)}{\sqrt{2\pi l}}, \frac{\sin(\theta)}{\sqrt{2\pi l}}, -\frac{\log(r)}{\sqrt{2\pi l}}\right ).\label{psi for torus case}
\end{equation}
A sequence of conformal structures $c_k$ corresponding to elements $\omega_k$ in the fundamental domain diverges to the boundary of the moduli space if and only if $\Im \omega_k\to \infty$.

Let $\Lambda > 0$. Suppose that $\vec \Phi_k\in \mc E_\Sigma$ be a sequence of weak immersions into $\R^3$ with
\begin{equation}
    \sup_{k\in \N} \mc E (\vec \Phi_k) < \Lambda
\end{equation}
and suppose that the conformal classes $c_k$ associated to $\vec \Phi_k$ diverge to the boundary of the moduli space. \cite{LaurainRiviereGreenFunction} show that the conformal factors $u_k$ of the map $\vec \Phi_k \circ \chi_k$ defined by $(\vec \Phi_k \circ \chi_k)^* g_{\R^3} = e^{2u_k} g_{\R^2}$ satisfy
\begin{equation}
\sup _{k\in \N}\|\nabla u_k\|_{L^{2,\infty}(\mb D\setminus B_{e^{-l_k}} )} <C(\Lambda).\label{L2 weak estimate for conformal factor in torus case}
\end{equation}
\subsection{The graph structure} \label{subsec:Graph structure}
Let us briefly summarize the analysis so far. Up to now, we have only used the assumption \eqref{bounded energies for phi k}, not \eqref{energy threshold}. With \zcref{thm:thick part convergence}, we have found weak limits of the immersions on the thick parts. \zcref{prop:Thin part bubble neck decomposition} allowed us to find bubbles on the thin part. Finally, in both of these cases, further concentration may occur around finitely many points $r$, which are tracked in \zcref{subsec:Analysis concentration points}. The neck regions between different bubbles are constructed in such a way that the conclusion of \zcref{lem:almost pointwise control conformal factor} holds for $\delta = 1/3$. The following definition organizes the previous analysis. We define a graph structure whose vertices are bubbles of our sequence such that edges describe adjacent bubbles. To each bubble, we associate the data from the previous compactness theorems, while to each edge, we associate the points around which the bubbles are attached to each other. Recall from \eqref{thick and thin parts} that $V' = V_{\mathrm{thick}}\cup  V_{\mathrm{thin}}$.

\begin{defi}[Bubble graph structure for $p\geq 2$] \label{def:Graph structure}
Let $0<\Lambda<\infty$. Suppose that $\Sigma$ is a genus-$p$ surface, $p\geq 2$, and $\vec \Phi_k \in \mc E_{\Sigma}$ is a sequence of weak immersions such that
\[ \sup_{k\in \N}\mc E(\vec \Phi_k) < \Lambda.\]
The set $V'$ is defined in \zcref{subsec:ThinPart}. We give $V'$ a graph structure\footnote{More precisely, the structure of a directed multigraph permitting loops, i.e., a set of vertices $V$ with edges $e\colon E \to  V\times V$. It will follow from \eqref{quotient of scales for adjacent bubbles} and \eqref{quotient of scales for adjacent bubbles 2} that loops never actually occur.} $G' = (V', E')$ by defining the set of vertices $e' = e'(i,l) \in  V' \times V'$
\begin{equation}
e(i,l) \cqq \begin{cases}
(v_1, (i,0)),& l=0\quad\text{and}\quad N_i \geq 1,\\
((i,l-1), (i,l)),&l \in \left \{1,\ldots, N_i-1\right \},\\
((i,N_i-1), w_2),&l=N_i\quad\text{and}\quad N_i \geq 1,\\
(v_1, v_2),&N_i=0,
\end{cases}\label{edges on the thin parts}
\end{equation}
for $i\in \left \{1,\ldots, \mc N\right \}$, $l\in \left \{0,\ldots, N_i\right \}$ and $v_1,\, v_2 \in V_{\mathrm{thick}}$ such that $q^i_1 \in v_1, \, q^i_2 \in v_2$. Here, $\mc N$ is the number of thin parts from \zcref{subsec:Mumford-Deligne compactness} and $N_i$ is the number of bubbles in each thin part from \zcref{prop:Thin part bubble neck decomposition}. We define

\begin{equation}\label{all the bubbles}
V^v_{\text{conc}} \cqq \bigcup_{r \in R^v}  V^{v,r}_{\text{conc}},\quad V_{\text{conc}} \cqq \bigcup_{v \in V'} V^{v}_{\text{conc}}, \quad V \cqq V' \cup V_{\text{conc}}. 
\end{equation}

Here, $R^{v}$ is as defined below. $V$ is given a graph structure $G=(V,E)$ by attaching for each $v \in V'$ and $r \in R^{v}$ the tree $T^{v,r}$ to $G'$ by adding the edge $(v, (v, r,1))$. Furthermore, to each $v \in V$, the data corresponding to \eqref{limit immersion on nodal part} if $v\in V_{\mathrm{thick}}$, \eqref{bubble convergence in neck regions} if $v\in V_{\mathrm{thin}}$, and \eqref{concentration point convergence} if $v\in V_{\text{conc}}$ is associated. The corresponding quantities are denoted by $s^v_k,\, y^v_k,\, R^v$, $\vec \Psi^v_k$, and $ \vec \Psi^v$. We also define
\begin{equation}
Q^v \cqq \begin{cases}
Q \cap v, &  v\in V_{\mathrm{thick}},\\
\{0,\infty\}, &  v \in V_{\mathrm{thin}},\\
\emptyset, &  v \in V_{\text{conc}}.
\end{cases}\label{Qv for bubbles v}
\end{equation}
To each edge $e \in E$, we asssociate two points $(q_1(e), q_2(e))$. If $e$ is an edge in $T^{v,r}$, they are defined via \eqref{q1 q2 for edge in concentration tree}. If $e = e(i,l) \in E'$, we set
\begin{equation}
(q_1(e), q_2(e)) \cqq \begin{cases}
(q^i_1, \infty),& l=0\quad\text{and}\quad N_i \geq 1,\\
(0, \infty),&l \in \left \{1,\ldots, N_i-1\right \},\\
(0, q^i_2),& l=N_i\quad\text{and}\quad N_i \geq 1,\\
(q^i_1, q^i_2),& N_i = 0.
\end{cases}\label{end points q1 and q2}
\end{equation}
If $e = (v, (v, r,1))$ for some $v \in V'$, $r \in R^{v}$, we set
\begin{equation}
(q_1(e), q_2(e)) \cqq (r, \infty).\label{q1 q2 connecting edge from tree to rest}
\end{equation}
Finally, to each edge $e$ and $0<\alpha, \beta \leq \alpha_0$ corresponds a neck region
\begin{equation}
\Omega_k^{\alpha,\beta}(e) \cqq \begin{cases}
\eqref{neck region around concentration points},&\text{if $e = ((v,r,m_1),(v,r,m_2))$,}\\
B_{\alpha (\rho^{1}_k)^{-1}}(0) \setminus B_{\beta^{-1}}(0),& \text{if $e=(v, (v,r,1))$,}\\
B_{\alpha a^{i,l}_k (a^{i,l+1}_k)^{-1}}(0)\setminus B_{\beta^{-1}}(0),&  \text{if $e = e(i,l), \, l \in \{0,\ldots, N_i\}$},
\end{cases}\label{neck region Omega k (e)}
\end{equation}
where $\alpha_0$ is sufficiently small. The corresponding reparametrizations of $\vec \Phi_k$ associated to $\Omega^{\alpha_0,\alpha_0}_k(e)$ are given by
\begin{equation}
\vec \Theta_k(e) \cqq \begin{cases}
\vec \Theta^{m_2}_k,& \text{if $e = ((v,r,m_1),(v,r,m_2))$}, \\
\vec \Theta^{1}_k,& \text{if $e = (v, (v,r,1))$}, \\
\vec \Phi_k \circ \chi^i_k(a^{l+1}_k \cdot), & \text{if $e = e(i,l), \, l \in \{0,\ldots, N_i\}$},
\end{cases} \label{Theta k (e) parametrization}
\end{equation}
where $\vec \Theta^{m}_k$ was defined in \eqref{Theta k parametrization in neck region}. $\vec \Theta_k(e)$ restricted to $\Omega_k^{\alpha_0, \alpha_0}(e)$ satisfies the assumptions of \zcref{lem:almost pointwise control conformal factor} for $\delta = 1/3$ and $\Lambda$ given by \eqref{bounded conformal factors in nodal region 1} and \eqref{Lambda prime choice}. Hence, we may associate to each edge the integer $m = m(e) \in \Z \setminus \{0\}$ from \zcref{lem:almost pointwise control conformal factor}. Sometimes, we will identify $v$ with its limiting immersion $\vec \Psi^v$.
\end{defi}

The definition for tori is similar. We may cut the torus in such a way that \eqref{alpha 0 choice for thin part} holds for $\vec \Phi_k\circ \chi_k$, where $\chi_k$ is given by \eqref{psi for torus case}, see also \cite[Section 3.2.1]{LaurainRiviereEnergyQuantization}. This and \eqref{L2 weak estimate for conformal factor in torus case} allows us to apply \zcref{prop:Thin part bubble neck decomposition} with the choice of $\eps$ given by \zcref{lem:almost pointwise control conformal factor} to split the torus into bubble regions and neck regions.

\begin{defi}[Graph structure for $p=1$] \label{def:Graph structure for p equals 1}
We define $V' \cqq \left \{0,\ldots, N-1\right \}$, $E' \cqq \{(0,1),\ldots, (N-2,N-1), (N-1,0)\}$ and $G' = (V', E')$, where $N$ is the number of bubbles coming from \zcref{prop:Thin part bubble neck decomposition}. The edge $(N-1,0)$ corresponds to the neck region in which we placed the cut for the chart. As done in \zcref{subsec:Analysis concentration points}, to each $v \in V'$ and $r \in R^{v}$, we can associate the tree $T^{v,r}$ with vertices $V^{v,r}_{\text{conc}}$ and define $V_{\text{conc}}$ and $V$ as in \eqref{all the bubbles}. $G$ is defined by attaching the trees $T^{v,r}$ to $G'=(V',E')$ for all $v\in V', \, r \in R^{v}$. The remaining data is defined analogously as in \zcref{def:Graph structure} (treating the bubbles $V'$ as $V_{\mathrm{thin}}$).
\end{defi}

Notice that the lower semicontinuity of the Dirichlet energy and the Willmore energy on each of the bubbles implies
\begin{equation+}
\liminf _{k\to \infty}\mc W(\vec \Phi_k) \geq \sum_{v \in V} \mc W(\vec \Psi^v),\quad \liminf _{k\to \infty}\mc E(\vec \Phi_k) \geq \sum_{v \in V} \mc E(\vec \Psi^v).\label{lower semicontinuity bubble graph energy}
\end{equation+}
\section{Scales and positions of the bubbles} 
\label{sec:Scales and positions}
\subsection{Adjacent bubbles}\label{subsec:Adjacent bubbles}
The methods in this section are general and allow to compare different bubbles in the graph structure defined in Definitions \ref{def:Graph structure} and \ref{def:Graph structure for p equals 1}.

\begin{lemma} \label{lem: ends and branch points}
Suppose $v_1, \, v_2 \in V$ are connected by $e = (v_1, v_2) \in E$ and let $m=m(e)$ be as in Definitions \ref{def:Graph structure} or \ref{def:Graph structure for p equals 1}. Then, if $m\geq 1$, $\vec \Psi^{v_1}$ has a branch point of order $m$ at $q_1(e)$ while $\vec \Psi^{v_2}$ has an end of order $-m$ at $q_2(e)$ in the sense of \zcref{lem: Muller Sverak variant}. Furthermore, it holds
\begin{equation}
\limk \frac{s^{v_1}_k}{s^{v_2}_k} = \infty.\label{quotient of scales for adjacent bubbles}
\end{equation}
If $\Omega^{\alpha,\beta}_k(e)$ is as in \eqref{neck region Omega k (e)} with corresponding reparametrization $\vec \Theta_k(e)$ from \eqref{Theta k (e) parametrization}, then it holds for $0<\alpha < \alpha_0$
\begin{equation}
\diam\left ( \vec \Theta_k(e) \vert _{\Omega_k^{\alpha, \alpha_0}(e)}\right ) \leq C s^{v_1}_k \alpha^{m-1/3}, \quad \mc A\left ( \vec \Theta_k(e) \vert _{\Omega_k^{\alpha, \alpha_0}(e)}\right ) \leq C (s^{v_1}_k)^2 \alpha^{2(m-1/3)},\label{diameter estimate 1}
\end{equation}
where $C$ does not depend on $k$ or $\alpha$. It holds
\begin{equation}
\vec \Psi^{v_1}(q_1(e)) = \limk (s^{v_1}_k)^{-1}(y^{v_2}_k - y^{v_1}_k).\label{yk alignment 1}
\end{equation}
If $m \leq -1$, then $\vec \Psi^{v_1}$ has an end of order $m$ at $q_1(e)$, while $\vec \Psi^{v_2}$ has a branch point of order $-m$ at $q_2(e)$ and it holds
\begin{equation}
\limk \frac{s^{v_1}_k}{s^{v_2}_k} = 0,\label{quotient of scales for adjacent bubbles 2}
\end{equation}
\begin{equation}
\diam\left ( \vec \Theta_k(e) \vert _{\Omega_k^{\alpha_0, \alpha}	(e)}\right ) \leq C s^{v_2}_k \alpha^{-m - 1/3}, \quad \mc A\left ( \vec \Theta_k(e) \vert _{\Omega_k^{\alpha_0, \alpha}(e)}\right ) \leq C (s^{v_2}_k)^2 \alpha^{2(-m-1/3)},\label{diameter estimate 2}
\end{equation}
and
\begin{equation}
\vec \Psi^{v_2}(q_2(e)) = \limk (s^{v_2}_k)^{-1}(y^{v_1}_k - y^{v_2}_k). \label{yk alignment 2}
\end{equation}
\end{lemma}

\begin{proof}
We may assume that $e$ is of the form $e(i,l)$ from \eqref{edges on the thin parts}. The case when $e$ is of the form $(v, (v,r,1))$ or $((v,r,m_1), (v,r,m_2))$ is analogous (the only change is that we use additionally the definition of $q_1(e)$ as in \eqref{xm2 converges to a point in R m2}, \eqref{q1 q2 for edge in concentration tree}, and \eqref{q1 q2 connecting edge from tree to rest}). Suppose first that the edge $e(i,l)$ connects the two bubbles $v_1 = (i,l-1)$ and $v_2 = (i,l)$ in $V_{\mathrm{thin}}$, i.e., $l \in \left \{1,\ldots, N_i-1\right \}$. The conformal factor of $\vec \Psi^l_k = (s^l_k)^{-1} (\vec \Theta_k(e(i,l)) - y^l_k)$ is given by
\begin{equation}
\lambda_{\vec \Psi^l_k}(z) = - \log(s^l_k) +  \lambda_{\vec \Theta_k(e(i,l))}( z),\label{conformal factor Psi j k}
\end{equation}
where $\lambda_{\vec \Theta_k(e(i,l))}$ is the conformal factor of $\vec \Theta_k(e(i,l))$. By \eqref{conformal factor estimate thin part} and \eqref{conformal factor Psi j k}, the scales $s^l_k$ in \eqref{bubble convergence in neck regions} satisfy for $\alpha>0$
\begin{equation}
\|\lambda_{\vec \Psi^l_k}\|_{L^\infty(B_{\alpha^{-1} } \setminus \bigcup_{q \in \left \{0\right \} \cup R^l} B_{\alpha}(q))} \leq C_\alpha. \label{lambda Psi l k bounded}
\end{equation}
As in \zcref{def:Graph structure} and \zcref{lem:almost pointwise control conformal factor}, there is $m = m(e(i,l)) \in \Z\setminus \left \{0\right \}$ such that
\begin{equation}
| \lambda_{\vec \Theta_k(e(i,l))}(x) -  \lambda_{\vec \Theta_k(e(i,l))}(y) - (m-1) \log(|x|/|y|)| \leq \frac{1}{3} |\log(|x|/|y|)| + C (\Lambda)\label{recall lambda k estimate}
\end{equation}
for all $x,y  \in  \Omega_k^{\alpha_0, \alpha_0}(e(i,l))$. Using \eqref{conformal factor Psi j k} and the fact that $\vec\Theta_k(e(i,l-1))(z_2)=\vec\Theta_k(e(i,l))(a_k^l (a_k^{l+1})^{-1}z_2)$, we observe
\begin{align}
        |\lambda_{\Psi^l_k}(z_1)-\lambda_{\Psi_k^{l-1}}(z_2)|
        &= \Bigg| \log \left (\frac{s^{l-1}_k}{s^{l}_k}\right ) - \log\left(\frac{a_k^l}{a_k^{l+1}}\right) + (m-1)\log\left(\frac{a_k^{l+1}}{a_k^{l}}\frac{|z_1|}{|z_2|}\right)\\
        &\qquad + \lambda_{\vec\Theta_k(e(i,l))}(z_1)-\lambda_{\vec\Theta_k(e(i,l))}(a_k^l(a_k^{l+1})^{-1}z_2)-(m-1)\log\left(\frac{a_k^{l+1}}{a_k^{l}}\frac{|z_1|}{|z_2|}\right) \Bigg|.
\end{align}
 Suitably choosing $z_1,z_2$ and using \eqref{conformal factor Psi j k}, \eqref{lambda Psi l k bounded}, \eqref{ratios of the ai and bi}, and \eqref{recall lambda k estimate}, we deduce
\begin{equation}
\left | \log \left (\frac{s^{l-1}_k}{s^{l}_k}\right ) - m \log\left ( \frac{a^{l}_k}{a^{l+1}_k}\right ) \right | \leq \frac{1}{3} \log\left ( \frac{a^{l}_k}{a^{l+1}_k}\right )  + C(\Lambda).\label{sl and sl+1 scale}
\end{equation}
\eqref{sl and sl+1 scale} directly implies \eqref{quotient of scales for adjacent bubbles} and \eqref{quotient of scales for adjacent bubbles 2} since $\frac{a^{l}_k}{a^{l+1}_k} \to \infty$ as $k\to \infty$ by \eqref{ratios of the ai and bi}. Using the notation $\underline \lambda(r) = \dashint_{\partial B_r} \lambda(x) \dif l(x)$, it holds as a consequence of \eqref{ratios of the ai and bi} and \eqref{bubble convergence in neck regions} that
\[\underline \lambda_{\vec \Psi^l_k}(r) \to \underline\lambda_{\vec \Psi^l}(r)\quad\text{pointwise for all $r>\alpha_0^{-1}$}.\]
In particular, owing to \eqref{recall lambda k estimate}, \eqref{conformal factor Psi j k}, $\underline \lambda_{\vec \Psi^l}$ also satisfies \eqref{recall lambda k estimate} on $\C \setminus B_{\alpha_0^{-1}}$. Since $\lambda_{\vec \Psi^l}$ satisfies \eqref{Muller Sverak estimate}, see also \zcref{rem:order one end is embedded}, we deduce that the end (respectively branching) order of $\vec \Psi^l$ at $q_2(e) = \infty$ equals $-m(e)$. A similar argumentation for $\vec \Psi^{l-1}_k$ yields that the branching (respectively end) order of $\vec \Psi^{l-1}$ at $q_1(e) = 0$ equals $m(e)$. \eqref{diameter estimate 1} and \eqref{diameter estimate 2} are consequences of \eqref{lambda Psi l k bounded} and \eqref{recall lambda k estimate}. Indeed, suppose for simplicity $m\leq -1$. Choosing $z_0 \in \partial B_{\alpha_0^{-1}} \setminus R^l$, we deduce from \eqref{conformal factor Psi j k}, \eqref{lambda Psi l k bounded}, and \eqref{recall lambda k estimate} for sufficiently small $\alpha>0$
\begin{align}
    \mathcal A\left ( \vec \Theta_k(e) \vert _{\Omega_k^{\alpha_0, \alpha}(e)}\right ) &= \int_{\Omega_k^{\alpha_0, \alpha}(e)} e^{2\lambda_{\Theta_k(e)}(z)}\dif d z \\
    &\leq \int _{\alpha^{-1}} ^{\alpha_0 \frac{a^{i,l}_k}{a^{i,l+1}_k}} 2\pi r e^{2 \lambda_{\vec \Theta_k(e)}(z_0)} 
    \left (\alpha_0 r\right )^{2(m-1 + 1/3)} \dif d r\\
    &\leq C\int_{\alpha^{-1}}^\infty (s^{v_2}_k)^2  r^{2(m + 1/3) -1} \dif d r \\
    &\leq C (s^{v_2}_k)^2 \alpha ^{-2(m+1/3)}.
    \end{align}
The diameter estimate is handled analogously. The case $m\geq 1$ is also similar, choosing $z_{0,k} \cqq \frac{a^{i,l}_k}{a^{i,l+1}_k} z_0$ for $z_0 \in \partial B_{\alpha_0 } \setminus R^{l-1}$ instead.

From \eqref{lambda Psi l k bounded}, it follows for any $\alpha>0$
\[\diam\bigg (\vec \Phi_k \circ \chi^i_k\bigg ( B_{\alpha^{-1} b^l_k}\setminus \bigcup_{q \in \left \{0\right \} \cup R^l} B_{\alpha a^{l+1}_k}(q)\bigg )\bigg ) \leq C s^l_k. \]
Assume $m\geq 1$. We choose $\alpha_1 < \alpha_0$ sufficiently small such that $z^l \in B_{\alpha_1^{-1}} \setminus \bigcup _{q \in \left \{0\right \} \cup R^l} B_{\alpha_1}(q)$, where $z^l$ was chosen in \zcref{prop:Thin part bubble neck decomposition}. Together with \eqref{diameter estimate 1}, we have for any $z\neq 0$
\begin{align}
&\quad \, (s^{l-1}_k)^{-1}|\vec \Phi_k \circ \chi^i_k(a^{l}_k z) - \underbrace{\vec \Phi_k \circ \chi^i_k(a^{l+1}_k z^{l})}_{=y^l_k}| \notag\\
&\leq (s^{l-1}_k)^{-1}\diam \left (\vec \Phi_k \circ \chi^i_k \left (B_{|z| a^l_k}\setminus B_{\alpha_1^{-1}a^{l+1}_k}\right )\right ) \notag\\
&\quad\, + (s^{l-1}_k)^{-1}\diam \bigg (\vec \Phi_k \circ \chi_k \bigg (B_{\alpha_1^{-1}b^l_k}\setminus \bigcup _{q\in \left \{0\right \}\cup R^l}B_{\alpha_1 a^{l+1}_k} (q)\bigg )\bigg )\notag\\
&\leq C   |z|^{m-1/3} + C_{\alpha_1} \frac{s^l_k}{s^{l-1}_k} \label{diameter estimate applied}.
\end{align}
In particular, \eqref{diameter estimate applied}, $m\geq 1$, and \eqref{quotient of scales for adjacent bubbles} imply
\begin{equation}
\lim_{z\to 0} \limk (s^{l-1}_k)^{-1}|\vec \Phi_k \circ \chi^i_k(a^{l}_k z) - y^l_k| = 0.\label{uniform convergence around branch point in bubble}
\end{equation}
As $\vec \Psi^{l-1}$ has a branch point at $q_1(e)=0$, it is in particular continuous through $0$ and it holds
\begin{align}
\vec \Psi^{l-1}(q_1(e)) &= \lim_{z\to 0} \vec \Psi^{l-1}(z) = \lim_{z\to 0}\limk  (s^{l-1}_k)^{-1}\left ( \vec \Phi_k \circ \chi^i_k (a^{l}_k z) - y^{l-1}_k\right )\\
&= \limk (s^{l-1}_k)^{-1}(y^l_k - y^{l-1}_k),
\end{align}
which is \eqref{yk alignment 1}. The last equality comes from \eqref{uniform convergence around branch point in bubble}.

The case $m\leq -1$ is analogous. In the case that $e$ connects one or two bubbles from the thick part, it suffices to look at the behavior of $\vec \Psi^{v_1}$ around $q_1(e)$. We work with the parametrization $\tilde{\vec \Psi}^v$ as in \eqref{weak convergence around point in Q}. Since $\vec \Psi^v(q_1(e)) = \tilde{\vec \Psi}^v(0)$ and by \zcref{rem:Behavior of Psi v around Q v}, we have that $\|\lambda_k - \log s^{v_1}_k \|_{L^\infty(B_{\alpha_0}\setminus B_{\alpha})} \leq C_{\alpha}$, we can finish the proof as before.
\end{proof}

\subsection{Inversions} \label{subsec:Inversions}
For $p\in \R^3$, we define the inversion
\[I_p\colon\R^3 \setminus \{p\} \to \R^3 \setminus \{0\}, \quad I_p(z)\cqq \frac{z-p}{|z-p|^2}.\]
For $\nu \in \mb S^2$, we also define the reflection $J_{\nu}$
\[J_{\nu}(x) \cqq x - 2 \nu \langle \nu,x \rangle.  \]
If $p_k \not \in \vec \Phi_k(\Sigma)$, the inverted immersion $\hat{\vec{\Phi}}_k \cqq I_{p_k} \circ \vec \Phi_k$ is still a weak, conformal immersion with respect to the metric $h_k$. In particular, the conformal structures and thus also the corresponding nodal surface remain unchanged. It holds that
\begin{equation}
\hat{\vec{\Phi}}_k =  (s^v_k)^{-1} I_{p^v_k} \circ ((s^v_k)^{-1}(\vec \Phi_k - y^v_k)),\label{inversion hat Phi k}
\end{equation}
where $p^v_k \cqq (s^v_k)^{-1}(p_k - y^v_k)$. Therefore, we set
\begin{equation}
\hat{\vec{\Psi}}_k^v \cqq  I_{p^v_k} \circ \vec \Psi^v_k,\quad v \in V,\label{definition hat Psi v k}
\end{equation}
which will be used to describe limits associated to the inverted immersion. Two cases can happen: If $\limk p^v_k = \infty$, then after passing to a subsequence, $\frac{p^v_k}{|p^v_k|} \to \nu \in \mb S^2$. An easy calculation shows that
\begin{equation}
|p^v_k|^2 I_{p^v_k} (\cdot) + p^v_k \to J_\nu \quad\text{in $C_{\loc}^{\infty}(\R^3;\R^3)$ as $k\to \infty$}.\label{inversion converges to reflection}
\end{equation}
\eqref{inversion hat Phi k}, \eqref{inversion converges to reflection}, \eqref{limit immersion on nodal part}, and \eqref{bubble convergence in neck regions} yield
\begin{equation}
|p^v_k|^2 \hat{\vec{\Psi}}^v_k + p^v_k \wto \hat{\vec{\Psi}}^v \cqq J_\nu \circ \vec \Psi^v \begin{array}{l}
\quad\text{in $W^{2,2}_{\loc}(v \setminus (Q^v \cup R^v);\R^3)$ if $v\in V_{\mathrm{thick}}$,}\\
 \quad\text{in $W^{2,2}_{\loc}(\C \setminus (Q^v \cup R^v);\R^3)$ if $v \in V_{\mathrm{thin}} \cup V_{\text{conc}}$.}
\end{array} \label{inversion on bubbles 1}
\end{equation}
In particular, $\hat{s}^v_k = (|p^v_k|^2 s^v_k)^{-1}$ and $\hat{y}^v_k = - \hat{s}^v_k p^v_k$.
 
If $\limk p^v_k \qqc p^v \in \R^3$, then \eqref{limit immersion on nodal part} and \eqref{bubble convergence in neck regions} yield
\begin{equation}
\hat{\vec{\Psi}}^v_k \wto \hat{\vec{\Psi}}^v \cqq I_{p^v} \circ \vec \Psi^v \begin{array}{l}
\quad\text{in $W^{2,2}_{\loc}(v \setminus (Q^v \cup \hat{R}^v);\R^3)$ if $v\in V_{\mathrm{thick}}$,}\\
 \quad\text{in $W^{2,2}_{\loc}(\C \setminus (Q^v \cup \hat{R}^v);\R^3)$ if $v \in V_{\mathrm{thin}} \cup V_{\text{conc}}$,}
\end{array}\label{inversion on bubbles 2}
\end{equation}
where $\hat{R}^v \subset R^v \cup (\vec \Psi^v)^{-1}(p^v)$ is the set of concentration points for the inversions. In particular, $\hat{s}^v_k = (s^v_k)^{-1}$ and $\hat{y}^v_k = 0$.

 Notice however that not all of the bubbles corresponding to the immersions $\vec \Phi_k$ translate to bubbles of the immersions $\hat{\vec \Phi}_k$ via \eqref{inversion on bubbles 1}, \eqref{inversion on bubbles 2}. Indeed, bubbles from $v\in V_{\mathrm{thin}} \cup V_{\text{conc}}$ arose because of energy concentration. If the inversions $\hat{\vec{\Psi}}^v_k$ do not concentrate any energy, we do not see them in the inversions, i.e., there is no matching bubble in $\hat{V}$, the set of bubbles for the inversions $\hat{\vec{\Phi}}_k$. Likewise, there might be different scales on which energy becomes concentrated with respect to the inversions $\hat{\vec{\Phi}}_k$ that we did not capture for $\vec{\Phi}_k$.

\subsection{Bubble descents and bubble ascents}\label{subsec:Bubble descent and ascent}
We will introduce bubble descents (respectively bubble ascents), which allow us to traverse along the bubbles in the graph in such a way that we decrease (respectively increase) the scale $s^v_k$ of the bubbles in each step.

\subsubsection*{Bubble descent in \texorpdfstring{$V$}{V}} 
Let $v_1\in V$ and suppose that there is $q \in Q^{v_1}\cup R^{v_1}$ such that $\vec \Psi^{v_1}$ has a branch point around $q$. We choose a maximal path $v_i \in V$, $i\in \left \{1,\ldots, i_0\right \}$, called \emph{a bubble descent} such that for all $i$, it holds
\begin{enumerate}[wide=0pt,leftmargin=*, labelsep=0.5em,label = 1\alph*)]
\item $v_i$ and $v_{i+1}$ are connected by an edge $e_i \in E$ (we change the orientation of the edge in such a way that it starts at $v_i$ and ends at $v_{i+1}$). Furthermore, $\vec \Psi^{v_i}$ has a branch point at $q_1(e_i)$; \label{item 1a}
\item the bubbles $v_i, \, i\in \left \{1,\ldots, i_0\right \}$, are pairwise distinct.\label{item 1b}
\end{enumerate}
To make this precise, assume that we have chosen $v_1,\ldots, v_i$ satisfying \ref{item 1a} and \ref{item 1b} with $q_1(e_1) = q$. Note that for $i>1$, $\vec \Psi^{v_i}$ has an end at $q_2(e_{i-1})$ which follows from \zcref{lem: ends and branch points} and the fact that $\vec \Psi^{v_{i-1}}$ has a branch point by \ref{item 1a}. If possible, we choose $q_i \in Q^{v_i}\cup R^{v_i}\setminus \left \{q_2(e_{i-1})\right \}$ such that $\vec \Psi^{v_i}$ has a branch point at $q_i$. Then there is an edge $e$ starting at $v_i$ (after a potential change of orientation) such that $q_i = q_1(e)$ and we set $e_{i} \cqq e$. We set $v_{i+1}$ such that $e_i = (v_{i},v_{i+1})$. The process stops if either there is no $q_i \in Q^{v_i} \cup R^{v_i} \setminus \left \{q_2(e_i)\right \}$ such that $\vec \Psi^{v_i}$ has a branch point at $q$, or $v_{i+1} = v_j$ for some $j\in \left \{1,\ldots, i\right \}$. Also notice that the process has to stop as there are only finitely many bubbles. However, $v_{i+1} = v_j$ can be excluded as \eqref{quotient of scales for adjacent bubbles} implies 
\begin{equation}
\limk\frac{s^{v_i}_k}{s^{v_{j}}_k} = \infty \quad \text{for $i<j$}. \label{scale of bubbles along the descent}
\end{equation}
It follows that $\vec \Psi^{v_{i_0}}$ has no branch points in $Q^{v_{i_0}}\cup R^{v_{i_0}}$. Furthermore, it follows from \eqref{yk alignment 1} for $j\in \left \{1,\ldots, i_0-1\right \}$
\begin{equation}
\vec \Psi^{v_j}(q_1(e_j)) = \limk (s^{v_j}_k)^{-1}(y^{v_{j+1}}_k - y^{v_j}_k).\label{yk scale for adjacent bubble descent}
\end{equation}
In particular, for $j\in \left \{2,\ldots, i_0-1\right \}$, using \eqref{scale of bubbles along the descent}, we deduce
\begin{equation}
0 = \limk (s^{v_1}_k)^{-1}(y^{v_{j+1}}_k - y^{v_j}_k).\label{smaller scale for yk}
\end{equation}
Adding \eqref{yk scale for adjacent bubble descent} for $j=1$ and \eqref{smaller scale for yk} for $j\in \left \{2, \ldots, i_0-1\right \}$, we obtain
\begin{equation}
\vec \Psi^{v_1}(q_1(e_1)) = \limk (s^{v_1}_k)^{-1}(y^{v_{i_0}}_k - y^{v_1}_k).\label{yk scale for bubble descent}
\end{equation}

\subsubsection*{Bubble ascent in \texorpdfstring{$V$}{V}}
A similar procedure can be done by traversing along the ends instead of branch points. Pick $v_1 \in V$ and $q \in Q^{v_1} \cup R^{v_1}$ such that $\vec \Psi^{v_1}$ has an end around $q$. Similar to the bubble descent, we define a maximal path $v_i \in V$, $i \in \left \{1,\ldots, i_0\right \}$, called \emph{a bubble ascent} satisfying
\begin{enumerate}[wide=0pt,leftmargin=*, labelsep=0.5em,label = 2\alph*)]
\item for $i\in \left \{1, \ldots, i_0 - 1\right \}$, $v_i$ and $v_{i+1}$ are connected by an edge $e_i \in E$ in such a way that it starts at $v_i$ and ends at $v_{i+1}$ (after a change of orientation). Furthermore, $\vec \Psi^{v_i}$ has an end at $q_1(e_i)$; \label{item 2a}
\item the bubbles $v_i, \, i \in \left \{1,\ldots, i_0\right \}$, are pairwise different.\label{item 2b}
\end{enumerate}
A maximal path is well-defined as we only have finitely many bubbles. This procedure terminates if and only if either $\vec \Psi^{v_{i_0}}$ has no ends around any point in $Q^{v_{i_0}}\cup R^{v_{i_0}}$, or, if we denote by $v_{i_0+1}$ the next bubble satisfying \ref{item 2a}, satisfies $v_{i_0+1} = v_{j}$ for some $j\in \left \{1,\ldots, i_{0}\right \}$. \eqref{quotient of scales for adjacent bubbles 2} implies that the second alternative cannot occur. It follows that $\vec \Psi^{v_{i_0}}$ has no ends. As for the bubble descent, using again \zcref{lem: ends and branch points}, we arrive at
\begin{equation}
\vec \Psi^{v_{i_0}}(q_2(e_{i_0-1})) = \limk (s^{v_{i_0}}_k)^{-1}(y^{v_{1}}_k - y^{v_{i_0}}_k).\label{yk scale for bubble ascent}
\end{equation}
\subsubsection*{Bubble descents and bubble ascents in \texorpdfstring{$V'$}{V'}}
Finally, bubble descents and bubble ascents in $V'$ are defined analogously, with the only difference that $v_1 \in V'$ and in each step, we require $q_i \in Q^{v_i} \setminus \{q_2(e_{i-1})\}$. This ensures that $v_i \in V'$ for all $i\in \{1,\ldots, i_0\}$. 
If $v_1,\ldots, v_{i_0}$ is a bubble descent in $V'$, then $\vec \Psi^{v_{i_0}}$ has no branch points around any of the points in $Q^{v_{i_0}}$. Likewise, if $v_1,\ldots, v_{i_0}$ is a bubble ascent in $V'$, then $\vec \Psi^{v_{i_0}}$ has no ends around any of the points in $Q^{v_{i_0}}$.

\section{Proof of Theorem \ref{thm:Asymptotic convergence}}\label{sec:Proof main theorem}
From now on, we will use the $8\pi$ threshold \eqref{energy threshold}.
\subsection{Excluding true branch points and finding one catenoid} \label{subsec:one catenoid}
\begin{lemma}\label{prop: no density 2}
It holds $m(e) = \pm 1$ for all $e\in E$. In particular, no branch points of order $2$ or ends of order $-2$ can occur for the maps $\vec \Psi^v$, $v\in V$.
\end{lemma}
\begin{proof}
Notice that the Li--Yau inequality implies that no branch points or ends of order $m$ with $|m|\geq 3$ can occur, as this needs at least $12\pi$ Willmore energy.

We begin by proving the following claim: Whenever $v\in V$ is such that $\vec \Psi^v$ has an end of order $-2$ at some point $r_1\in Q^v\cup R^v$, then there is $r_2 \in Q^v \cup R^v$ such that $\vec \Psi^v$ has a branch point of order $2$ at $r_2$.

We may work with inversions as given by \zcref{subsec:Inversions} such that $p^v \not \in \im \vec \Psi^v$. Since $\vec \Psi^v$ has an end of order $-2$, it follows that $\mc W(I_{p^v} \circ \vec \Psi^{v})\geq 8\pi$ by the Li--Yau inequality \eqref{Li Yau inequality}. Because of the energy threshold \eqref{energy threshold}, equality has to hold, which is the case if and only if $\vec \Psi^v$ is a minimal, branched immersion with density 2 at infinity. We give a proof of this fact in the more general setting of integral 2-varifolds in Appendix \hyperref[sec:Appendix Equality in the Li--Yau inequality]{C}.
 
Because of the end of order $-2$ around $r_1$, $\vec \Psi^v$ is not injective. \cite[5.1(2), 5.2(2)(b)]{Allard} implies that the varifold induced by $\vec \Phi^v$ is a minimal cone of density 2 at infinity, and since the Dirichlet energy is finite, the induced varifold is a density 2 plane. After identifying the image with $\C$ and possibly changing the orientation, the conformality of $\vec \Psi^v$ and the behavior around the branch and end points implies that $\vec \Psi^v$ is a meromorphic map, i.e., a holomorphic map to the Riemann sphere $\hat{\C}$. Since the density is $2$, $\vec \Psi^v$ has degree 2. The Riemann--Hurwitz formula, see \cite{Jost}, implies that $\vec \Psi^v$ has $2\genus(v)+1$ branch points of order $2$ and one end of order $-2$, which proves the claim.

Suppose now that there is $v = v_1\in V$ such that $\vec \Psi^v$ has either a branch point of order $2$ or an end of order $-2$. By the claim, we can find in both cases $r \in Q^v\cup R^v$ such that $\vec \Psi^v$ has a branch point of order $2$ at $r$. We start a bubble descent at $r$ in $v_1$ as in \zcref{subsec:Bubble descent and ascent} which yields the bubbles $v_1,\ldots, v_{i_0} \in V$ connected by edges $e_1,\ldots, e_{i_0-1} \in E$. The claim shows that we can choose this bubble descent in such a way that $m(e_i) = 2$ for all $i$. Indeed, if $\vec \Psi^{v_i}$ has a branch point of order $2$ at $q_1(e_i)$, $\vec \Psi^{v_{i+1}}$ has an end of order $-2$ at $q_2(e_i)$ and the claim implies that $\vec \Psi^{v_{i+1}}$ has another branch point of order $2$ at some $r_{i+1} \in Q^{v_{i+1}} \cup R^{v_{i+1}}$.

Since the claim now shows that $\vec \Psi^{v_{i_0}}$ has another branch point of order $2$ somewhere, this contradicts the maximality of the bubble descent, finishing the proof.
\end{proof}

The following lemma is a variation of \cite[Theorem 5.3]{KuwertLi}. We will use slightly different arguments and additionally use the fact that $\omega^3_p >8\pi$, see \eqref{eq:intro:MND}.

\begin{lemma} \label{lem:one catenoid}\leavevmode
\begin{enumerate}[label = \alph*)]
    \item There is at least one bubble $v\in V$ such that $\vec \Psi^v$ immerses a catenoid.
    \item \label{rem:all types of bubbles} All bubbles are of the following types:
\begin{itemize}
\item[\namedlabel{P}{$(P)$}] planes of density 1; $\mc W(P)=0$;
\item[\namedlabel{S}{$(S)$}] round spheres of density 1; $\mc W(S)=4\pi$;
\item[\namedlabel{C}{$(C)$}] catenoids; $\mc W(C)=0$;
\item[\namedlabel{IC1}{$(IC_1)$}] inversions of catenoids with inversion center in the image of the catenoid; $\mc W(IC_1)=4\pi$;
\item[\namedlabel{IC2}{$(IC_2)$}] inversions of catenoids with inversion center not in the image of the catenoid; $\mc W(IC_2)=8\pi$.
\end{itemize}
It follows $\genus(v) = 0$ for all $v\in V_{\mathrm{thick}}$. The only possibility for $\vec \Psi^v$ to be non-injective is the case \ref{IC1} or \ref{IC2}.
\end{enumerate}
\end{lemma}
\begin{proof}
 \emph{a)} Suppose first that $p = \genus(\Sigma) \geq 2$. Let us first show that $\mc M < \mc N+1$, where $\mc M$ and $\mc N$ were defined in \zcref{subsec:Mumford-Deligne compactness}. We already know that $\mc M\leq \mc N+1$, so assume that $\mc M=\mc N+1$. In this case, it follows from \eqref{sum of genera} that
 \begin{equation}
 \sum_{v\in V_{\mathrm{thick}}} \genus(v) = p.\label{sum of genera is p}
 \end{equation}
Fix any $v\in V_{\mathrm{thick}}$. If $\vec \Psi^v$ has two ends of order $-1$ and is in addition non-injective, the argument from \zcref{prop: no density 2} shows that $\vec \Psi^v$ can be viewed as a holomorphic map from $v$ to $\hat \C$. The Riemann-Hurwitz formula implies that $\vec \Psi^v$ has $2$ ends of order $-1$ and $2\genus(v)+2\geq 1$ branch points of order $2$, which we excluded in \zcref{prop: no density 2}. So in particular, if $\vec \Psi^v$ has density 2 at infinity, then there are two ends of order $-1$, $\vec \Psi^v$ is embedded and unbranched. By \cite{SchoenUniqueness}, $\vec \Psi^v$ immerses a catenoid. Thus, $v$ is homeomorphic to a sphere and $\genus(v)= 0$. 

So for all $v\in V_{\mathrm{thick}}$ with $\genus(v)\geq 1$, $\vec \Psi^v$ has at most one end of order $-1$ and it follows that $\mc W(\vec \Psi^v) \geq \beta^3_{\genus(v)} - 4\pi$. Furthermore, doing a bubble ascent in $V$ starting at any $v_1\in V$, we find at least one bubble $v_{i_0}$ such that $\vec \Psi ^{v_{i_0}}$ has no ends. If $v_{i_0} \in V_{\mathrm{thick}}$, then $\mc W(\vec \Psi^{v_{i_0}}) \geq \beta^3_{\genus(v_{i_0})}$, otherwise $\mc W(\vec \Psi^{v_{i_0}}) \geq 4\pi$. Using \eqref{sum of genera is p}, this implies
\begin{equation}
    \liminf_{k\to \infty} \mc W(\vec \Phi_k)\geq \sum_{v\in V_{\mathrm{thick}}\cup \{v_{i_0}\}} \mc W(\vec \Psi^v) \geq 4\pi + \sum_{\substack{v \in V_{\mathrm{thick}}\\ \genus(v)\geq 1}} (\beta_{\genus(v)}^3 -4\pi)\geq \omega^3_p>8\pi,\label{Marques Neves consequence}
\end{equation}
a contradiction to \eqref{energy threshold}. We conclude $\mc M< \mc N+1$.

This also means that the graph $G=(V,E)$ from \zcref{def:Graph structure} has a closed cycle (after possibly changing the orientation). This is in particular true if $p=1$, see \zcref{def:Graph structure for p equals 1}, hence we may drop the assumption $p\geq 2$ from now on. Denote such a cycle by $v_1,\ldots, \, v_n, \, v_{n+1} = v_1$ connected by edges $e_1,\ldots, e_n$. Using that this is a cycle and applying \eqref{quotient of scales for adjacent bubbles} and \eqref{quotient of scales for adjacent bubbles 2} shows that there must be $v_i, \, i \in \left \{1,\ldots ,n\right \}$, such that 
\begin{equation}
\limk \frac{s^{v_{i-1}}_k}{s^{v_i}_k} = \infty \quad \text{and}\quad \limk \frac{s^{v_{i+1}}_k}{s^{v_i}_k} = \infty.\label{finding smallest bubble in cycle}
\end{equation}
Here, $n+1$ is identified with $1$. Indeed, if \eqref{finding smallest bubble in cycle} was false for all $i$, $\limk \frac{s^{v_{i-1}}_k}{s^{v_i}_k}$ would be independent of $i$, showing that
\[1=\limk \frac{s^{v_1}_k}{s^{v_1}_k} = \limk \prod_{i=1}^n \frac{s^{v_i}_k}{s^{v_{i+1}}_k} \in \{0,\infty\}.\]
\eqref{finding smallest bubble in cycle} and using \zcref{lem: ends and branch points} prove that $v_i$ has an end at $q_2(e_{i-1})$ and at $q_1(e_{i})$. As above, $\vec \Psi^{v_i}$ immerses a catenoid.

\emph{b)} Using a), we can invert the catenoid bubble such that $\mc W(I_{p^{v_i}} \circ \vec \Psi^{v_i}) = 8\pi$ and using \zcref{subsec:Inversions}, we see that all other bubbles $\vec \Psi^{v}$ must either be unbranched (in the sense that no \emph{true} branch points occur, i.e., of order $2$ or higher) immersions of minimal surfaces with at most two ends or inversions of them. Another application of \cite{SchoenUniqueness} yields the result.
\end{proof}

\subsection{No catenoids are immersed on thick parts} \label{subsec:no catenoids}

\begin{lemma}\label{lemma:no catenoids}\leavevmode
\begin{enumerate}[label = \alph*)]
    \item For $v\in V_{\mathrm{thick}}$, $\vec \Psi^{v}$ does not immerse a catenoid \ref{C} or its inversions \ref{IC1} and \ref{IC2}.
    \item If $v \in V_{\mathrm{thin}}$ and $\vec \Psi^v$ immerses \ref{C}, the two ends are located at the two points in $Q^v$.
\end{enumerate}
\end{lemma}
\begin{proof}
\emph{a)} Suppose there is $v=v_1 \in V_{\mathrm{thick}} $ and the bubble $\vec \Psi^{v_1}$ immerses a catenoid \ref{C}. From \eqref{number of Q in v} and $\genus(v_1) = 0$ by \zcref{lem:one catenoid}\ref{rem:all types of bubbles}, we know that $\# (Q \cap v_1) \geq 3$. As $\vec \Psi^{v_1}$ has only two ends, there is $q \in Q^{v_1}$ such that $\vec \Psi^{v_1}$ has a branch point around $q$. We do a bubble descent in $V'$ starting at $q$. As stated in \zcref{subsec:Bubble descent and ascent}, the final bubble $v_{i_0}$ has no branch points around the points in $Q^{v_{i_0}}$. As $\# Q^{v_{i_0}} \geq 2$, $\vec \Psi ^{v_{i_0}}$ immerses \ref{C} and \eqref{yk scale for bubble descent} holds. 

Choose $p_0 \in \R^3 \setminus \im \vec \Psi ^{v_{i_0}}$ in such a way that $p^{\inv}_k \cqq y^{v_{i_0}}_k+ s^{v_{i_0}}_k p_0 \not \in \vec \Phi_k(\Sigma)$ for all $k$. The definition of $p^{\inv}_k$, \eqref{scale of bubbles along the descent}, and \eqref{yk scale for bubble descent} imply that
\begin{equation}
\vec \Psi^{v_1}(q_1(e_1)) = \limk (s^{v_1}_k)^{-1}(p^{\inv}_k - y^{v_1}_k).\label{inversion center location}
\end{equation}
Working with the inversions $\hat{\vec{\Phi}}_k = I_{p^{\text{inv}}_k} \circ \vec \Phi_k$, we deduce from the choice of $p_0$ and \zcref{subsec:Inversions} that $\mc W(\hat{\vec{\Psi}}^{v_{i_0}})= \mc W(  I_{p_0} \circ \vec \Psi^{v_{i_0}})= 8\pi$ as it is \ref{IC2} by construction. Furthermore, \eqref{inversion center location} together with \eqref{inversion on bubbles 2} yields that $\hat{\vec{\Psi}}^{v_1}$ immerses \ref{IC1}. Together, we arrive at at least $12\pi$ Willmore energy, exceeding \eqref{energy threshold}. Furthermore, \eqref{inversion on bubbles 1} and \eqref{inversion on bubbles 2} show that the immersions \ref{IC1} and \ref{IC2} cannot occur in thick parts either.

\emph{b)} If this was wrong, one of the points in $Q^v$ would have a branch point and we could run the same bubble descent in $V'$ as before. 
\end{proof}
 
 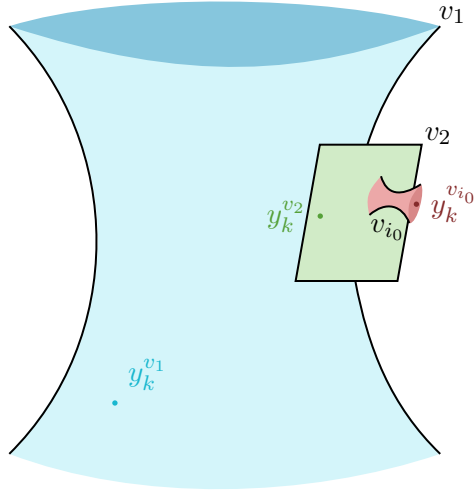
\begin{figure}[H] 
\centering 
\tikzset{every picture/.style={line width=0.75pt}} 

\begin{tikzpicture}[x=0.75pt,y=0.75pt,yscale=-1,xscale=1]

\draw  [draw opacity=0][fill={rgb, 255:red, 134; green, 198; blue, 221 }  ,fill opacity=1 ] (225.5,41.39) .. controls (280.5,24.07) and (385.5,25.57) .. (440.61,41.39) .. controls (374,90.07) and (337.73,97) .. (357.73,127) .. controls (377.73,157) and (265.5,64.57) .. (225.5,41.39) -- cycle ;

\draw  [draw opacity=0][fill={rgb, 255:red, 212; green, 245; blue, 250 }  ,fill opacity=1 ] (225.5,41.39) .. controls (315.13,70.64) and (373.13,66.64) .. (440.61,41.39) .. controls (383.13,98.64) and (382.13,195.14) .. (440.61,256.5) .. controls (374.13,279.14) and (289.63,281.14) .. (225.5,256.5) .. controls (285.63,191.64) and (283.13,102.14) .. (225.5,41.39) -- cycle ;
\draw    (225.5,41.39) .. controls (283.87,101.22) and (283.21,198.99) .. (225.5,256.5) ;
\draw  [fill={rgb, 255:red, 209; green, 235; blue, 199 }  ,fill opacity=1 ] (380.52,100.97) -- (431.36,100.97) -- (419.27,169.54) -- (368.43,169.54) -- cycle ;
\draw    (404.24,101.04) .. controls (410.52,81.66) and (420.07,62.54) .. (440.61,41.39) ;
\draw    (398.24,169.61) .. controls (404.57,203.28) and (415.91,232.45) .. (440.61,256.5) ;
\draw  [draw opacity=0][fill={rgb, 255:red, 216; green, 136; blue, 136 }  ,fill opacity=1 ] (423,129.7) .. controls (421.63,134.57) and (422,132.32) .. (425.14,140.58) .. controls (429.42,140.61) and (433.29,123.86) .. (430.98,120.89) .. controls (416.54,131.86) and (424.38,124.82) .. (423,129.7) -- cycle ;
\draw  [draw opacity=0][fill={rgb, 255:red, 236; green, 168; blue, 168 }  ,fill opacity=1 ] (411.07,116.84) .. controls (414.41,123.24) and (417.07,129.48) .. (430.98,120.89) .. controls (424.42,128.61) and (424.29,131.48) .. (425.14,140.58) .. controls (419.73,128.63) and (411.61,131.77) .. (405.24,136.53) .. controls (403.5,127.98) and (403.79,124.74) .. (411.07,116.84) -- cycle ;
\draw    (411.07,116.84) .. controls (413.77,126.24) and (420.78,127.43) .. (430.98,120.89) ;
\draw    (405.28,136.32) .. controls (413.58,129.06) and (421.64,131.36) .. (425.18,140.37) ;

\definecolor{myblue}{HTML}{18B7CF}
\definecolor{mygreen}{HTML}{56A039}
\definecolor{myred}{HTML}{852B2B}

\draw (438.48,30.04) node [anchor=north west][inner sep=0.75pt]    {$v_{1}$};
\draw (431.36,90.79) node [anchor=north west][inner sep=0.75pt]    {$v_{2}$};
\draw (404.36,137.29) node [anchor=north west][inner sep=0.75pt]    {$v_{i_0}$};
 \node[circle,fill=myblue,inner sep=0pt,minimum size=2pt,label=north east:{$\tc{myblue}{y^{v_1}_k} $}] (a) at (278.23,230.9) {}; 
 \node[circle,fill=mygreen,inner sep=0pt,minimum size=2pt,label=left:{$\tc{mygreen}{y^{v_2}_k} $}] (a) at (380.73,136.9) {};
 \node[circle,fill=myred,inner sep=0pt,minimum size=2pt,label=east:{$\tc{myred}{y^{v_{i_0}}_k}$}] (a) at (428.73,130.9) {};

\end{tikzpicture}
\caption{We suppose that $v_1$ is a bubble from the thick part such that $\vec \Psi^{v_1}$ immerses a catenoid. Doing a bubble descent, we find bubbles $v_2,\ldots, v_i$ such that the final bubble $\vec \Psi^{v_i}$ is another catenoid and is immersed close to $\vec \Psi^{v_1}$ in the sense of \eqref{yk scale for bubble descent}. A suitable inversion shows that this configuration has at least $12\pi$ Willmore energy, a contradiction to \eqref{energy threshold}.}
\end{figure}
\subsection{Finding two spheres} \label{subsec:The two spheres}
The following proposition allows us to exhaust all the Willmore energy in two round spheres after a suitable inversion. 
\begin{prop} \label{prop:Two spheres}
There are suitable inversions $I_{p_k}$, $p_k \in \R^3\setminus \vec \Phi_k(\Sigma)$ such that the graph $\hat{G} = (\hat V, \hat E)$ associated to $\hat{\vec{\Phi}}_k \cqq I_{p_k} \circ \vec \Phi_k$ contains two round spheres, i.e., there are $S_1, S_2 \in \hat{V}$ such that $\hat{\vec{\Psi}}^{S_1}$ and $\hat{\vec{\Psi}}^{S_2}$ are both \ref{S}. Furthermore, we can ensure that $S_1 \in V_{\mathrm{conc}}$.
\end{prop}
\begin{proof}
    
Recall that $\hat{V}_{\mathrm{thick}}$. Let us fix 
\begin{equation}
(\omega,1) \cqq \sigma^1 \in V_{\mathrm{thick}}.\label{omega 1 definition}
\end{equation}
Because we excluded \ref{C}, \ref{IC1}, and \ref{IC2} in \zcref{lemma:no catenoids}, $\vec \Psi^{(\omega,1)}$ immerses either \ref{P} or \ref{S}. Furthermore, it holds $\vec \Psi^{(\omega,1)}(z^{(\omega,1)}) = 0$ by definition of $\vec \Psi^{(\omega,1)}_k$. We choose a sequence $\tilde p_k \in \R^3$ such that $\limk \tilde p_k = 0$ and $p_k \cqq s^{(\omega,1)}_k \tilde p_k + y^{(\omega,1)}_k \not \in \vec \Phi_k(\Sigma)$. We define $\hat{\vec{\Phi}}_k \cqq I_{p_k} \circ \vec \Phi_k$. By an abuse of notation, we drop all hats in the notation, such that $\vec{\Phi}_k$ now denotes the inverted immersion. As in \zcref{def:Graph structure}, this sequence of immersions yields a graph structure of bubbles.

\eqref{inversion on bubbles 2} ensures that $\vec \Psi^{(\omega,1)}$ immerses \ref{P}. Notice also that the end of $\vec \Psi^{(\omega,1)}$ is located around $z^{(\omega,1)}\in R^{(\omega,1)} \subset (\omega,1)$. This also implies that all points in $Q^{(\omega,1)}$ are branch points, which is the main reason why we inverted the immersions.  
Doing a bubble ascent starting at $z^{(\omega,1)}$ in $(\omega,1)$ gives us a path $v_2, \ldots, v_{i_0} \in V^{(\omega,1), z^{(\omega,1)}} _{\text{conc}} \subset V$ and we set
\begin{equation}
S_1 \cqq v_{i_0} \in V^{(\omega,1), z^{(\omega,1)}} _{\text{conc}}.\label{definition S1}
\end{equation}
\zcref{subsec:Bubble descent and ascent} shows that $\vec \Psi^{S_1}$ immerses a closed surface and is thus of type \ref{S} or \ref{IC2}.

We will now do a bubble descent in $V'$ as in \zcref{subsec:Bubble descent and ascent} to find a second sphere. As $\# Q^{(\omega,1)}\geq 3$, we pick $p_1 \in Q^{(\omega,1)}$ around which $\vec \Psi^{(\omega,1)}$ has a branch point. $p_1$ exists by the discussion above. We find $a_1,\ldots, a_{i_0} \in V'$ connected by edges $e_1,\ldots, e_{i_0-1}$ such that $q_1(e_1) = p_1$. Then we deduce that $\vec \Psi^{a_{i_0}}$ immerses a catenoid and $a_{i_0} \in V_{\mathrm{thin}}$ by the discussion in \zcref{subsec:Bubble descent and ascent} and \zcref{lemma:no catenoids}. We now start a bubble ascent in $V'$ at the unique point in $Q^{a_{i_0}} \setminus \left \{q_2(e_{i_0-1})\right \}$ (as $Q^{a_{i_0}} = \left \{0,\infty\right \}$ consists of two elements). We find bubbles $a_{i_0}, \, a_{i_0+1}, \ldots, a_{i_1} \in V'$ connected by edges $e_{i_0}, \ldots, e_{i_1-1}$. $a_{i_1}$ is such that no ends occur around the points in $Q^{a_{i_1}}$. 
\begin{enumerate}[wide=0pt,leftmargin=*, labelsep=0.5em,label = Case \arabic*:]
\item \label{Case 1 two spheres part} We suppose that the bubbles $a_1,\ldots, a_{i_1}$ are pairwise different. In this case, $a_1 \neq a_{i_1}$. We define 
\begin{equation}
(\omega,2) \cqq a_{i_1} \in V'.\label{omega 2 definition}
\end{equation}
If $\vec \Psi^{(\omega,2)}$ has an end around $r\in R^{(\omega,2)}$, then a bubble ascent in $V$ starting at $r$ yields some $S_2 \in V^{(\omega,2), r}_{\text{conc}}$ such that $\vec \Psi^{S_2}$ is closed, i.e., without ends, and thus either \ref{S} or \ref{IC2} holds. From \eqref{energy threshold}, it follows that both $\vec \Psi^{S_1}$ and $\vec \Psi^{S_2}$ are of type \ref{S}. It also follows that $\vec \Psi^{(\omega,2)}$ is \ref{P}, which is a consequence of \zcref{lemma:no catenoids}. In the case that there is no $r\in R^{(\omega,2)}$ such that $\vec \Psi^{(\omega,2)}$ has an end around $r$, we deduce that $\vec \Psi^{(\omega,2)}$ is closed and we set $S_2 \cqq (\omega,2)$. It follows once more that $\vec \Psi^{S_2}$ is of type \ref{S}.
\item \label{Case 2 in sphere prop} This case will lead to a contradiction. We suppose that some bubble repeats, i.e., there is some $i_2 \in \left \{i_0+1, \ldots, i_1\right \}$ for which there is $i_{3} \in \left \{1,\ldots, i_0-1\right \}$ such that $a_{i_2} = a_{i_3}$ (the bubbles in $\left \{a_1,\ldots,a_{i_0}\right \}$ and in $\left \{a_{i_0},\ldots, a_{i_1}\right \}$ are pairwise disjoint by the properties of bubble descents and bubble ascents). We also may choose $i_2$ to be minimal with this property. In this case, $q_1(e_{i_3}) \neq q_2(e_{i_2-1})$, as otherwise $e_{i_2-1} = e_{i_3}$ (up to a change of orientation) and $q_2(e_{i_3}) = q_1(e_{i_2-1})$, which was excluded by minimality and the assumption that $q_1(e_{i_0})\neq q_2(e_{i_0-1})$. Furthermore, $\vec \Psi^{a_{i_3}} = \vec \Psi^{a_{i_2}}$ has branch points at $q_1(e_{i_3})$ and $q_2(e_{i_2-1})$ by construction. Finally, \eqref{yk scale for bubble descent} and \eqref{yk scale for bubble ascent} yield
\[\vec \Psi^{a_{i_3}}(q_1(e_{i_3})) = \limk (s^{a_{i_3}}_k)^{-1}(y^{a_{i_0}}_k - y^{a_{i_3}}_{k}) =\limk (s^{a_{i_2}}_k)^{-1}(y^{a_{i_0}}_k - y^{a_{i_2}}_{k})= \vec \Psi^{a_{i_3}}(q_2(e_{i_2-1})).\]
This implies that $\vec \Psi^{a_{i_3}}$ is non-injective and hence of type \ref{IC1} or \ref{IC2}. It follows
\begin{equation}
\mc W (\vec \Psi^{a_{i_3}}) \geq 4\pi.\label{half inverted catenoid has willmore energy}
\end{equation}
We now repeat the same bubble descent and bubble ascent in $V'$ around a point $p_2 \in Q^{(\omega,1)}\setminus \left \{p_1\right \}$. $\vec \Psi^{(\omega,1)}$ has a branch point of order 1 here. The corresponding bubbles are denoted by $b_j$, $j\in \left \{1,\ldots, j_0, j_0+1,\ldots, j_1\right \}$. If a bubble repeats among the $b_j$, following the same argumentation shows that there is some $j_3 \in \left \{1,\ldots, j_0-1\right \}$ such that
\[\mc W(\vec \Psi^{b_{j_3}}) \geq 4\pi.\]
But because $\vec \Psi^{S_1}$ contains at least $4\pi$ Willmore energy, it must hold $b_{j_3} = a_{i_3}$. However, we deduce from \eqref{yk scale for bubble descent} that
\begin{align}
\vec \Psi^{(\omega,1)}(p_1) &= \vec \Psi^{(\omega,1)}(q_1(e_1)) = \limk (s^{(\omega,1)}_k)^{-1}(y^{a_{i_2}}_k - y^{(\omega,1)}_k) =\limk (s^{(\omega,1)}_k)^{-1}(y^{b_{j_2}}_k - y^{(\omega,1)}_k)\notag\\
& = \vec \Psi^{(\omega,1)}(p_2).\label{omega 1 immersion has multiplicity 2 point}
\end{align}
However, as $\vec \Psi^{(\omega,1)}$ is injective (it immerses \ref{P}), this is a contradiction. This means that the bubbles $b_j$ are pairwise different and we can apply Case 1 to the sequence $b_j$ to find another sphere. Together with \eqref{half inverted catenoid has willmore energy}, we exceed \eqref{energy threshold}. In total, what we have shown is that Case 2 cannot occur and the conclusion from Case 1 holds.
\end{enumerate}

\end{proof}
\begin{remark}\label{rem:no ends in Rv apart from omega i}
From now on up to \zcref{rem:Forcing two sphere bubbles}, we will work with the inverted immersions $\hat{\vec{\Phi}}_k$ and denote these as $\vec \Phi_k$ by an abuse of notation.

We see that all the Willmore energy is concentrated in the bubbles $S_1, S_2$. In particular, \ref{IC1} and \ref{IC2} do not occur as bubbles thanks to our specific choice of inversion. Furthermore, if $v \in V' \setminus \left \{(\omega,1),(\omega,2)\right \}$, $\vec \Psi^v$ cannot have ends around points in $R^v$ because a bubble ascent around such an end would end in a closed surface with at least $4\pi$ Willmore energy. In other words, if $v\in V' \setminus \left \{(\omega,1),(\omega,2)\right \}$ and $\vec \Psi^v$ has an end, it is located around a point in $Q^v$.
\end{remark}
\subsection{Number of catenoids}\label{subsec:Number of catenoids}
\begin{defi}
We denote by $C\subset V$ the set of bubbles $v$ such that $\vec \Psi^v$ immerses \ref{C}.
\end{defi}
\begin{prop}\label{lemma:number of catenoids}
There are $p+1$ catenoids, i.e., $\# C = p+1$, where $p$ denotes the genus of $\Sigma$. Furthermore,
\begin{equation}
V = \{S_1, S_2\} \cup \{(\omega,1),(\omega,2)\} \cup V_{\mathrm{thick}} \cup C, \label{all bubbles}
\end{equation}
where $S_1, \, S_2, \, (\omega,1),$ and $(\omega,2)$ were defined in \zcref{subsec:The two spheres}. It holds $V_{\text{conc}}\subset \{S_1,S_2\}$ and $C \subset V_{\mathrm{thin}}$.
\end{prop}
\begin{proof} Denote by $C' = C \cap V'$. We consider the sum
\[S \cqq \sum _{v\in V'} \#\left \{q \in Q^v, \, \text{$\vec \Psi^v$ has an end around $q$}\right \}.\]
Each $v \in \left \{(\omega,1),(\omega,2)\right \}\subset V'$ contributes 0 to the sum by construction. If $v\in V' \setminus (\left \{(\omega,1),(\omega,2)\right \} \cup C')$, then $\vec \Psi^v$ immerses \ref{P} and $v$ contributes 1 to $S$. Finally, each $v\in C'$ contributes 2 to $S$. In total, we get
\[S = 0 \cdot 2 + 1\cdot (\# V' - \# C' - 2) + 2 \cdot \# C' = \# V' + \# C' -2.\]
Since for each edge $e\in E'$ connecting $v_1$ to $v_2$, either $\vec \Psi^{v_1}$ has an end around $q_1(e)$ and $\vec \Psi^{v_2}$ has a branch point around $q_2(e)$ or the other way around by \zcref{lem: ends and branch points}, it holds $S = \# E'$. We deduce
\[\# C' =\# E' - \# V' + 2 = p + 1,\]
where we used $\# E' - \# V' + 1 = \mc N - \mc M +1 =  p$ by \eqref{sum of genera}.

It holds $C' \subset V_{\mathrm{thin}}$ by \zcref{lemma:no catenoids}. Since $\limk \mc E(\vec \Phi_k) = (p+3)8 \pi$ and each round sphere $\vec \Psi^{S_1},\,\vec \Psi^{S_1}$ and each catenoid $\vec \Psi^{c},\, c\in C'$ adds $8\pi$ Dirichlet energy, the entire Dirichlet energy is exhausted by the bubbles $S_1,\, S_2$, and $C'$. Since any bubble in $V_{\mathrm{thin}}$ and $V_{\mathrm{conc}}$ adds $\eps$ Dirichlet energy in the sense of \eqref{at least eps energy in each bubble} and \cite[(A.6)]{ScharrerWestEnergyQuantization}, this implies \eqref{all bubbles} and $V_{\mathrm{conc}} \subset \{S_1,S_2\}$. It follows that $C = C'$.
\end{proof}
\begin{remark}\label{rem:all the bubbles}
The union \eqref{all bubbles} is not disjoint. We already know that $(\omega,1) \in V_{\mathrm{thick}}$, and it will follow from \zcref{thm:Double tree structure} that $(\omega,2) \in V_{\mathrm{thick}}$ so that $C = V_{\mathrm{thin}}$. Furthermore, in \zcref{rem:Forcing two sphere bubbles}, we will choose a suitable inversion to ensure that $S_2 \in V_{\mathrm{conc}}$. This means that the bubbles from $V_{\mathrm{thick}}$ are planes \ref{P}, the ones from $V_{\mathrm{thin}}$ are catenoids \ref{C}, and the ones from $V_{\mathrm{conc}}$ are spheres \ref{S}.
\end{remark}

\subsection{Double tree structure}\label{subsec:Double tree structure}
\begin{defi}\label{def:Two subsets V Ti}
Let $V_{T_i}\subset V'$, $i\in \left \{1,2\right \}$, be the set of bubbles in $V'$ that descend from $(\omega,i)$ in the sense that $v \in V_{T_i}$ if and only if there is a bubble descent in $V'$ of the form $v_1,\ldots, v_{i_0}$ starting at $v_1 = (\omega,i)$ such that $v = v_j$ for some $j \in \left \{1,\ldots, i_0\right \}$. 
\end{defi}
\begin{prop}
It holds $C = V_{T_1}\cap V_{T_2}$ and $ V_{T_1}\cup V_{T_2}= V'$. Furthermore, considering the induced subgraph $T_i \cqq G'[V_{T_i}]$ in $G'$, $T_i$ has the structure of a rooted tree with root $(\omega,i)$ and whose set of leaves is $C$.
\end{prop}
\begin{proof}
Any bubble ascent in $V'$ has to end in either $(\omega,1)$ or $(\omega,2)$ by \zcref{rem:no ends in Rv apart from omega i}. As each bubble ascent in $V'$ ending on $(\omega,i)$ corresponds to a (subpath of a) bubble descent starting in $(\omega,i)$, we see that $V_{T_1}\cup V_{T_2} = V'$. If $c \in C$, we can do two bubble ascents starting at $0 \in Q^c$ or $\infty \in Q^c$. They cannot both end at the same bubble, since in the proof of \zcref{prop:Two spheres}, it was shown that Case \hyperref[Case 2 in sphere prop]{2} cannot occur, implying that $C \subset V_{T_1} \cap V_{T_2}$. If $v \in V' \setminus C$, $v$ has at most one end, implying that the bubble ascent starting at $v$ is unique. This shows that $C = V_{T_1} \cap V_{T_2}$. If $T_i$ does not have a tree structure, there is a cycle in $T_i$. Any such cycle has to contain a catenoid, hence some $c\in C$, see \zcref{subsec:one catenoid}. As above, any $c\in C$ connects to one bubble in $V_{T_1}$ and one in $V_{T_2}$, so the cycle cannot be contained in just one of the $V_{T_i}$. Finally, as catenoids occur precisely as the last bubbles in a bubble descent, it is clear that $C\subset V_{T_i}$ is the set of leaves.
\end{proof}
We may choose the orientation of $T_i$ such that the edges point away from the root $(\omega,i)$. With this notation, a bubble descent is a directed path in $T_i$ to a leaf and a bubble ascent is a directed path to the root. We will show that $T_1$ and $T_2$ are \emph{isomorphic} as rooted trees with fixed leaves.

\begin{lemma}\label{thm:Double tree structure}
$T_1$ and $T_2$ are isomorphic as rooted trees with fixed leaves. More precisely, there is a graph isomorphism $\gamma\colon V_{T_1} \to V_{T_2}$ satisfying $\gamma(c)=c$ for all $c\in C$ and $\gamma((\omega,1)) =(\omega,2)$. Furthermore, for all $v\in T_1$, the bubbles on $v$ and $\gamma(v)$ live on the same scale and position, i.e.,
\begin{equation}
\limk \frac{s^{v}_k}{s^{\gamma(v)}_k}\in (0,\infty)\quad\text{and}\quad \limk (s^{\gamma(v)}_k)^{-1}|y^{\gamma(v)}_k - y^v_k| < \infty.\label{same scale for corresponding bubbles in tree}
\end{equation}
\end{lemma}
\begin{proof}
If $c\in C\subset V_{T_1}$, we define $\gamma(c) \cqq c\in V_{T_2}$. Fix now $v \in V_{T_1} \setminus C$. Denote by $D_v \subset C$ the leaf set of $v$ in $T_1$, i.e., the leaves which are descendants of $v$ in $T_1$. Notice that $\# D_v \geq 2$ as $v$ has at least two children in $T_1$ by \eqref{number of Q in v} and the fact that $v\in V_{\mathrm{thick}}$ with $\genus(v) = 0$. We define $\gamma(v) \in V_{T_2}$ to be the first common ascendant of the leaves $D_v$ inside $T_2$. 

There are at least two points $p_1, \, p_2 \in Q^v$ (respectively $q_1, \, q_2 \in Q^{\gamma(v)})$ around which $\vec \Psi^v$ (respectively $\vec \Psi^{\gamma(v)}$) have branch points. By the choice of $\gamma(v)$, we may choose $p_1,\, p_2, \, q_1, \, q_2$ in such a way that there are $c_1, \, c_2 \in D_v$, $c_1\neq c_2$, such that $c_i$ results from a bubble descent starting at $p_i$ in $T_1$ and at $q_i$ in $T_2$, $i=1,2$, see also \zcref{fig:graph structure in contradiction}. \zcref{lem: ends and branch points} and \eqref{yk scale for bubble descent} now implies
\begin{align}
\vec \Psi^{v}(p_1)- \vec \Psi^{v}(p_2) &= \limk (s^{v}_k)^{-1}(y^{c_1}_k - y^{v}_k)-(s^{v}_k)^{-1}(y^{c_2}_k - y^{v}_k) = \limk (s^{v}_k)^{-1}(y^{c_1}_k - y^{c_2}_k),\\
\vec \Psi^{\gamma(v)}(q_1)-\vec \Psi^{\gamma(v)}(q_2) &= \limk (s^{\gamma(v)}_k)^{-1}(y^{c_1}_k - y^{\gamma(v)}_k)- (s^{\gamma(v)}_k)^{-1}(y^{c_2}_k - y^{\gamma(v)}_k)=\limk (s^{\gamma(v)}_k)^{-1}(y^{c_1}_k - y^{c_2}_k).
\end{align}
The two terms on the left-hand side are nonzero by the injectivity of $\vec \Psi^v$ and $\vec \Psi^{\gamma(v)}$, see  \zcref{lem:one catenoid}\ref{rem:all types of bubbles} and \zcref{rem:no ends in Rv apart from omega i}, and do not depend on $k$. Hence
\begin{equation}
\limk \frac{s^v_k}{s^{\gamma(v)}_k} = \frac{|\vec \Psi^{\gamma(v)}(p_1) - \vec \Psi^{\gamma(v)}(p_2)|}{|\vec \Psi^v(p_1) - \vec \Psi^v(p_2)|} \in (0,\infty).\label{comparison between scales}
\end{equation}
Furthermore,
\begin{equation}
\left (\limk \frac{s^{v}_k}{s^{\gamma(v)}_k}\right )\vec \Psi^{v}(p_1) - \vec \Psi^{\gamma(v)}(q_1) = \limk (s^{\gamma(v)}_k)^{-1}(y^{\gamma(v)}_k - y^v_k). \label{comparison between the yk}
\end{equation}
\eqref{comparison between scales} and \eqref{comparison between the yk} are \eqref{same scale for corresponding bubbles in tree}.

Fix again $c_1 \in D_v$ with $p_1$ and $ q_1$ as before. Suppose now that $D_{\gamma(v)}\neq D_{v}$. Since $D_{v}\subset D_{\gamma(v)}$, there is $c_3 \in D_{\gamma(v)}\setminus D_v$. Denote by $v' \in T_1$ the first common ascendant of $v$ and $c_3$ in $T_1$ as in \zcref{fig:graph structure in contradiction}. The bubble $v$ results from a bubble descent starting at $r_1 \in Q^{v'}$ and $c_3$ comes from a bubble descent starting at $r_2 \in Q^{v'} \setminus \left \{r_1\right \}$. $r_2 \neq r_1$ follows from the minimality of $v'$. Finally, assume that $c_3$ comes from a bubble descent in $T_2$ starting at $q_3 \in Q^{\gamma(v)}$. 
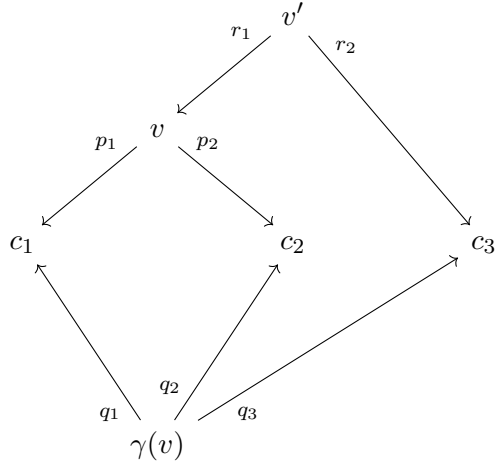
\begin{figure}[H] 
\centering 
\begin{tikzcd}[row sep=2.7em]
    &                                                                                                   & v' \arrow[ld, "r_1"' very near start] \arrow[rrdd, "r_2" very near start] &  &     \\
    & v \arrow[rd, "p_2" very near start] \arrow[ld, "p_1"' very near start]                                           &                                                            &  &     \\
c_1 &                                                                                                   & c_2                                                        &  & c_3 \\
    &                                                                                                   &                                                            &  &     \\
    & \gamma(v) \arrow[luu, "q_1" very near start] \arrow[ruu, "q_2" very near start] \arrow[rrruu, "q_3"' very near start] &                                                            &  &    
\end{tikzcd}
\caption{The structure used in the proof.}\label{fig:graph structure in contradiction}
\end{figure}
Another application of \zcref{lem: ends and branch points} yields
\begin{align}
\vec \Psi^{v'}(r_1) &= \limk (s^{v'}_k)^{-1}(y^{c_1}_k - y^{v'}_k),\label{graph chain 1}\\
\vec \Psi^{\gamma(v)}(q_1) &= \limk (s^{\gamma(v)}_k)^{-1}(y^{c_1}_k - y^{\gamma(v)}_k),\label{graph chain 2}\\
\vec \Psi^{\gamma(v)}(q_3) &= \limk (s^{\gamma(v)}_k)^{-1}(y^{c_3}_k - y^{\gamma(v)}_k),\label{graph chain 3}\\
\vec \Psi^{v'}(r_2) &= \limk (s^{v'}_k)^{-1}(y^{c_3}_k - y^{v'}_k).\label{graph chain 4}
\end{align}
In particular, applying $\limk\frac{s^{v'}_k}{s^v_k} = \infty$ from \zcref{lem: ends and branch points}, \eqref{same scale for corresponding bubbles in tree}, \eqref{graph chain 2}, and \eqref{graph chain 3} yields
\begin{align}
0 &= \limk (s^{v'}_k)^{-1}(y^{c_1}_k - y^{\gamma(v)}_k),\label{graph chain 5} \\
0 &= \limk (s^{v'}_k)^{-1}(y^{c_3}_k - y^{\gamma(v)}_k).\label{graph chain 6}
\end{align}
Combining \eqref{graph chain 1}, \eqref{graph chain 5}, \eqref{graph chain 6}, and \eqref{graph chain 4} shows
\[\vec \Psi^{v'}(r_1) = \vec \Psi^{v'}(r_2).\]
This contradicts the injectivity of $\vec \Psi^{v'}$ as $r_1\neq r_2$. It follows that $D_{\gamma(v)} = D_{v}$. We have shown that there exists $\gamma\colon T_1\to T_2$ with
\begin{equation}
    D_{\gamma(v)} = D_{v}\quad\text{for all $v\in T_1$}.\label{eq:equal descendants of v and gammav}
\end{equation}
Repeating the argument for $v\in T_2$ (the fact that $(\omega,1) \in V_{\mathrm{thick}}$ was nowhere used), we find that
\[H \cqq \left \{D_v, \, v\in T_1\right \} = \left \{D_w,\, w\in T_2\right \}.\]
$H$ has the structure of a rooted tree given by its Hasse diagram with respect to set inclusion. Recall that each vertex $v\in T_i\setminus C$ has at least two children. It follows from \cite[Theorem 3.5.2]{SempleSteel} that there is a rooted tree isomorphism between $T_i$ and $H$, sending $v$ to $D_v$. 
In particular, using \eqref{eq:equal descendants of v and gammav}, we obtain that $\gamma$ is an isomorphism.
\end{proof}

\begin{remark} \label{rem:Forcing two sphere bubbles}
Owing to \eqref{same scale for corresponding bubbles in tree}, the convergences \eqref{limit immersion on nodal part} and \eqref{bubble convergence in neck regions} still hold if we assume that $s^{v}_k = s^{\gamma(v)}_k$ and $y^{v}_k = y^{\gamma(v)}_k$. This proves \eqref{same scale and position for two bubbles}. Notice that due to the isomorphism $\gamma$ from \zcref{thm:Double tree structure}, $\#Q^{(\omega,2)} = \#Q^{(\omega,1)} \geq 3$ and it follows $(\omega,2) \in V_{\mathrm{thick}}$. 

\end{remark}

We chose the inversion in \zcref{prop:Two spheres} in such a way that $S_1 \in V_{\mathrm{conc}}$. However, we left open the question whether $S_2 = (\omega,2)$ or $S_2 \in V_{\mathrm{conc}}$. We will do suitable inversions to make sure that $S_2 \in V_{\mathrm{conc}}$ and that \eqref{same scale and position for two spheres} holds. 

\begin{prop} \label{lem:Forching two sphere bubbles}
There exist Möbius transformations $\Xi_k$ such that the graph $\hat{G}$ associated to $\hat{\vec{\Phi}}_k \cqq \Xi_k \circ \vec \Phi_k$ via \zcref{def:Graph structure} (where $\vec \Phi_k$ is as chosen in \zcref{rem:no ends in Rv apart from omega i}) satisfies $\{S_1, S_2\} = V_{\mathrm{conc}}$ and \eqref{same scale and position for two spheres} holds. Additionally, $\vec \Psi^{S_1}(\infty) = \vec \Psi^{S_2}(\infty)$.
\end{prop}
\begin{proof}
We will first invert $\vec \Phi_k$ in such a way that we can ensure $S_2 \in V_{\mathrm{conc}}$ and then invert once more to guarantee \eqref{same scale and position for two spheres}. Quantities from the first inversion are labeled with $\tilde{}$, quantities from the second inversion are labeled with $\hat{}$.

From \zcref{lem: ends and branch points}, \eqref{same scale and position for two bubbles}, and \zcref{thm:Double tree structure}, we deduce
\begin{equation}
\left \{\vec \Psi^{(\omega,1)}(q), \, q\in Q^{(\omega,1)}\right \} = \left \{\vec \Psi^{(\omega,2)}(q), \, q\in Q^{(\omega,2)}\right \},\label{macroscopic bubbles intersect}
\end{equation}
and both of these finite sets have more than one element as $\vec \Psi^{(\omega,1)}$ and $\vec \Psi^{(\omega,2)}$ are injective. Since they either immerse \ref{P} and \ref{S} or they both immerse \ref{P}, this implies that the images of $\vec \Psi^{(\omega,1)}$ and $\vec \Psi^{(\omega,2)}$ intersect infinitely often. We pick $P_0 \in \im \vec \Psi^{(\omega,1)} \cap \im \vec \Psi^{(\omega,2)} \setminus \left \{\vec \Psi^{(\omega,1)}(q), \, q\in Q^{(\omega,1)}\right \}$. We choose $P_k'\in \R^3$ satisfying $\limk P_k' = P_0$ and $P_k \cqq s^{(\omega,1)}_k P_k' + y^{(\omega,1)}_k \not \in \vec \Phi_k(\Sigma)$. 

In this way, \eqref{inversion on bubbles 2} ensures that the bubbles $\tilvec{\Psi}^{(\omega,1)}$ and $\tilvec{\Psi}^{(\omega,2)}$ associated to $\tilvec{\Phi}_k \cqq I_{P_k} \circ \vec \Phi_k$ are \ref{P} with their end located around a point in $\tilde{R}^{(\omega,1)}$ and $\tilde{R}^{(\omega,2)}$. On the scale $s^{(\omega,1)}_k$, the remaining bubbles in $V' \setminus \{(\omega,1), (\omega,2)\}$ concentrate around $\left \{\vec \Psi^{(\omega,1)}(q), \, q\in Q^{(\omega,1)}\right \}$, i.e., far away from the inversion point, and thus we have $(s^v_k)^{-1}(P_k - y^v_k) \to \infty$ as $k\to \infty$ for all $v \in V' \setminus \{(\omega,1), (\omega,2)\}$. In particular, \eqref{inversion on bubbles 1} shows that $\tilvec{\Psi}^v$ is obtained from $\vec{\Psi}^v$ by composition with isometries for all $v \in V' \setminus \{(\omega,1), (\omega,2)\}$. In total, this first inversion guarantees that $\{\tilde{S}_1,\tilde{S}_2\} = \tilde{V}_{\text{conc}}$.

Composing $\tilvec{\Phi}_k$ with one more inversion, we can additionally make sure that the scales of $\tilde S_1$ and $\tilde S_2$ become comparable in the sense of \eqref{same scale and position for two spheres}: Suppose without loss of generality that after passing to a subsequence
 \begin{equation}
 \limk \frac{s^{\tilde S_1}_k}{s^{\tilde S_2}_k} = 0.\label{choosing the smaller spherical bubble}
 \end{equation}
After a potential rescaling and translation of $\tilvec{\Phi}_k$, we may assume that $(s^{\tilde S_1}_k, y^{\tilde S_1}_k)=(1,0)$. By \cite[Lemma 5.13]{RiviereLectureNotes}, we can pick $Q_0 \in \R^3$ and $0<r_0<1-|Q_0|$ such that 
\begin{equation}
\tilvec{\Phi}(\Sigma)  \cap B_{r_0}(Q_0) = \emptyset \quad\text{for all $k \in \N$.}\label{p0 away from surface}
\end{equation}
We define $ \Xi_k \cqq  I_{Q_0} \circ I_{P_k}$ and let $\hat{\vec{\Phi}}_k \cqq  \Xi_k \circ \vec \Phi_k$.

Once again, on the scale $s^{\tilde S_1}_k$, the bubbles in $\tilde V' \cup \{\tilde S_1\}$ live away from the inversion center $Q_0$ in the sense of \eqref{inversion on bubbles 1} and \eqref{p0 away from surface} and are not altered by the inversion (it holds $\hat{\vec{\Psi}}^{\tilde{S}_1} = I_{Q_0} \circ \tilde{\vec{\Psi}}^{\hat{S}_1}$, but both still immerse \ref{S}). In particular, there is still a bubble $\hat{S}_2 \in \hat{V}_{\text{conc}}^{(\hat{\omega},2)}$ as $\hat{\vec{\Psi}}^{(\hat{\omega},2)}$ agrees with $\tilde{\vec{\Psi}}^{(\hat{\omega},2)}$ up to an isometry, implying that the end is located around a point in $\hat{R}^{(\hat \omega, 2)}$. However, the domain corresponding to $\tilde{S}_2$ and $\hat{S}_2$ may be different. By \eqref{p0 away from surface}, it holds
\begin{equation}
\hat{\vec{\Phi}}_k(\Sigma) =   I_{Q_0} \circ \tilde{\vec{\Phi}}_k(\Sigma)   \subset B_{r_0^{-1} }(0),\label{image bounded after inversion}
\end{equation}
meaning that $\mc A(\hat{\vec{\Phi}}_k)$ is uniformly bounded from above, see \cite[Lemma 1.1]{Simon}. It follows that $ s^{\hat S_1}_k =   1$ is the largest scale in $\hat{G}$ up to multiplication with a bounded factor.

We will now explicitly find the domain on which the second spherical bubble $\hat{S}_2$ is immersed. Let $e= ((\tilde{\omega},2), \tilde{S}_2)$ be the edge connecting $(\tilde{\omega},2)$ to $ \tilde{S}_2$ in $\tilde{G}$ and let $\Omega^{\alpha_0,\alpha_0}_k(e) = B_{\alpha_0 \rho_k^{-1}}(0) \setminus B_{\alpha_0^{-1}}(0)$ and $\tilde{\vec{\Theta}}_k(e)$ be the corresponding neck region and reparametrization of $\tilvec{\Phi}_k$ as defined in \eqref{neck region Omega k (e)} and \eqref{Theta k parametrization in neck region}, \eqref{Theta k (e) parametrization}. Here, $\rho_k = \rho^1_k$ is the radius from \eqref{Theta k parametrization in neck region}. Denote by $\lambda_{\tilvec{\Theta}_k(e)}$ the conformal factor of $\tilvec{\Theta}_k(e)$. As in the proof of \zcref{lem: ends and branch points}, we use the notation $\underline \lambda(r) = \dashint_{\partial B_r}\lambda(x) \dif l (x)$ to denote the average on circles. From the definition of $\tilvec{\Theta}_k(e)$, \eqref{bounded conformal factors in nodal region 2}, the fact that we work in converging, conformal charts, and \eqref{same scale and position for two bubbles}, we deduce for some $C$ independent of $k$ that
\begin{equation}
    \alpha_0 \rho_k^{-1}e^{\underline{\lambda}_{\tilvec{\Theta}_k(e)}(\alpha_0 \rho_k^{-1})} \leq  C s^{(\tilde{\omega},2)}_k = Cs^{(\tilde{\omega},1)}_k.\label{scale on one end small}
\end{equation}
Similarly, it holds $ \alpha_0 ^{-1} e^{\underline{\lambda}_{\tilvec{\Theta}_k(e)}(\alpha_0^{-1}  )} \geq  C^{-1} s^{\tilde{S}_2}_k$. From \zcref{lem: ends and branch points} and \eqref{choosing the smaller spherical bubble}, it follows
\begin{equation}
    \limk   s^{(\tilde{\omega},1)}_k =\limk \frac{s^{(\tilde{\omega},1)}_k}{s^{\tilde{S}_1}_k}  = 0 = \limk \frac{s^{\tilde{S}_1}_k}{s^{\tilde{S}_2}_k}=\limk\frac{1}{s^{\tilde{S}_2}_k},\label{scale quotients for S1 and S2}
\end{equation}
so by continuity, there are radii $r_k$ for $k$ sufficiently large such that 
\begin{equation}
r_k e^{\underline{\lambda}_{\tilvec{\Theta}_k(e)}(r_k )} = 1.\label{choosing intermediate scale}
\end{equation}
The local Harnack estimates \cite[Lemma V.2]{BernardRiviere}, \eqref{scale on one end small}, \eqref{scale quotients for S1 and S2}, and \eqref{choosing intermediate scale} imply $\limk \frac{r_k}{\alpha_0\rho_k^{-1}} = \limk \frac{\alpha_0^{-1}}{r_k} = 0$. As a consequence of \zcref{lem: ends and branch points} and $s^{(\tilde{\omega},1)}_k = s^{(\tilde{\omega},2)}_k$, on the scale $1=s ^{\tilde{S}_1}_k$, all bubbles in $\tilde{V}'$, including $(\tilde\omega, 2)$, concentrate around $\tilvec{\Psi}^{\tilde{S_1}}(\infty) \in \R^3$. Since ${\partial B_{\alpha_0\rho_k^{-1}}}$ lies in the domain corresponding to the bubble of $(\tilde{\omega},2)$, we see that
\begin{equation}
\limk \| \tilvec{\Theta}_k(e) - \tilvec{\Psi}^{\tilde{S}_1}(\infty)\|_{L^{\infty}(\partial B_{\alpha_0\rho_k^{-1}})} =0.\label{Theta k remains bounded on one end}
\end{equation}

 For $x_k \in \Omega^{\alpha_0, \alpha_0}_k(e)$ such that $|x_k| = r_k$, we deduce
\begin{align}
|\tilvec{\Theta}_k(e)(x_k)| &\leq  \left |\tilvec{\Theta}_k(e)(x_k) - \tilvec{\Theta}_k(e)\left (\alpha_0\rho_k^{-1} \frac{x_k}{r_k}\right )\right | + \underbrace{\left |  \tilvec{\Theta}_k(e)\left (\alpha_0\rho_k^{-1} \frac{x_k}{r_k}\right )- \tilvec{\Psi}^{\tilde{S}_1}(\infty)\right |}_{\to 0 \text{ as $k\to \infty$ by \eqref{Theta k remains bounded on one end}}}+\underbrace{| \tilvec{\Psi}^{\tilde{S}_1}(\infty)|}_{< \infty} \\
&\leq C + C r_k\int ^{\alpha_0 \rho_k^{-1} r_k^{-1}} _{1} e^{\lambda_{\tilvec{\Theta}_k(e)}(tx_k)} \dif t.
\end{align} 
From \zcref{lem:almost pointwise control conformal factor} and the fact that $m=-1$ in this neck region\footnote{We know that $m=\pm 1$ from \zcref{prop: no density 2} and \zcref{lem: ends and branch points} yields the correct sign.}, we see
\begin{equation}
\lambda_{\tilvec{\Theta}_k(e)}(tx_k) \leq |\lambda_{\tilvec{\Theta}_k(e)}(tx_k)-\lambda_{\tilvec{\Theta}_k(e)}(x_k) + 2 \log (t)|+\lambda_{\tilvec{\Theta}_k(e)}(x_k) - 2 \log (t) \leq  \lambda_{\tilvec{\Theta}_k(e)}(x_k) -\frac{5}{3} \log(t) + C.\label{Lemma 2 1 application}
\end{equation}
With \eqref{Lemma 2 1 application} and \eqref{choosing intermediate scale},
\[C r_k\int ^{\alpha_0 \rho_k^{-1} r_k^{-1}} _{1} e^{\lambda_{\tilvec{\Theta}_k(e)}(tx_k)} \dif t \leq C r_k\int_1^\infty e^{\lambda_{\tilvec{\Theta}_k(e)}(x_k) } t^{-\frac{5}{3}} \dif t \leq C.\]
Hence, $|\tilvec{\Theta}_k(e)(x_k) - Q_0| \in (r_0, C)$. Setting $\hat{\vec{\Theta}}_k(e) = I_{Q_0} \circ\tilde{\vec{\Theta}}_k(e)$, we see from \eqref{choosing intermediate scale} that $r_k e^{\underline{\lambda}_{\hat{\vec{\Theta}}_k(e)}(r_k)}$ is uniformly bounded from below and above. In particular, the conformal factor of the map $x\mapsto \hat{\vec{\Theta}}_k(e)(r_k x)$ is bounded from below on compact subsets of $\C \setminus \{0\}$. As we had seen earlier, $\limk s^{\hat{S}_2}_k < \infty$. This excludes the possibility that $\limk s^{\hat{S}_2}_k = 0$, as otherwise the conformal factor of $x\mapsto \hat{\vec{\Theta}}_k(e)(r_k x)$ would diverge to $-\infty$ locally uniformly on $\C \setminus \{0\}$. In particular, we may choose $\hat{\vec{\Psi}}^{\hat{S}_2}_k(x) = \hat{\vec{\Theta}}_k(e)(r_k x)$, so that $s^{\hat{S}_2}_k =s^{\hat{S}_1}_k= 1$ and $y^{\hat{S}_1}_k =  y^{\hat{S}_2}_k = 0$, which proves \eqref{same scale and position for two spheres}. Using the properties of bubble descents, this also implies $\vec \Psi^{S_1}(\infty) = \vec \Psi^{S_2}(\infty)$. 
\end{proof}

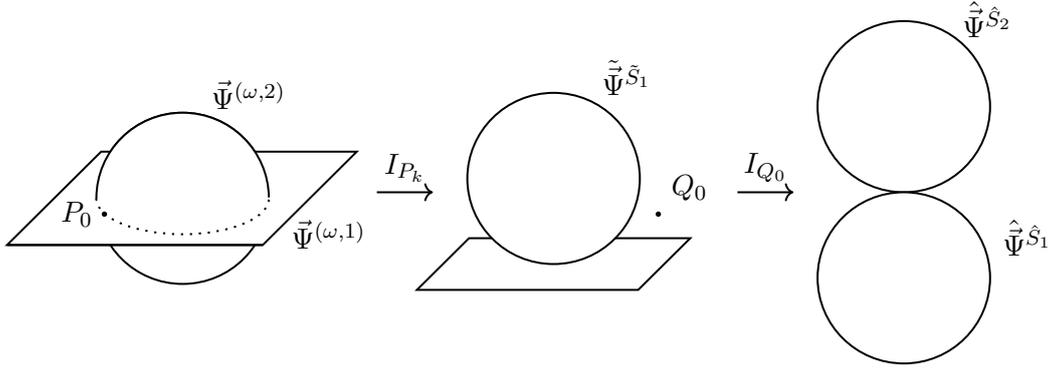
\begin{figure}[H] 
\centering 
\tikzset{every picture/.style={line width=0.75pt}} 

\begin{tikzpicture}[x=0.75pt,y=0.75pt,yscale=-1,xscale=1]

\draw  [fill={rgb, 255:red, 255; green, 255; blue, 255 }  ,fill opacity=1 ] (130,163.06) .. controls (130,139.28) and (149.28,120) .. (173.06,120) .. controls (196.83,120) and (216.11,139.28) .. (216.11,163.06) .. controls (216.11,186.83) and (196.83,206.11) .. (173.06,206.11) .. controls (149.28,206.11) and (130,186.83) .. (130,163.06) -- cycle ;
\draw  [fill={rgb, 255:red, 255; green, 255; blue, 255 }  ,fill opacity=1 ] (132.39,139.58) -- (260,139.58) -- (213.06,186.53) -- (85.44,186.53) -- cycle ;

\draw[->]    (270,160) --node[above]{$I_{P_k}$} (298,160) ;

\draw   (316,183.11) -- (426.45,183.11) -- (400.45,209.11) -- (290,209.11) -- cycle ;
\draw  [fill={rgb, 255:red, 255; green, 255; blue, 255 }  ,fill opacity=1 ] (315.17,153.06) .. controls (315.17,129.28) and (334.45,110) .. (358.23,110) .. controls (382.01,110) and (401.28,129.28) .. (401.28,153.06) .. controls (401.28,176.83) and (382.01,196.11) .. (358.23,196.11) .. controls (334.45,196.11) and (315.17,176.83) .. (315.17,153.06) -- cycle ;


 \node[circle,fill=black,inner sep=0pt,minimum size=2pt,label=north east:{$Q_0$}] (a) at (410.5, 171) {}; 

\draw[->]    (450,160) -- node[above]{$I_{Q_0}$} (478,160) ;
\draw  [fill={rgb, 255:red, 255; green, 255; blue, 255 }  ,fill opacity=1 ] (490,116.94) .. controls (490,93.17) and (509.28,73.89) .. (533.06,73.89) .. controls (556.83,73.89) and (576.11,93.17) .. (576.11,116.94) .. controls (576.11,140.72) and (556.83,160) .. (533.06,160) .. controls (509.28,160) and (490,140.72) .. (490,116.94) -- cycle ;
\draw  [fill={rgb, 255:red, 255; green, 255; blue, 255 }  ,fill opacity=1 ] (490,203.06) .. controls (490,179.28) and (509.28,160) .. (533.06,160) .. controls (556.83,160) and (576.11,179.28) .. (576.11,203.06) .. controls (576.11,226.83) and (556.83,246.11) .. (533.06,246.11) .. controls (509.28,246.11) and (490,226.83) .. (490,203.06) -- cycle ;
\draw  [draw opacity=0][dash pattern={on 0.84pt off 2.51pt}] (216.08,163.68) .. controls (215.3,173.36) and (196.33,181.11) .. (173.06,181.11) .. controls (149.28,181.11) and (130.01,173.03) .. (130.01,163.06) .. controls (130.01,163.05) and (130.01,163.04) .. (130.01,163.03) -- (173.06,163.06) -- cycle ; \draw  [dash pattern={on 0.84pt off 2.51pt}] (216.08,163.68) .. controls (215.3,173.36) and (196.33,181.11) .. (173.06,181.11) .. controls (149.28,181.11) and (130.01,173.03) .. (130.01,163.06) .. controls (130.01,163.05) and (130.01,163.04) .. (130.01,163.03) ;  
\draw [fill={rgb, 255:red, 255; green, 255; blue, 255 }  ,fill opacity=1 ]   (85.44,186.53) -- (213.06,186.53) ;
\draw  [draw opacity=0][fill={rgb, 255:red, 255; green, 255; blue, 255 }  ,fill opacity=1 ] (130,163.6) .. controls (130,163.42) and (130,163.23) .. (130,163.05) .. controls (130,139.27) and (149.27,120) .. (173.05,120) .. controls (196.68,120) and (215.86,139.04) .. (216.1,162.61) -- (173.05,163.05) -- cycle ; \draw   (130,163.6) .. controls (130,163.42) and (130,163.23) .. (130,163.05) .. controls (130,139.27) and (149.27,120) .. (173.05,120) .. controls (196.68,120) and (215.86,139.04) .. (216.1,162.61) ;  

\node[
  circle,  fill=black,  inner sep=0pt,  minimum size=2pt,  label={[xshift=1pt ] west:{$P_0$}}] (a) at (134,171) {};

\draw (186,102.4) node [anchor=north west][inner sep=0.75pt]    {$\vec \Psi^{(\omega, 2)}$};
\draw (226,172.4) node [anchor=north west][inner sep=0.75pt]    {$\vec \Psi^{(\omega, 1)}$};
\draw (381,92.4) node [anchor=north west][inner sep=0.75pt]    {$ \tilvec{\Psi}^{\tilde{S}_1}$};
\draw (581,172.4) node [anchor=north west][inner sep=0.75pt]    {$\hat{\vec{\Psi}}^{\hat{S}_1}$};
\draw (561,62.4) node [anchor=north west][inner sep=0.75pt]    {$\hat{\vec{\Psi}}^{\hat{S}_2}$};

\end{tikzpicture}
\caption{The first inversion ensures that both spherical bubbles are in $V_{\mathrm{conc}}$. The second inversion ensures that they are on the same scale.}
\end{figure}

\subsection{Strong \texorpdfstring{$W^{2,2}_{\loc}$}{W^(2,2)}-convergence}\label{subsec:strong convergence}
Let $v\in V$ and let $z\in v \setminus (Q^v \cup R^v)$ if $v \in V_{\mathrm{thick}}$ or $z \in \C \setminus (Q^v \cup R^v)$ if $v \in V_{\mathrm{thin}} \cup V_{\mathrm{conc}}$ (or $p=1$). There is a neighborhood $U$ around $z$ and conformal, converging coordinates $\omega, \omega_k\colon\mb D \to U$ such that $\vec \Psi^v_k \circ \omega_k \wto \vec \Psi^v \circ \omega$ in $W^{2,2}(\mb D; \R^3)$. Denote by $\lambda_k$, $\lambda$ the conformal factors of $ {\vec{\Psi}}^v_k \circ \omega_k$, $ {\vec{\Psi}}^v \circ \omega$ in $\mb D$. Because of the lower semicontinuity of the Willmore energy, we deduce
\begin{equation}
\limk \frac{1}{4}\|e^{-\lambda_k} \Delta (\vec \Psi^v_k \circ \omega_k)\|^2_{L^2(\mb D; \R^3)} = \limk \mc W( {\vec{\Psi}}^v_k \circ \omega_k \vert _{\mb D}) = \mc W( {\vec{\Psi}}^v \circ \omega \vert _{\mb D}) = \frac{1}{4} \|e^{-\lambda} \Delta (\vec \Psi^v \circ \omega)\|^2_{L^2(\mb D; \R^3)}.\label{no Willmore energy in V after inversion}
\end{equation}
Indeed, otherwise $\liminf_{k\to \infty} \mc W(\vec \Phi_k) > 8\pi$. The uniform Harnack estimate 
\begin{equation}
\|\lambda\|_{L^\infty(\mb D)} + \sup_{k\in \N}\|\lambda_k\|_{L^\infty(\mb D)} <\infty\label{uniform harnack estimate}
\end{equation}
 holds. The weak $W^{2,2}$-convergence $\vec \Psi^v_k \circ \omega_k \wto \vec \Psi^v \circ \omega$ implies $e^{\lambda_k} \to e^{\lambda}$ pointwise a.e.\ in $\mb D$ after passing to a subsequence. \eqref{uniform harnack estimate} also shows $e^{-\lambda_k} \to e^{-\lambda}$ pointwise a.e.\ and so $e^{-\lambda_k}\Delta (\vec \Psi^v_k \circ \omega_k) \wto e^{-\lambda} \Delta (\vec \Psi^v \circ \omega) $ in $L^2(\mb D; \R^3)$. Together with \eqref{no Willmore energy in V after inversion}, we see
\begin{equation}
e^{-\lambda_k}\Delta (\vec \Psi^v_k \circ \omega_k) \to e^{-\lambda} \Delta (\vec \Psi^v \circ \omega) \quad \text{in $L^2(\mb D; \R^3)$.}\label{strong convergence H}
\end{equation}
Using again \eqref{uniform harnack estimate}, we also deduce $\Delta (\vec \Psi^v_k \circ \omega_k) \to  \Delta (\vec \Psi^v \circ \omega) $ in $L^2(\mb D; \R^3)$. Applying standard elliptic regularity estimates, e.g. \cite[Section 6.3.1]{Evans}, yields that
\[\vec \Psi^v_k \circ \omega_k \to \vec \Psi^v \circ \omega \quad\text{in $W^{2,2}(B_{1/2}; \R^3)$}.\]
This implies the local \emph{strong} $W^{2,2}$-convergence in \eqref{limit immersion on nodal part}, \eqref{bubble convergence in neck regions}, and \eqref{concentration point convergence}. The varifold convergence will follow from \eqref{varifold convergence result}. This completes the proof of \zcref{thm:Asymptotic convergence} if the genus $p$ is at least two.

\subsection{The torus case}\label{subsec:torus case}
Let us now prove \zcref{thm:Asymptotic convergence} in the case $p=1$.

As in \zcref{subsec:one catenoid}, the cyclicity of $G$ ensures that there is some bubble $v = i \in V'$ such that $\vec \Psi^v$ immerses a catenoid. $\vec \Psi^v$ has two ends of order $-1$ around $0$ and $\infty$, which implies by \zcref{lem: ends and branch points} that $m(e)=1$, where $e=(i-1,i)$ (or $e=(N-1,0)$ in the case that $i=0$). It is not difficult to verify, cf. \cite[Section 3.2]{LaurainRiviereEnergyQuantization}, that we can choose $z_k \in  B_{\alpha_0 a^{i}_k} \setminus B_{\alpha_0^{-1} b^{i}_k}$ such that
\[\limk \frac{z_k}{b^i_k} = \limk \frac{a^i_k}{z_k} = \infty\]
and
\begin{equation}
\lim _{\alpha\to 0} \limk \int _{B_{\alpha^{-1} |z_k|} \setminus B_{\alpha |z_k|}} |\nabla \vec n_{\vec \Phi_k\circ \chi_k}|^2 \dif x = 0.\label{zk with no energy around it}
\end{equation}
In particular, using \eqref{L2 weak estimate for conformal factor in torus case}, we have local Harnack estimates for the conformal factor as in \cite[Theorem~5.5]{RiviereLectureNotes} and after passing to a subsequence, we deduce
\[e^{-\lambda_k(z_k)} |z_k|^{-1} (\vec \Phi_k\circ \chi_k (|z_k| \cdot ) - \vec \Phi_k \circ \chi_k (z_k)) \wto \vec \Psi_{\infty} \quad\text{in $W^{2,2}_{\loc}(\C \setminus \left \{0\right \};\R^3)$}\]
and \eqref{zk with no energy around it} ensures that $\vec \Psi_{\infty}$ immerses a flat plane of some density. \eqref{weakened version of pointwise gradient estimate} yields that $\vec \Psi_{\infty}$ is a density 1 plane. We choose $\tilde p_k \in \R^3$ such that $\limk \tilde p_k = 0$ and $p_k \cqq e^{\lambda_k(z_k)} |z_k| \tilde p_k +  \vec \Phi_k\circ \chi_k(z_k) \not \in \vec \Phi_k(\mb T^2)$. We work with the inverted immersions $I_{p_k } \circ \vec \Phi_k$. While this changes the location (and possibly the number) of the bubbles obtained in \zcref{prop:Thin part bubble neck decomposition}, what we have gained is that \eqref{inversion on bubbles 2} guarantees now that there is some bubble $(\omega,1) \in V'$ for which $\vec \Psi ^{(\omega,1)}$ immerses \ref{P} \emph{and} the end is located around a point in $R^{(\omega,1)}$. From here, we can essentially repeat the arguments in \zcref{subsec:The two spheres}, \zcref{subsec:Number of catenoids}, \zcref{thm:Double tree structure} and \zcref{rem:Forcing two sphere bubbles}. In total, $G$ has the form
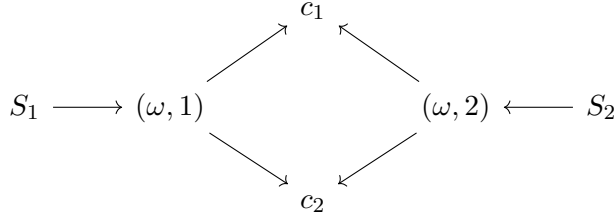
\begin{figure}[H] 
\centering 
\begin{tikzcd}
              &                                & c_1 &                                &               \\
S_1 \arrow[r] & (\omega,1) \arrow[ru] \arrow[rd] &     & (\omega,2) \arrow[lu] \arrow[ld] & S_2 \arrow[l] \\
              &                                & c_2 &                                &              
\end{tikzcd}
\caption{The graph structure $G$. $S_1$ and $S_2$ are the two round spheres in $V_{\mathrm{conc}}$. $c_1$ and $c_2$ are the two bubbles of type \ref{C}. $(\omega,1)$ and $(\omega,2)$ are of type \ref{P}. No further bubbles can appear as all the Dirichlet energy is already exhausted.}\label{fig:G for p 1}
\end{figure}
This finishes the proof of \zcref{thm:Asymptotic convergence}.

\section{The choice of inversions}\label{sec:Inversions}
Suppose that $\vec \Phi_k' \in \mc E_{\Sigma}$ satisfy the assumptions of \zcref{thm:Asymptotic convergence}. Let $\Xi_k$ be the associated Möbius transformations from \zcref{thm:Asymptotic convergence} and set $\vec \Phi_k \cqq \Xi_k \circ \vec \Phi_k'$. Let $T=(V_T,E_T)$ with root $\omega$ be the associated rooted tree such that the graph $G' = (V',E')$ from Definitions \ref{def:Graph structure} and \ref{def:Graph structure for p equals 1} can be identified with $(T + T)/\sim$, where $\sim$ identifies the leaves $C $. Let $p^{\text{inv}}_k \in \R^3 \setminus \im \vec \Phi_k$ and consider $\hat{\vec{\Phi}}_k \cqq I_{p^{\text{inv}}_k} \circ \vec \Phi_k$. In order to distinguish between the objects corresponding to the sequences $\vec \Phi_k$ and $\hat{\vec{\Phi}}_k$, all quantities coming from Definitions \ref{def:Graph structure} and \ref{def:Graph structure for p equals 1} applied to $\hat{\vec{\Phi}}_k$ are denoted by a hat. Suppose for now that $p\geq 2$. Notice that we may identify $V'$ with a subset of $\hat{V'}$. Indeed, either $v\in V_{\mathrm{thick}}$ which is invariant under inversions, or $v\in C= V_{\mathrm{thin}}$ and $\vec \Psi^v$ immerses a catenoid \ref{C}. Then $\hat{\vec{\Psi}}^v$ is either \ref{C}, \ref{IC1}, or \ref{IC2}, implying that some Dirichlet energy is concentrated on this scale and $v \in \hat{V}_{\text{thin}}$. Also notice that $\hat{Q}^v = Q^v$ for all $v\in V'$.

\begin{prop}\label{prop:Choice of inversions}
Suppose $p\geq 2$. After passing to a subsequence, the graph structure $\hat{G}$ corresponding to $\hat{\vec{\Phi}}_k$ is of one of the following four forms:
\begin{enumerate}[wide=0pt,leftmargin=*, labelsep=0.5em,label = Type \arabic*:]
\item 
\begin{equation}
\hat V_{\mathrm{thick}} = V_{\mathrm{thick}}, \quad \hat{V}_{\text{thin}} = V_{\mathrm{thin}},\quad \hat V_{\mathrm{conc}} = \emptyset.\label{set of bubbles 1}
\end{equation}
There is $\hat{\omega} \in \hat{V}_{\text{thin}}$ such that $\hat{\vec{\Psi}}^{\hat{\omega}}$ is of type \ref{IC2}, while for $v \in \hat{V}_{\text{thin}}\setminus \{\hat{\omega}\}$, $\hat{\vec{\Psi}}^v$ is of type \ref{C}. $\hat{G}' = \hat{G}$ is isomorphic to $ (\hat{T} + \hat{T})/\sim$, where $\hat{T}=T$ as trees but with the root at $\hat{\omega}$ and the induced orientation of the edges pointing away from $\hat{\omega}$. $\sim$ identifies the leaves and the root of $\hat{T}$.
\item \begin{equation}
\hat V_{\mathrm{thick}} = V_{\mathrm{thick}}, \quad \hat{V}_{\text{thin}} = V_{\mathrm{thin}},\quad \hat V_{\mathrm{conc}} = \{\hat{S}\}.\label{set of bubbles 2}
\end{equation}
Analogous to Type 1, with the only change that $\hat{\vec{\Psi}}^{\hat{\omega}}$ is of type \ref{IC1} and there exists $ \hat{S}\in \hat V_{\mathrm{conc}}^{\hat{\omega}}$ of type \ref{S}.
\item \begin{equation}
\hat V_{\mathrm{thick}} = V_{\mathrm{thick}}, \quad \hat{V}_{\text{thin}} = V_{\mathrm{thin}},\quad \hat V_{\mathrm{conc}} \subset \{\hat{S}_1,\hat{S}_2\}.\label{set of bubbles 3}
\end{equation}
For $v\in \hat{V}_{\text{thin}}$, $\hat{\vec{\Psi}}^v$ is of type \ref{C}. There is $\hat{\omega} \in V_{T} \setminus C$ such that $\hat{G}'$ is isomorphic to $(\hat{T} + \hat{T})/\sim$, where $\hat{T}=T$ as trees but with the root at $\hat{\omega}$ and the induced orientation of the edges. For $i\in \{1,2\}$, $\hat{\vec{\Psi}}^{(\hat{\omega},i)}$ is either of type \ref{S} or of type \ref{P} and there exist $S_i \in V_{\mathrm{conc}}^{(\hat{\omega},i)}$ of type \ref{S}, $i\in \{1,2\}$.
\item
\begin{equation}
\hat V_{\mathrm{thick}} = V_{\mathrm{thick}}, \quad \hat{V}_{\text{thin}} = V_{\mathrm{thin}}\cup \{(\hat{\omega},1), (\hat{\omega},2)\},\quad \hat V_{\mathrm{conc}} \subset \{\hat{S}_1,\hat{S}_2\}.\label{set of bubbles 4}
\end{equation}
$\hat{G}'$ is isomorphic to $(\hat{T} + \hat{T})/\sim$, where $\hat{T}$ is obtained from $T$ by replacing an edge $e = (v_1, v_2) \in E_{T}$ by the two edges $(v_1,\hat{\omega})$ and $(\hat{\omega},v_2)$, treating $\hat{\omega}$ as the root of $\hat{T}$ and changing the orientation of the edges such that they point away from the root $\hat{\omega}$. For $v\in \hat{V}_{\text{thin}} \setminus \{(\hat{\omega},1), (\hat{\omega},2)\}$, $\hat{\vec{\Psi}}^v$ is of type \ref{C}. For $i\in \{1,2\}$, either $\hat{\vec{\Psi}}^{(\hat{\omega},i)}$ is of type \ref{S} or of type \ref{P} and there exists $S_i \in V_{\mathrm{conc}}^{(\hat{\omega},i)}$ of type \ref{S}.
\end{enumerate}
In all cases, \eqref{same scale and position for two bubbles} and \eqref{scale equation in main theorem} still hold.
\end{prop}
\begin{proof}
 Let $p^v_k \cqq (s^v_k)^{-1}(p^{\text{inv}}_k - y^v_k)$ for $v\in V$. We will do a case analysis corresponding to the different types.
\begin{enumerate}[wide=0pt,leftmargin=*, labelsep=0.5em,label = Type \arabic*:]
\item If there is $\hat{\omega}\in C$ such that $p^{\hat{\omega}}_k \to p^{\hat{\omega}} \in \R^3$ and $p^{\hat{\omega}} \not \in \im \vec \Psi^{\hat{\omega}}$, then $\hat{\vec{\Psi}}^{\hat{\omega}}$ is of type \ref{IC2} by \eqref{inversion on bubbles 2}. 

In particular, if $v\in V_{\mathrm{thin}} \setminus \{\hat{\omega}\}$, $\hat{\vec{\Psi}}^v$ is still \ref{C} and if $v\in V' \setminus C$, $\hat{\vec{\Psi}}^v$ is \ref{P} because the Willmore energy is exhausted. The inverted catenoid \ref{IC2} and the $p$ catenoids \ref{C} already exhaust the Dirichlet energy, implying that no additional bubbles develop so that \eqref{set of bubbles 1} holds.

As undirected graphs, $\hat{G} = G = (T+T)/\sim$. As $\hat{\vec{\Psi}}^{\hat{\omega}}$ has no ends, we may apply the same arguments as in \zcref{subsec:The two spheres} and \zcref{subsec:Double tree structure} to see that $\hat{G}' = \hat{G} = (\hat{T} + \hat{T})/\sim$, where $\hat{T}$ is the tree $T$ with the root $\hat{\omega}$ and the induced orientation of the edges pointing away from $\hat{\omega}$. With this orientation, \eqref{same scale and position for two bubbles} and \eqref{scale equation in main theorem} still hold.

\item If there is $\hat{\omega}\in C$ such that $p^{\hat{\omega}}_k \to p^{\hat{\omega}}  \in \R^3$ and $p^{\hat{\omega}}  \in \im \vec \Psi^{\hat{\omega}}$, then $\hat{\vec{\Psi}}^{\hat{\omega}}$ is of type \ref{IC1} by \eqref{inversion on bubbles 2}. Furthermore, $\hat{\vec{\Psi}}^{\hat{\omega}}$ has an end around $(\vec \Psi^{\hat{\omega}})^{-1}(p^{\hat{\omega}})$. Thus, there is $\hat S \in V_{\mathrm{conc}} ^{\hat{\omega}, (\vec \Psi^{\hat{\omega}})^{-1}(p^{\hat{\omega}})}$ such that $\hat{\vec{\Psi}}^{\hat{\omega}}$ is \ref{S}. $\hat{G}'$ is as in Case 1, and $\hat{G}$ is obtained by gluing $\hat S$ to $\hat{\omega}$.

\item If there is $\hat{\omega}\in V_{\mathrm{thick}}$ such that $p^{\hat{\omega}}_k \to p \in \R^3$ and $p \not \in \vec \Psi^{\hat{\omega}}(Q^{\hat{\omega}})$, we can view $\hat{\omega}$ as $(\omega,1)$ as chosen in \zcref{subsec:The two spheres}. In particular, apart from changing the root and the conclusion in \zcref{rem:Forcing two sphere bubbles}, the argumentation is as before.
\item The remaining case is that for all $v\in V'$, it holds either that $p^v_k \to \infty$ or $p^v_k \to p$ and $p \in \vec \Psi^v(Q^v)$, where we set $\vec \Psi^{v}(q) = \infty$ if $\vec \Psi^{v}$ has an end around $q$. In particular, for all $v\in V'$, $\hat{\vec{\Psi}}^v$ immerses \ref{P} if $v\in V' \setminus C$ and \ref{C} if $v\in C$ by \zcref{subsec:Inversions}. If $p^v_k \to \infty$ for all $v\in V'$, then neither the limiting immersions $\vec \Psi^v$ nor the structure of $G$ change apart from a reflection by \zcref{subsec:Inversions}, which is covered by Type 3. \eqref{inversion on bubbles 2} shows that if $p^v_k \to  \vec{\Psi}^v(q)$ for some $q\in Q^v$, then the end of $\hat{\vec{\Psi}}^v$ forms around $q$.

Assume now that at least for one bubble $v\in V'$, $p^v_k$ remains bounded. Inductively choose a maximal sequence $v_1, \ldots, v_{i_0} \in V'$ such that 
\begin{enumerate}[label = \roman*)]
\item For $i\in \{1,\ldots, i_0-1\}$, $v_i$ and $v_{i+1}$ are connected by an edge $e_i = (v_i, v_{i+1}) \in E$,
\item For $i \in \{1,\ldots, i_0-1\}$, $\vec \Psi^{v_i}$ has a branch point at $q_1(e_i)$, whereas $\hat{\vec{\Psi}}^{v_i}$ has an end at $q_1(e_i)$\footnote{Notice that $q_1(e_1) \in Q^{v_i}$ still corresponds to a point in $\hat{Q}^{v_i}$, where we view $V'$ again as a subset of $\hat{V}'$.}. 
\end{enumerate}
Notice that the edges $e_i \in E$ need not correspond to edges $e_i \in \hat{E}$. The bubbles $v_i$ are in descending order with respect to the ordering in $G$, so a maximal sequence is well-defined. $\vec \Psi^{v_{i_0}}$ has an end at $q_2(e_{i_0-1})$. If $\hat{\vec{\Psi}}^{v_{i_0}}$ has a branch point at $q_2(e_{i_0-1})$, then $p^{v_{i_0}}_k$ remains bounded and by assumption, there is some $q' \in Q^{v_{i_0}}\setminus \{q_2(e_{i_0}-1)\}$ such that $\vec \Psi^{v_{i_0}}$ has a branch point at $q'$ while $\hat{\vec{\Psi}}^{v_{i_0}}$ has an end at $q'$. This however contradicts the maximality of the sequence $v_{1},\ldots, v_{i_0}$. So $\hat{\vec{\Psi}}^{v_{i_0}}$ has an end around $q_2(e_{i_0-1})$. \zcref{lem: ends and branch points} shows that $v_{i_0-1}$ and $v_{i_0}$ cannot be connected directly by an edge $\hat{e}$ in $\hat{G}$ with endpoints $q_1(e_{i_0-1})$ and $q_2(e_{i_0-1})$. So there must be $\hat{w}_1,\ldots, \hat{w}_{i_1} \in \hat{V}_{\text{thin}}$ such that for $i\in \{0,\ldots, i_1\}$, $\hat{w}_i$ and $\hat{w}_{i+1}$ are connected by an edge $\hat{e}_i \in \hat{E}$, where $\hat{w}_0 \cqq v_{i_0-1}$ and $\hat{w}_{i_1+1} \cqq v_{i_0}$.

In particular, there must be $\hat{w}_{i_2}$, $i_2\in \{1,\ldots, i_1\}$, such that $\hat{\vec{\Psi}}^{\hat{w}_{i_2}}$ has a branch point at $q_{2}(\hat{e}_{i_2-1})$ and at $q_{1}(\hat{e}_{i_2})$. Hence, either $\hat{\vec{\Psi}}^{\hat{w}_{i_2}}$ is closed or there is $r \in \hat{R}^{\hat{w}_{i_2}}$ such that $\hat{\vec{\Psi}}^{\hat{w}_{i_2}}$ has an end around $r$. In the first case, we set $\hat{S}_1 = \hat{w}_{i_2}$. In the second case, there is some $\hat S_1 \in \hat V_{\mathrm{conc}}^{\hat{w}_{i_2}, r}$ such that $\hat{\vec{\Psi}}^{\hat{w}_{i_2}}$ is closed. We treat $\hat{S}_1$ and $\hat{w}_{i_2}$ as the new $S_1$ and $(\omega,1)$ and we repeat the same arguments as in \zcref{subsec:The two spheres} and \zcref{subsec:Double tree structure}. In particular, as in \zcref{rem:all the bubbles}, the entire Dirichlet energy is exhausted by two spheres and the catenoids, implying that $i_1=1$.
\end{enumerate} 
 
\end{proof}

\begin{cor}
Suppose $p=1$. $\hat{G}$ is of one of the following forms:
\begin{enumerate}[wide=0pt,leftmargin=*, labelsep=0.5em,label = Type \arabic*:]
\item $\hat{V} = \hat{V}' = \{\hat{c}_1, \hat{c}_2\}$, $\hat{c}_1$ is \ref{IC2} and $\hat{c}_2$ is \ref{C}.
\item $\hat{V} =  \{\hat{c}_1, \hat{c}_2, \hat{S}\}$, $\hat{c}_1$ is \ref{IC1}, $\hat{c}_2$ is \ref{C} and $\hat{S} \in V_{\mathrm{conc}}^{\hat{c}_1}$ is \ref{S}.
\item \begin{equation}
\hat V' = \{(\hat{\omega},1),(\hat{\omega},2), \hat{c}_1, \hat{c}_2\},\quad \hat V_{\mathrm{conc}} \subset \{\hat{S}_1,\hat{S}_2\}.\label{set of bubbles 4 for p=1}
\end{equation}
$\hat{G}$ is as in \zcref{fig:G for p 1} with the exception that for $i\in \{1,2\}$, either $\hat{\vec{\Psi}}^{(\hat{\omega},i)}$ is of type \ref{S} or of type \ref{P} and there exists $S_i \in V_{\mathrm{conc}}^{(\hat{\omega},i)}$ of type \ref{S}.
\end{enumerate}
\end{cor}
\begin{proof}
The proof is analogous to \zcref{prop:Choice of inversions}.
\end{proof}
\begin{figure}[H]
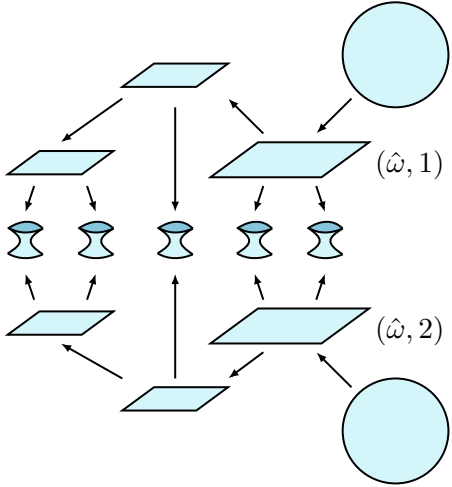
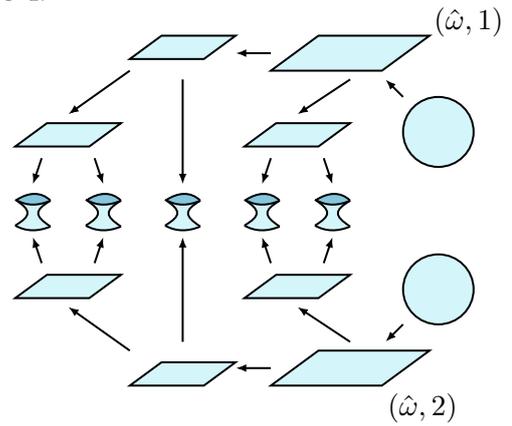

\begin{minipage}[c]{.48\textwidth}
\begin{flushleft}
Type 1: \vspace*{-0.5cm}
\end{flushleft}
  \centering   
  \tikzset{every picture/.style={line width=0.75pt}} 

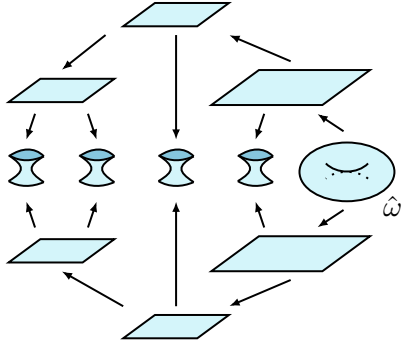
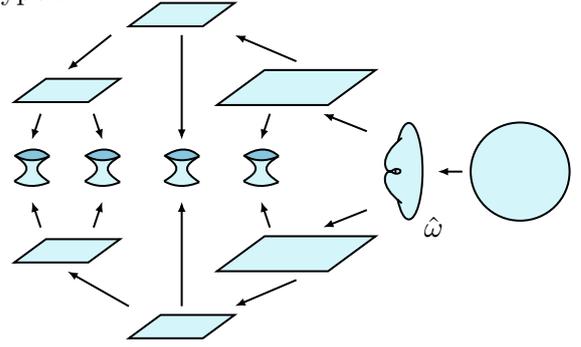
\begin{tikzpicture}[x=0.55pt,y=0.55pt,yscale=-1,xscale=1]

\draw  [fill={rgb, 255:red, 212; green, 245; blue, 250 }
  ,fill opacity=1 ] (240,221.86) .. controls (240,210.74) and (254.47,201.72) .. (272.32,201.72) .. controls (290.16,201.72) and (304.63,210.74) .. (304.63,221.86) .. controls (304.63,232.98) and (290.16,242) .. (272.32,242) .. controls (254.47,242) and (240,232.98) .. (240,221.86) -- cycle ;
\draw    (257.95,216.08) .. controls (265.25,223.52) and (279.52,223.43) .. (286.68,216.08) ;
\draw  [dash pattern={on 0.84pt off 2.51pt}]  (286.68,226.49) .. controls (279.43,220.3) and (265.06,220.21) .. (257.95,226.49) ;

\draw [-{Latex[length=4pt]}]   (108,128) -- (78,152) ;
\draw  [fill={rgb, 255:red, 212; green, 245; blue, 250 }
  ,fill opacity=1 ] (255.6,176) -- (180,176) -- (212.4,152) -- (288,152) -- cycle ;
\draw [{Latex[length=4pt]}-]   (192,128) -- (234,146) ;
\draw [-{Latex[length=4pt]}]   (156,128) -- (156,200) ;
\draw  [fill={rgb, 255:red, 212; green, 245; blue, 250 }
  ,fill opacity=1 ] (92.4,174) -- (42,174) -- (63.6,158) -- (114,158) -- cycle ;
\draw  [fill={rgb, 255:red, 212; green, 245; blue, 250 }
  ,fill opacity=1 ] (170.4,122) -- (120,122) -- (141.6,106) -- (192,106) -- cycle ;
\draw [-{Latex[length=4pt]}]   (96,182) -- (102,200) ;
\draw [-{Latex[length=4pt]}]   (60,182) -- (54,200) ;
\draw [{Latex[length=4pt]}-]   (252,182) -- (270,194) ;
\draw [-{Latex[length=4pt]}]   (216,182) -- (210,200) ;
\draw  [draw opacity=0][fill={rgb, 255:red, 134; green, 198; blue, 221 }  ,fill opacity=1 ] (42.78,210.73) .. controls (53.45,216.09) and (59.15,214.7) .. (66,211.15) .. controls (59.15,203.59) and (49.13,205.14) .. (42.78,210.73) -- cycle ;
\draw  [draw opacity=0][fill={rgb, 255:red, 212; green, 245; blue, 250 }
  ,fill opacity=1 ] (42.78,210.73) .. controls (53.45,216.09) and (53.14,222.25) .. (42,228.2) .. controls (48.05,232.12) and (60.25,232.12) .. (66,228.88) .. controls (53.76,222.25) and (56.22,215.78) .. (66,211.15) .. controls (58.84,212.54) and (55.61,216.55) .. (42.78,210.73) -- cycle ;
\draw    (42,210.47) .. controls (53.15,217.88) and (53.15,221.43) .. (42,228.2) ;
\draw    (66,228.88) .. controls (54.85,221.47) and (54.85,217.92) .. (66,211.15) ;
\draw    (42,228.2) .. controls (49.32,232.53) and (57.89,232.38) .. (66,228.88) ;
\draw    (42.78,210.73) .. controls (50.8,205.48) and (59.52,204.6) .. (66,211.15) ;
\draw    (42.78,210.73) .. controls (50.1,215.06) and (57.89,214.65) .. (66,211.15) ;

\draw  [draw opacity=0][fill={rgb, 255:red, 134; green, 198; blue, 221 }  ,fill opacity=1 ] (90.78,210.73) .. controls (101.45,216.09) and (107.15,214.7) .. (114,211.15) .. controls (107.15,203.59) and (97.13,205.14) .. (90.78,210.73) -- cycle ;
\draw  [draw opacity=0][fill={rgb, 255:red, 212; green, 245; blue, 250 }
  ,fill opacity=1 ] (90.78,210.73) .. controls (101.45,216.09) and (101.14,222.25) .. (90,228.2) .. controls (96.05,232.12) and (108.25,232.12) .. (114,228.88) .. controls (101.76,222.25) and (104.22,215.78) .. (114,211.15) .. controls (106.84,212.54) and (103.61,216.55) .. (90.78,210.73) -- cycle ;
\draw    (90,210.47) .. controls (101.15,217.88) and (101.15,221.43) .. (90,228.2) ;
\draw    (114,228.88) .. controls (102.85,221.47) and (102.85,217.92) .. (114,211.15) ;
\draw    (90,228.2) .. controls (97.32,232.53) and (105.89,232.38) .. (114,228.88) ;
\draw    (90.78,210.73) .. controls (98.8,205.48) and (107.52,204.6) .. (114,211.15) ;
\draw    (90.78,210.73) .. controls (98.1,215.06) and (105.89,214.65) .. (114,211.15) ;

\draw  [draw opacity=0][fill={rgb, 255:red, 134; green, 198; blue, 221 }  ,fill opacity=1 ] (144.78,210.73) .. controls (155.45,216.09) and (161.15,214.7) .. (168,211.15) .. controls (161.15,203.59) and (151.13,205.14) .. (144.78,210.73) -- cycle ;
\draw  [draw opacity=0][fill={rgb, 255:red, 212; green, 245; blue, 250 }
  ,fill opacity=1 ] (144.78,210.73) .. controls (155.45,216.09) and (155.14,222.25) .. (144,228.2) .. controls (150.05,232.12) and (162.25,232.12) .. (168,228.88) .. controls (155.76,222.25) and (158.22,215.78) .. (168,211.15) .. controls (160.84,212.54) and (157.61,216.55) .. (144.78,210.73) -- cycle ;
\draw    (144,210.47) .. controls (155.15,217.88) and (155.15,221.43) .. (144,228.2) ;
\draw    (168,228.88) .. controls (156.85,221.47) and (156.85,217.92) .. (168,211.15) ;
\draw    (144,228.2) .. controls (151.32,232.53) and (159.89,232.38) .. (168,228.88) ;
\draw    (144.78,210.73) .. controls (152.8,205.48) and (161.52,204.6) .. (168,211.15) ;
\draw    (144.78,210.73) .. controls (152.1,215.06) and (159.89,214.65) .. (168,211.15) ;

\draw  [draw opacity=0][fill={rgb, 255:red, 134; green, 198; blue, 221 }  ,fill opacity=1 ] (198.78,210.73) .. controls (209.45,216.09) and (215.15,214.7) .. (222,211.15) .. controls (215.15,203.59) and (205.13,205.14) .. (198.78,210.73) -- cycle ;
\draw  [draw opacity=0][fill={rgb, 255:red, 212; green, 245; blue, 250 }
  ,fill opacity=1 ] (198.78,210.73) .. controls (209.45,216.09) and (209.14,222.25) .. (198,228.2) .. controls (204.05,232.12) and (216.25,232.12) .. (222,228.88) .. controls (209.76,222.25) and (212.22,215.78) .. (222,211.15) .. controls (214.84,212.54) and (211.61,216.55) .. (198.78,210.73) -- cycle ;
\draw    (198,210.47) .. controls (209.15,217.88) and (209.15,221.43) .. (198,228.2) ;
\draw    (222,228.88) .. controls (210.85,221.47) and (210.85,217.92) .. (222,211.15) ;
\draw    (198,228.2) .. controls (205.32,232.53) and (213.89,232.38) .. (222,228.88) ;
\draw    (198.78,210.73) .. controls (206.8,205.48) and (215.52,204.6) .. (222,211.15) ;
\draw    (198.78,210.73) .. controls (206.1,215.06) and (213.89,214.65) .. (222,211.15) ;

\draw [-{Latex[length=4pt]}]   (120,314) -- (78,290) ;
\draw  [fill={rgb, 255:red, 212; green, 245; blue, 250 }
  ,fill opacity=1 ] (288,266) -- (212.4,266) -- (180,290) -- (255.6,290) -- cycle ;
\draw [{Latex[length=4pt]}-]   (192,314) -- (234,296) ;
\draw [-{Latex[length=4pt]}]   (156,314) -- (156,242) ;
\draw  [fill={rgb, 255:red, 212; green, 245; blue, 250 }
  ,fill opacity=1 ] (114,268) -- (63.6,268) -- (42,284) -- (92.4,284) -- cycle ;
\draw  [fill={rgb, 255:red, 212; green, 245; blue, 250 }
  ,fill opacity=1 ] (192,320) -- (141.6,320) -- (120,336) -- (170.4,336) -- cycle ;
\draw [-{Latex[length=4pt]}]   (96,260) -- (102,242) ;
\draw [-{Latex[length=4pt]}]   (60,260) -- (54,242) ;
\draw [{Latex[length=4pt]}-]   (252,260) -- (270,248) ;
\draw [-{Latex[length=4pt]}]   (216,260) -- (210,242) ;

\draw (295,236.4) node [anchor=north west][inner sep=0.75pt]    {$\hat{\omega}$};

\end{tikzpicture}  
\end{minipage}%
\hfill 
\begin{minipage}[c]{0.48\textwidth}  
\begin{flushleft}
Type 2: \vspace*{-0.5cm}
\end{flushleft}
  \centering  
  \input{Semi_inverted_catenoid_graph.tex}
\end{minipage}
\par
\begin{minipage}[c]{.48\textwidth}  
\begin{flushleft}
Type 3: \vspace*{-0.5cm}
\end{flushleft}
  \centering  
    \input{root_changed_graph.tex}
\end{minipage}%
\hfill 
\begin{minipage}[c]{0.48\textwidth}  
\begin{flushleft}
Type 4: \vspace*{-0.5cm}
\end{flushleft}
  \centering  
  \input{vertex_added_graph.tex}
\end{minipage}  
\caption{Depicted are the graphs $\hat{G}$ associated to the inversions $\hat{\vec{\Phi}}_k$ for the different types if $G$ has the form from \zcref{Caption 2}. In case 3 and 4, there is also the possibility that instead of $\hat{\vec{\Psi}}^{(\hat{\omega},i)}$ being \ref{P}, it could also be \ref{S} and no bubble from $\hat{V}_{\text{conc}}$ is attached to it, where $i\in \{1,2\}$.}
\label{fig:types}
\end{figure}

\section{Applications to minimizers of several problems}\label{sec:Applications to minimizers of several problems}
Let $\Sigma$ be a genus-$p$ surface, $p\geq 1$. In this section, $\vec \Phi_k \in \mc E_{\Sigma}$ will always be such that \eqref{bounded Willmore aaaaand area} holds.
After passing to a subsequence, we let $G = (V,E)$ be the graph from \zcref{def:Graph structure} for $p\geq 2$ and from \zcref{def:Graph structure for p equals 1} for $p=1$. The area bound \eqref{bounded Willmore aaaaand area} implies that $\sup_{k\in \N}s^v_k < \infty$ for all $v\in V$. Recall that we denote by $V_{\mathrm{macro}} \subset V$ the set of macroscopic bubbles, i.e., those that satisfy $\limk s^v_k > 0$.

By \cite[Lemma 1.1]{Simon}, we can always shift $\vec \Phi_k$ such that 
\begin{equation}
\sup_{k\in \N}\|\vec \Phi_k\|_{L^{\infty}(\Sigma)} \leq \sup_{k\in \N} C \sqrt{\mc W(\vec \Phi_k) \mc A(\vec \Phi_k)} < \infty\label{Simon diameter estimate}
\end{equation}
and in the following, we assume \eqref{Simon diameter estimate} to hold. Then for any $v\in V_{\mathrm{macro}}$, we do not need to shift and dilate the immersion, so we will assume that $(s^v_k, y^v_k) = (1,0)$ for all $k$ and $v\in V_{\mathrm{macro}}$. We will prove the continuity result \zcref{thm:bubble convergence for extrinsic quantities} for $\mc A, \, \mc V$, and $\mc M$ defined in \eqref{are, volume and total mean curvature}.

\begin{proof}[Proof of \zcref{thm:bubble convergence for extrinsic quantities}]

The convergence on the bubble domains is handled as in \cite[Lemma 3.1]{MondinoScharrer}. We need to estimate the contribution from the neck regions $\Omega_{k}^{\alpha,\beta}(e)$ for $e\in E$, i.e.,
\begin{equation}
\lim_{\alpha,\beta \to 0} \limk \mc F\left (\vec \Theta_k(e)\vert _{\Omega_k^{\alpha,\beta}(e)}\right ) = 0,\label{vanishing functional on neck region}
\end{equation}
where $\vec \Theta_k(e)$ and $\Omega_k^{\alpha,\beta}(e)$ were defined in Definitions \ref{def:Graph structure} and \ref{def:Graph structure for p equals 1}. From \zcref{lem: ends and branch points} and the fact that $\sup_{k\in \N}s^v_k < \infty$ for all $v\in V$, we know that the area in the neck regions is vanishing, i.e., \eqref{vanishing functional on neck region} for $\mc F = \mc A$. Because of \eqref{Simon diameter estimate}, $\mc V$ also vanishes in the neck region. The Cauchy--Schwarz inequality and \eqref{bounded Willmore aaaaand area} yields that $\mc M$ vanishes as well.
\end{proof}

As in \cite[Example 2.4]{RuppScharrer}, $\vec \Phi \in \mc E_\Sigma$ induces an oriented integral varifold $V^{\mathrm{o}}_{\vec \Phi}$ in $\R^3$ in the sense of \cite{Hutchinson} given by
\begin{equation}
V^{\mathrm{o}}_{\vec \Phi}(\phi) \cqq \int_ \Sigma \phi(\vec \Phi, \vec n_{\vec \Phi}) \dif \mu_{\vec \Phi}\quad \text{for all $\phi \in C_c^0(\R^3\times \mb S^2)$,}\label{oriented varifold}
\end{equation}
where we identified $\mb S^2$ with $\mb G^{\mathrm{o}}(2,3)$, the Grassmannian manifold of oriented 2-dimensional subspaces of $\R^3$.

We define for $v\in V$ and $\alpha \in (0,1)$
\begin{equation}
B(v,\alpha) \cqq \begin{cases}
v \setminus \bigcup _{q\in Q^v \cup R^v} B_{\alpha} ^{\overline{h}}(q), & \text{$v \in V_{\mathrm{thick}}$,}\\
B_{1/\alpha}(0) \setminus \bigcup _{q\in Q^v \cup R^v} B_{\alpha}(q),& \text{$v\in V_{\mathrm{thin}}\cup V_{\mathrm{conc}}$}.
\end{cases} \label{bubble domains}
\end{equation}
It holds by \zcref{thm:bubble convergence for extrinsic quantities} that
\[\lim _{\alpha \to 0} \limk \bigg [\mc A(\vec \Phi_k) - \sum_{v\in V_{\mathrm{macro}}} \mc A(\vec \Psi^v_k \vert _{B(v,\alpha)})\bigg ] = 0,\]
which implies
\begin{equation}
\lim _{\alpha \to 0} \limk \bigg |V^{\mathrm{o}}_{\vec \Phi_k}(\phi) - \sum_{v \in V_{\mathrm{macro}}}V^{\mathrm{o}}_{\vec \Psi^v_k \vert _{B(v,\alpha)}}(\phi)\bigg | =0 \quad \text{for all $\phi \in C_c^0(\R^3\times \mb S^2)$} .\label{measure vanishing outside macroscopic bubbles}
\end{equation}
Standard Sobolev embeddings and the fact that the metric $({\vec \Psi^v_k})^* g_{\R^3}$ is uniformly controlled on $B(v,\alpha)$ (see \eqref{bounded conformal factors in nodal region 2}, \eqref{conformal factor estimate thin part}) imply $(\vec \Psi^v_k,\vec n_{\vec \Psi^v_k}) \to (\vec \Psi^v,\vec n_{\vec \Psi^v})$ pointwise a.e., so that it follows
\[\limk V^{\mathrm{o}}_{\vec \Psi^v_k \vert _{B(v,\alpha)}}(\phi) = V^{\mathrm{o}}_{\vec \Psi^v \vert _{B(v,\alpha)}}(\phi) = \int _{B(v,\alpha)} \phi(\vec \Psi^v, \vec n_{\vec \Psi^v}) \dif \mu_{\vec \Psi^v }.\]
Letting $\alpha \to 0$, we deduce from \eqref{measure vanishing outside macroscopic bubbles} that
\begin{equation}
\limk V^{\mathrm{o}} _{\vec \Phi_k}(\phi) = \sum_{v \in V_{\mathrm{macro}}} V^{\mathrm{o}}_{\vec \Psi^v}(\phi) \quad\text{for all $\phi \in C_c^0(\R^3\times \mb S^2)$} . \label{varifold convergence result}
\end{equation}

Suppose $\Omega \subset \R^3$ is a set of finite perimeter in $\R^3$, i.e., it holds
\[ \int _{\Omega} \Div f \dif x = - \int _{\partial^* \Omega} \langle \nu_{\Omega}, f \rangle \dif \mc H^{2} \quad \text{for all $f \in C^1_c(\R^3; \R^3)$},\]
where $\nu_{\Omega} = \frac{D\chi_{\Omega}}{|D\chi _{\Omega}|}$ $|D\chi _{\Omega}|$-a.e.\ and $\partial^* \Omega$ is the reduced boundary. The 3-current $c(\Omega)$ induced by $\Omega$ is given by
\[c(\Omega)(\omega) = \int _{\Omega} \omega \quad \text{for all $\omega \in C_c^\infty(\R^3; \extp^3 \R^3)$}. \]
Its boundary is thus
\[(\partial c(\Omega))(\omega) = c(\Omega)(d\omega) = - \int _{\partial^* \Omega} \langle \hodge \nu_{\Omega},  \omega \rangle \dif \mc H^{2}\quad \text{for all $\omega \in C_c^\infty(\R^3; \extp^2 \R^3)$}, \]
where $\hodge$ is the Hodge star operator in $\R^3$.
\subsection{Isoperimetrically constrained Willmore minimizers}\label{subsec:iso problem}
\begin{lemma}\label{lem:diverging conformal classes isoperimetric ratio}
The conformal classes associated to the sequence $\vec \Phi^{\mc I}_{p, \sigma}$ defined in \zcref{subsec:Applications to minimizers} diverge to the boundary of the moduli space as $\sigma \to \infty$.
\end{lemma}
\begin{proof}
We assume that the conformal classes remain bounded for some $\sigma_k \to \infty$. Then the set of vertices $V$ from the graph $G=(V,E)$ from \zcref{def:Graph structure} associated to the sequence $\vec \Phi^{\mc I}_{p, \sigma_k}$ consists of one thick part $V_{\mathrm{thick}} = \{\sigma_1\}$, no thin parts $V_{\mathrm{thin}} = \emptyset$ and the concentration bubbles $V_{\mathrm{conc}}$. Suppose that there is only one $v_0\in V_{\mathrm{macro}}$. \zcref{thm:bubble convergence for extrinsic quantities} implies that $\mc V(\vec \Psi^{v_0}) = 0$. $\vec \Psi^{v_0}$ has no ends because the area $\mc A(\vec \Phi^{\mc I}_{p,\sigma})$ is uniformly bounded. This implies that $\vec \Psi^{v_0}$ cannot be embedded as otherwise, $\mc V(\vec \Psi^{v_0})$ would coincide (up to a potential sign change) with the enclosed volume, see \cite[Lemma 6.6]{RuppScharrer}. We deduce by the Li--Yau inequality
\begin{equation}
\mc W(\vec \Psi^{v_0}) \geq 8\pi.\label{at least 8pi willmore for zero volume surface}
\end{equation}
Since $\mc W(\vec \Phi^{\mc I}_{p,\sigma_k}) < 8\pi$ for all $k$, equality has to hold in \eqref{at least 8pi willmore for zero volume surface}. As in \zcref{subsec:one catenoid}, $\vec \Psi^{v_0}$ is either of type \ref{IC2} or a Möbius transformation of a meromorphic map from $\hat{\C}$ to $\C$ if $v \in V_{\mathrm{conc}}$ or from $v$ to $\C$ if $v=\sigma_1$. The first case can be excluded by the assumption $\mc V(\vec \Psi^{v_0})=0$, the second case by the argument in \zcref{prop: no density 2}. Thus, there are at least two bubbles, $v_0,\, v_1 \in V_{\mathrm{macro}}$. The $8\pi$ bound for the Willmore energy implies that $\vec \Psi^{v_i}$ is \ref{S} for $i\in \{0,1\}$. In particular, the genus is 0 and $v_i \in V_{\mathrm{conc}}$. It follows that $\mc W(\vec \Psi^{\sigma_1}) = 0$. Arguing once more as in \zcref{subsec:one catenoid}, using that $\vec\Psi^{\sigma_1}$ has two ends and that the genus of $\sigma_1$ is positive, this yields that $\vec \Psi^{\sigma_1}$ has a branch point of order 2, which contradicts \zcref{prop: no density 2}.
\end{proof}
Now we are in a position to prove \zcref{cor:Bubbles in space iso problem}.
\begin{proof}[Proof of \zcref{cor:Bubbles in space iso problem}]
We let $\vec \Phi_k = \vec \Phi^{\mc I}_{p,\sigma_k}$. By the scaling invariance, we may assume that $\mc A(\vec \Phi_k) = 1$ for all $k$. Since $\mc W(\vec \Phi_k) < 8\pi$ for all $k$, $\vec \Phi_k$ is embedded by the Li--Yau inequality and $\vec \Phi_k(\Sigma)$ encloses some (smooth) domain $\Omega_k \subset \R^3$ by \cite[Lemma 6.6]{RuppScharrer}. We also choose the orientation of $\vec \Phi_k$ in such a way that $\vec n_{\vec \Phi_k}$ coincides with the inward pointing unit vector, i.e., $\vec n_{\vec \Phi_k} \circ (\vec \Phi_k)^{-1} = \nu_{\Omega_k}$.

The 2-current $c(V^{\mathrm{o}}_{\vec \Phi_k})$ induced by the oriented varifold $V^{\mathrm{o}}_{\vec \Phi_k}$ as defined in \cite{Hutchinson} is given by
\[c(V^{\mathrm{o}}_{\vec \Phi_k})(\omega) = \int _{\Sigma} \langle \hodge \vec n_{\vec \Phi_k}, \omega \circ \vec \Phi_k \rangle \dif \mu _{\vec \Phi_k} = - \partial c(\Omega_k)(\omega)=-  c(\Omega_k)(d \omega)\quad\text{for all $\omega \in C_c^1(\R^3; \extp^2\R^3)$}.\] 
By the compactness result in $BV$ and after passing to a subsequence, $\chi_{\Omega_k} \to \chi_{\Omega}$ in $L^1(\R^3)$ for some $\Omega$ of finite perimeter in $\R^3$. Using \eqref{varifold convergence result}, it follows
\begin{equation}
\sum_{v\in V_{\mathrm{macro}}} c(V^{\mathrm{o}}_{\vec \Psi^v}) = - \partial c(\Omega).\label{macroscopic limit still boundary of set of finite perimeter}
\end{equation}
Since $|\Omega_k| = \mc V(\vec \Phi_k)\to 0$ as $k\to \infty$, we deduce that $\sum_{v\in V_{\mathrm{macro}}} c(V^{\mathrm{o}}_{\vec \Psi^v}) = 0$ and $\chi_{\Omega} = 0$. Owing to \zcref{lem:diverging conformal classes isoperimetric ratio}, we may apply \zcref{thm:Asymptotic convergence} and \zcref{prop:Choice of inversions} to see that either $V_{\mathrm{macro}}$ consists of one element (immersing an inverted catenoid \ref{IC2} or a round sphere \ref{S}) or of exactly two elements $V_{\mathrm{macro}} = \{S_1, S_2\}$ immersing round spheres \ref{S}. The first case is excluded by \eqref{macroscopic limit still boundary of set of finite perimeter} and we see that $\vec \Psi^{S_1}$ and $\vec \Psi^{S_2}$ have the same image with different orientation.
\end{proof}

\subsection{Normalized mean curvature constrained Willmore minimizers} \label{subsec:normalized mean curvature constrained problem}
\begin{lemma}\label{lem:diverging conformal classes normalized mean curvature}
The conformal classes associated to the sequence $\vec \Phi^{\mc T}_{p, \tau}$ defined in \zcref{subsec:Applications to minimizers} diverge to the boundary of the moduli space as $\tau \nearrow \sqrt{8\pi}$.
\end{lemma}
\begin{proof}
The proof is similar to \zcref{lem:diverging conformal classes isoperimetric ratio}. In order to arrive at \eqref{at least 8pi willmore for zero volume surface}, we use $\mc T(\vec \Psi^{v_0}) = \sqrt{8\pi}$ by \zcref{thm:bubble convergence for extrinsic quantities} and $\mc W \geq \mc T^2$ by \cite[(1.19)]{MasterThesis}. This implies that equality holds in \cite[(1.19)]{MasterThesis}, which is the case if and only if $H_{\vec \Psi^{v_0}} \equiv \text{const}$. If $\vec \Psi^{v_0}$ is unbranched and embedded, $\vec \Psi^{v_0}$ is a minimizer $\vec \Phi^{\mc T}_{p,\sqrt{8\pi}}$ and in particular smooth. Alexandrov's theorem \cite{Alexandrov} implies that $\vec \Psi ^{v_0}$ is a round sphere, contradicting $\mc T(\vec \Psi ^{v_0}) = \sqrt{8\pi} $. If $\vec \Psi^{v_0}$ is either branched (with true branch points) or not embedded, we finish the proof as before. 
\end{proof}
We can now prove \zcref{cor:bubbles in space T problem}.
\begin{proof}[Proof of \zcref{cor:bubbles in space T problem}]
We repeat the proof of \zcref{cor:Bubbles in space iso problem} up to \eqref{macroscopic limit still boundary of set of finite perimeter}. As before, if $V_{\mathrm{macro}}$ consists of only one element, this has to immerse a round sphere or an inverted catenoid and both of these choices contradict \zcref{thm:bubble convergence for extrinsic quantities}. So $V_{\mathrm{macro}} = \{S_1, S_2\}$ and $\vec \Psi^{S_1}, \ \vec \Psi^{S_2}$ immerse round spheres \ref{S}. If their (signed) radii are $r_1,\ r_2 \neq 0$, then \zcref{thm:bubble convergence for extrinsic quantities} implies that
\begin{equation}
\sqrt{8\pi} = \lim _{\tau {\to \sqrt{8\pi}}} \mc T(\vec \Phi ^{\mc T} _{p, \tau}) = \frac{\mc M(\vec \Psi_{S_1})+\mc M(\vec \Psi_{S_2})}{\sqrt{\mc A(\vec \Psi_{S_1})+\mc A(\vec \Psi_{S_2})}} = \frac{4\pi (r_1 + r_2)}{\sqrt{4\pi(r_1^2 +r_2^2)}}. \label{normalized mean curvature constraint in limit}
\end{equation}
This is true if and only if $r_1 = r_2>0$. So $\vec \Psi_{S_1}$ and $\vec \Psi_{S_2}$ have the same orientation and the same radii. As in \eqref{macroscopic bubbles intersect}, we know that $\vec \Psi^{S_1}(\hat{\C})\cap \vec \Psi^{S_2}(\hat{\C}) \neq \emptyset$. Finally, this intersection must consist of exactly one point as otherwise, their induced currents cannot be the boundary of a set of finite perimeter, i.e., \eqref{macroscopic limit still boundary of set of finite perimeter} cannot hold. 
\end{proof}

\subsection{Conformally constrained minimization problem}\label{subsec:Conformally constrained minimization problem}
Suppose $\Sigma$ is a genus-$p$ surface, $p\geq 1$, and $c$ is a complex structure on $\Sigma$. Let $h$ be its corresponding Poincar\'e metric of unit volume. In \cite{KuwertSchatzleConformallyConstrainedMinimization, RiviereCrelle}, it was shown that the minimization problem
\begin{equation}
\beta^{\text{conf}}_{p}(h) \cqq \inf_{\substack{\vec \Phi \in \mc E_{\Sigma}\\ \vec \Phi\text{ is conformal w.r.t. $h$}}} \mc W(\vec \Phi)\label{conformally constrained minimization problem}
\end{equation}
admits smooth minimizers $\vec \Phi^{\text{conf}}_{p,h}$ as long as $\beta^{\text{conf}}_{p}(h) < 8\pi$.
In the case $p=1$, $h$ may be described by $\omega \in \C$ as in \eqref{fundamental domain}. In \cite[Proposition D.1]{Ndiaye}, it was shown that for $\Re \omega = 0$ and $\Im \omega$ sufficiently large, it holds $\beta^{\text{conf}}_{1}(\omega) < 8\pi$. Hence, the conclusion of \zcref{thm:Asymptotic convergence} in the case $p=1$ applies to the sequence $\vec \Phi^{\text{conf}}_{p,\omega}$ as $\Re \omega=0$ and $\Im \omega \to \infty$.
\appendix

\renewcommand{\thesection}{Appendix \Alph{section}}
\tocless\section{Proof of Lemma \ref{lem:almost pointwise control conformal factor}}
\addcontentsline{toc}{section}{Appendix A. Proof of Lemma \ref{lem:almost pointwise control conformal factor}}
\renewcommand{\thesection}{\Alph{section}}
\label{sec:Proof of Lemma}

 \begin{proof}[Proof of Lemma \ref{lem:almost pointwise control conformal factor}]
Suppose without loss of generality that $\delta <1/3$. We proceed by contradiction. Suppose that for all $k\in \N$ and $\eps(k) = \alpha_0(k) = 1/k$, $Q(k) = k$, there exists $\vec \Phi_k \in \mc E_{B_{R_k}\setminus B_{r_k}} $ with $R_k/r_k \geq k$ and conformal factor $\lambda_k$ such that
 \[\|\nabla \lambda _k\|_{L^{2,\infty}(B_{R_k}\setminus B_{r_k})} \leq \Lambda\]
 and
\begin{equation}
\sup_{r \in (r_k, R_k/2)} \int _{B_{2r}\setminus B_{r}} |\nabla \vec n_k|^2 \dif 	x < \frac{1}{k}\label{Dirichlet energy bounded by 1 over k}
\end{equation}
 such that \eqref{weakened version of pointwise gradient estimate} does not hold for $\vec \Phi_k$. Recall that by \cite[Lemma V.2]{BernardRiviere}, it holds for $\frac{4r_k}{R_k} < \alpha < 1$ and $r\in (4 r_k, \alpha R_k)$ that
 \begin{equation}
 \|\lambda_{k} - \bar \lambda_{k}(r)\|_{L^\infty(B_{\alpha^{-1} r}\setminus B_{r})} \leq C(\alpha,\Lambda),\label{local Harnack estimate}
 \end{equation}
 where $\bar \lambda_{k}(r) \cqq \dashint _{B_{\alpha^{-1} r}\setminus B_{r}} \lambda_k \dif x$ is the average. Let $s_k \in (r_k, R_k)$ be such that
 \begin{equation}
 \frac{s_k}{r_k} \to \infty, \quad \frac{R_k}{s_k} \to \infty.\label{sn stays away from boundary}
 \end{equation}
 Consider $\vec \Psi_k(z) \cqq e^{-\lambda_k(s_k) - \log(s_k)}(\vec \Phi_k(zs_k) - \vec \Phi_k((s_k,0)))$, which by \eqref{local Harnack estimate} and $\lambda_{\vec \Psi_k}(z) = \lambda_k(zs_k) - \lambda_k(s_k)$ satisfies
 \[ \limsup _{k\to \infty} \|\lambda_{\vec \Psi_k}\|_{L^\infty(K)} \leq C(K,\Lambda)< \infty,\]
 where $K\Subset \C \setminus \left \{0\right \}$. It follows that $\sup_{k\in \N}\|\vec \Psi_k\|_{W^{2,2}(K; \R^3)}\leq C(K,\Lambda)$. In particular, we see that after taking subsequences, 
 \[\vec \Psi_k \wto \vec \Psi_{\infty} \quad \text{in $W^{2,2}_{\loc}(\C\setminus \left \{0\right \};\R^3)$}.\]
 As the Dirichlet energy is lower semicontinuous under weak convergence, we deduce from \eqref{Dirichlet energy bounded by 1 over k} that
 \[\nabla \vec n_{\vec \Psi_{\infty}} \equiv 0,\]
 so that the image of $\vec \Psi_{\infty}$ is a flat plane. After a possible rotation, we can identify the image with $\C$ so that $\vec \Psi_{\infty}\colon\C \setminus \left \{0\right \} \to \C$ is a conformal map and hence either holomorphic or antiholomorphic. By the Liouville equation, the conformal factor $\lambda$ of $\vec \Psi_{\infty}$ is harmonic. As in \eqref{mean value formula estimate for nu}, we deduce $|\nabla \lambda(x)| \leq \frac{C}{|x|}$ for $x\in \C\setminus\left \{0\right \}$. The classification of harmonic functions on annuli yields
\[\lambda = (m-1) \log |\cdot | + \Re (g),\]
where $m \in \R$ and $g = g_1 + i g_2$ is holomorphic on $\C \setminus \left \{0\right \}$. In particular $|\nabla g_1| \leq \frac{C}{|x|}$ for $x\in \C \setminus \left \{0\right \}$. The Cauchy-Riemann equations yields that $|\nabla g_2| \leq \frac{C}{|x|}$ as well and so $|g'| \leq \frac{C}{|x|}$. In particular, $|g(z)| \leq C (1+ |\log(|z|)|)$ for all $z\in \C \setminus \left \{0\right \}$, which implies that $g$ is constant. It follows that $|\vec \Psi_\infty'(z)| = C |z|^{m-1}$. In particular, $m \in \Z\setminus \left \{0\right \}$ and after a potential shift, we have
\[\vec \Psi_{\infty}(z) = \frac{A}{m} z^{m} \]
for some fixed $A \in \C$. Notice that $m$ is bounded in terms of $\Lambda$. Fix $r>0$. As $\partial_1 \vec \Psi_k \wto \partial_1 \vec \Psi_{\infty}$ in $W^{1,2}(B_{2r}\setminus B_r; \R^3)$, the compactness of the trace operator $T\colon W^{1,2}(\Omega) \to L^2(\partial \Omega)$ for bounded $C^1$ domains $\Omega$, see \cite[Theorem 3.85]{Demengel}, yields that
\[e^{\lambda_{\vec \Psi_k}(z)} \to |A| r^{m-1} \quad\text{in $L^2(\partial B_r)$}.\]
The uniform Harnack estimate \eqref{local Harnack estimate} for $\lambda_{\vec \Psi_k}$ yields that
\[\lambda_{\vec \Psi_k}(z) \to \ln |A| + (m-1)\log(r) \quad\text{in $L^2(\partial B_r)$}.\]
We use the notation $\underline \lambda_k(r) \cqq \dashint_{\partial B_r} \lambda_k(x) \dif l(x)$ and $\underline \lambda_{\vec \Psi_k} (r) \cqq \dashint _{\partial B_r} \lambda _{\vec \Psi_k} (x) \dif l(x)$. Then
\begin{equation}
\underline \lambda_{\vec \Psi_k}(r)  \to \log |A| + (m-1) \log(r) \quad\text{pointwise for all $r\in (0,\infty)$}. \label{pointwise convergence}
\end{equation}
Thus, there is $\alpha_0\in (0,1)$ and $k_0$ such that for all $k\geq k_0$ and all $r \in [\alpha^{-1}_0 r_k, \alpha _0 R_k/2]$ there exists $m\in \Z\setminus \left \{0\right \}$ (a priori depending on $k$ \emph{and} $r$) such that we have
\begin{equation}
|\underline \lambda_k(2r) - \underline \lambda_k(r) - (m-1) \log(2)|  \leq \delta\log(2).\label{g is pointwise close to integer}
\end{equation}
Indeed, if this was not the case, for any $k\in \N$ and $\alpha=\frac{1}{k}$, there is $l= l(k)  \geq k$ such that we could find $t_k \in (k r_{l(k)}, k^{-1} R_{l(k)}/2)$ such that
\begin{equation}
|\underline \lambda_{l(k)}(2t_{l(k)}) - \underline \lambda_{l(k)}(t_{l(k)}) - (q-1) \log(2)|  > \delta\log(2)\quad\text{for all $q\in \Z\setminus \left \{0\right \}$}.\label{contradiction with q}
\end{equation}
Choosing $l(k)$ as a subsequence indexed by $k$ and setting $s_k =t_k $, we see that \eqref{sn stays away from boundary} holds. Hence, we can run the previous argumentation to arrive at \eqref{pointwise convergence} for some $m\in \Z\setminus \left \{0\right \}$. We use that $\underline \lambda_{\vec \Psi_k}(2) - \underline \lambda_{\vec \Psi_k}(1) -(m-1) \log(2) = \underline \lambda_k(2t_k) - \underline \lambda_k(t_k) - (m-1)\log(2)$ and \eqref{pointwise convergence} to bring \eqref{contradiction with q} to a contradiction (namely \eqref{contradiction with q} would not hold for $q=m$). 

As $\underline \lambda_k$ is continuous, $\underline \lambda_k(2r) - \underline \lambda_k(r)$ is continuous and a simple intermediate value argument (which is possible due to $\delta < \frac{1}{2}$) yields that $m$ in \eqref{g is pointwise close to integer} does not depend on $r$. In particular, 
\[|\underline \lambda_k(2^k r)- \underline \lambda_k(r) - k(m-1) \log(2)  | \leq k\delta \log(2)\]
for all $k\in \N$ and $r \in [\alpha_0^{-1} r_k, \alpha_0 2^{-k-1} R_k]$. \eqref{local Harnack estimate} yields 
\[|\underline \lambda_k(s) - (m-1) \log(s/r) - \underline \lambda_k(r)| \leq   \delta |\log(s/r)| + C(\Lambda),\quad s,r\in [2\alpha_0^{-1} r_k, \alpha_0  R_k/2],\]
which, together with \eqref{local Harnack estimate}, contradicts the assumption that \eqref{weakened version of pointwise gradient estimate} is violated for $k > 2 \alpha_0^{-1}$. \end{proof}

\renewcommand{\thesection}{Appendix \Alph{section}}
\tocless\section{Ends and branch points}
\addcontentsline{toc}{section}{Appendix B. Ends and branch points}
\renewcommand{\thesection}{\Alph{section}}
\label{sec:Appendix Ends and branch points}
Let us recall the behavior of conformal, branched weak immersions around the singular points. Suppose $\vec \Phi\colon\mb D  \to \R^3$ is a conformal, branched weak immersion with a singularity at $0$ and conformal factor $\lambda$ such that $\vec \Phi^* g_{\R^3} = e^{2\lambda} g_{\R^2}$. Suppose that the metric $e^{2\lambda} g_{\R^2}$ is complete at $0$. Then, it follows by the work of Müller and \v Sver\'ak \cite[Theorem 4.2.1]{MüllerSverak} that there is an integer $m\leq -1$ such that
\begin{equation}
\|\lambda - (m-1) \log | \cdot| \|_{L^\infty(B_{1/2})} < \infty.\label{Muller Sverak estimate}
\end{equation}
In this case, we say that $\vec \Phi$ has an end of order $m$ at 0. This result should be compared to \cite[Theorem~7]{NguyenGeometricRigidity}, which proves \eqref{Muller Sverak estimate} under the assumption that $e^{-\lambda} \in L^2(B_{1/2})$ instead of $e^{2\lambda} g_{\R^2}$ being complete.

If instead of the metric being complete, we assume that $e^\lambda \in L^2(\mb D)$, then \cite[Theorem 3.1]{KuwertLi}, see also \cite[Lemma 6.2]{RiviereLectureNotes}, show that \eqref{Muller Sverak estimate} holds for some $m\geq 1$. In this case, we say that $\vec \Phi$ has a branch point of order $m$ at $0$.

Notice that in the case $m=1$, we still say that $\vec \Phi$ has a branch point at $0$ even though the singularity is removable in the sense that $\vec \Phi$ is an immersion throughout the origin. For us, it will be more convenient to combine these different assumptions by assuming that the gradient of the conformal factor is well-behaved, namely that $\|\nabla \lambda\|_{L^{2,\infty}(\mb D)} < \infty$.
\begin{lemma} \label{lem: Muller Sverak variant}
Suppose $\vec \Phi\colon\mb D  \to \R^3$ is a conformal, branched weak immersion such that
\begin{equation}
\|\nabla \lambda\|_{L^{2,\infty}(\mb D)}+\|\nabla \vec n\|_{L^2(\mb D)}< \infty.\label{L2 weak norm for lambda and L2 norm for n}
\end{equation}
Then \eqref{Muller Sverak estimate} holds for some $m\in \Z \setminus \left \{0\right \}$. Furthermore,
\begin{equation}
\lim_{z\to 0}\frac{|\vec \Phi(z)|}{|z|^m} = \frac{e^{\omega}}{-m}\quad\text{if $m\leq -1$},\quad \lim_{z\to 0}\frac{|\vec \Phi(z) - \vec \Phi(0)|}{|z|^m} = \frac{e^{\omega}}{m}\quad\text{if $m\geq 1$}, \label{growth of immersion around end}
\end{equation}
where $\omega \cqq \lim_{z\to 0} \lambda(z) - (m-1) \log |z|$.
\end{lemma}
\begin{proof} 
The proof mostly follows \cite{MüllerSverak,KuwertLi}. We may assume that $\|\nabla \vec n \|_{L^2(\mb D)}\leq \frac{8\pi}{3}$. We can find $v \in C^0(\mb D) \cap W^{1,2}(\mb D)$, $\alpha \in \R$ and a holomorphic function $\phi$ on $\mb D \setminus \{0\}$ with $\Re \phi = h$ such that
\[\lambda(z) = v(z) + \alpha \log |z| + h(z).\]
The assumption $\nabla \lambda \in L^{2,\infty}(\mb D)$ implies that $\nabla h \in L^{2,\infty}(\mb D)$ as well. The mean-value formula applied to $\nabla h$ together with the Cauchy--Schwarz inequality (for Lorentz spaces) implies for $0<r<|x|<\frac{1}{2}$
\begin{equation}
|\nabla h(x)| = \left |\dashint _{B_{r}(x)} \nabla h (y) \dif y\right | \leq \frac{1}{\pi r^2} \|\nabla h\|_{L^1(B_r(x))} \leq \frac{C}{\pi r^2}  \mc L^2(B_r(x)) ^{1/2} \|\nabla h\|_{L^{2,\infty}(B_r(x))} \leq \frac{C}{r}. \label{mean value formula estimate for nu}
\end{equation}
Choosing $r = \frac{|x|}{2}$ yields $|\nabla h(x)| \leq \frac{C}{|x|}$ for $0<|x|<\frac{1}{2}$. In particular, the Cauchy--Riemann equations yield that $|\phi'(z)| \leq \frac{C}{|z|}$ as well and so $|\phi(z)| \leq C + C|\log(|z|)|$ for all $z\in B_{1/2} \setminus \left \{0\right \}$. This implies that $\phi$ can be extended holomorphically through 0, and is in particular bounded in $B_{1/2}$. The remaining proof can be finished as in \cite{MüllerSverak, KuwertLi}.
\end{proof}

\begin{remark} \label{rem:order one end is embedded}
When $\vec \Phi\colon\hat \C \setminus \mb D\to \R^3$ and it holds
\[ \|\nabla \lambda\|_{L^{2,\infty}(\C \setminus \mb D)}+\|\nabla \vec n\|_{L^2(\C \setminus \mb D)}< \infty, \]
we can argue in a similar manner. In this case, \eqref{Muller Sverak estimate} holds on the domain $\C \setminus B_2$ and the analogous estimates
\[\lim_{z\to \infty}\frac{|\vec \Phi(z) - \vec \Phi(\infty)|}{|z|^m} = \frac{e^{\omega}}{-m}\quad\text{if $m\leq -1$},\quad \lim_{z\to \infty}\frac{|\vec \Phi(z)|}{|z|^m} = \frac{e^{\omega}}{m}\quad\text{if $m\geq 1$}, \]
$\omega = \lim_{z\to \infty} \lambda(z) - (m-1) \log |z|$ hold. In this case, we say that $\vec \Phi$ has a branch point/end of order $-m$ at $\infty$.

\end{remark}

\renewcommand{\thesection}{Appendix \Alph{section}}
\tocless\section{Equality in the Li--Yau inequality}
\addcontentsline{toc}{section}{Appendix C. Equality in the Li--Yau inequality}
\renewcommand{\thesection}{\Alph{section}}
\label{sec:Appendix Equality in the Li--Yau inequality}
Suppose $U\subset \R^n$ is open, $2\leq n$, and $\mu$ is a Radon measure over $U$. An integral $2$-varifold $\mu$ in $U$ is a Radon measure over $U$ expressed in the form $\mu = \theta \mc H^2\mres M$, where $\theta \in L^1_{\loc}(U,\mc H^2;\N_0)$, and $M = \{\theta > 0\}\subset U$ is $\mc H^2$-rectifiable. Suppose that $W\subset \R^n$ is open and ${\vec \Phi}\colon\spt \mu \cap U \to W$ is proper, locally Lipschitz and injective. Then the image varifold ${\vec \Phi}_{\#} \mu$ in $W$ is defined as
\[\vec \Phi_{\#}\mu = \theta \circ \vec \Phi^{-1} \mc H^2 \mres \vec \Phi(M).\]
As a consequence of the area formula, see \cite[15.6]{SimonGMT}, it holds for $K\subset W$ compact
\begin{equation}
\int _{\vec \Phi(M)\cap K} \theta \circ \vec \Phi^{-1} \dif \mc H^2 = \int _{M\cap \vec \Phi^{-1}(K)}J^M_{\vec \Phi} \theta \dif \mc H^2,\label{area formula}
\end{equation}
where $J^M_{{\vec \Phi}}$ denotes the Jacobian of ${\vec \Phi}$ relative to $M$ \cite[15.7]{SimonGMT}.
We denote for $f \in C_c^1(U;\R^n)$
\begin{equation}
(\Div_\mu f)(x) \cqq \sum _{i=1}^2 \langle Df(x)(b_i), b_i\rangle,\label{divergence}
\end{equation}
where $\{b_1,b_2\}$ is an orthonormal basis of the approximate tangent space $T_x\mu$, which exists $\mu$-a.e.\ in $U$. We say that $\mu$ has generalized mean curvature $\vec H$ in $U$, if there is $\vec H\in L^1_{\loc}(U,\mu;\R^n)$ such that
\begin{equation}
\int_U \Div_\mu f \dif \mu = - 2 \int_U \langle f, \vec H\rangle \dif \mu\quad \text{for all $f\in C_c^1(U;\R^n)$.}\label{divergence theorem}
\end{equation}
We define the Willmore energy for such varifolds by
\[\mc W(\mu) \cqq \int_U |\vec H|^2 \dif \mu.\]
We call $\mu$ stationary in $U$, if $\mu$ has vanishing generalized mean curvature $\vec H=0$. 

Suppose $U=\R^n$, $\mc W(\mu)< \infty$, and $\Theta^2(\mu, \infty) \cqq \lim _{r\to \infty} \frac{1}{\pi r^2} \mu(B_r(0)) = 0$. By \cite[(1.2')]{Simon}, see also \cite[(A.3),(A.5)]{Kuwert2004removability}, the well-known monotonicity identity holds:
\begin{equation}  
\pi \Theta^2(\mu, x_0) + \int _{\R^n} \left | \frac{\vec H}{2} + \frac{(x-x_0)^\perp}{|x-x_0|^2} \right |^2 \dif \mu = \frac{1}{4} \mc W(\mu)\quad \text{for all $x_0 \in \R^n$}.\label{monotonicity identity}
\end{equation}
Here, $\Theta^2(\mu, x_0) = \lim_{r\to 0} \frac{1}{\pi r^2} \mu(B_r(x_0))$ and ${}^\perp$ is the projection onto the normal space $(T_x\mu)^\perp$. From this, the seminal Li--Yau inequality
\begin{equation}
\Theta^2(\mu, x_0) \leq \frac{1}{4\pi } \mc W(\mu)\label{Li Yau inequality}
\end{equation}
may be deduced. In particular, we see that equality in \eqref{monotonicity identity} holds if and only if
\begin{equation}
\vec H = -2 \frac{(x-x_0)^\perp}{|x-x_0|^2}\quad \text{$\mu$-a.e.\ in $\R^n$}.\label{mean curvature for equality in Li Yau}
\end{equation}
We define the inversion $I_{x_0}\colon \R^n\setminus \{x_0\}\to \R^n \setminus \{0\}$ given by $I_{x_0}(y) \cqq \frac{y-x_0}{|y-x_0|^2}$. 

\begin{proof}[Proof of \zcref{thm:equality-in-Li-Yau}]
Without loss of generality, we may assume $x_0=0$. We will first show that
\[\int_{\R^n} \Div _\nu f \dif \nu = 0\quad\text{for all $f\in C_c^1(\R^n\setminus \{0\};\R^n)$}\]
if and only if $\vec H = - 2 \frac{x^\perp}{|x|^2}$ $\mu$-a.e.\ in $\R^n$. Define $g(x) \cqq DI_{0}(x)^{-1}(f(I_{0}(x)))$ for $x\in \R^n \setminus \{0\}$. Then, $f(I_{0}(x)) = DI_{0}(x)(g(x))$ and 
\begin{align}
Df(I_{0}(x))(DI_{0}(x)(v)) &= D(f\circ I_{0})(x)(v) = D(DI_{0}(\cdot)(g(\cdot)))(x)(v)\notag \\
&= D^2I_{0}(x)(g(x),v) + DI_{0}(x)(Dg(x)(v)) \label{Calculation for derivatives}.
\end{align}
Suppose that $\{\tau_1,\tau_2\}$ is an orthonormal basis of $T_{x}\mu$. Then $T_{I_{0}(x)}\nu$ exists and has an orthonormal basis given by $\{ |x|^2DI_{0}(x)(\tau_1),|x|^2DI_{0}(x)(\tau_2)\}$. Using \eqref{divergence} and \eqref{Calculation for derivatives}, it follows that the divergence is given by
\begin{align}
(\Div_\nu f)(I_{0}(x)) &= \sum_{i=1}^2 \langle Df(I_{0}(x))(|x|^2 DI_{0}(x)(\tau_i)), |x|^2 DI_{0}(x)(\tau_i)\rangle\\
 &=|x|^4\sum_{i=1}^2 \langle D^2I_{0}(x)(g(x),\tau_i) + DI_{0}(x)(Dg(x)(\tau_i)) ,  DI_{0}(x)(\tau_i)\rangle.
\end{align}
A simple calculation shows that $D^2I_{0}$ is given by
\begin{align}
D^2 I_{0}(x)(v,w) &= -2 \frac{\langle x, w\rangle}{|x|^2} DI_0(x)(v) - 2\frac{\langle x, v\rangle }{|x|^2} DI_0(x)(w) - 2\frac{x}{|x|^4} \langle v, w\rangle .
\end{align}
Hence,
\begin{align}
(\Div_\nu f)(I_{0}(x)) &=|x|^4\sum_{i=1}^2 \Bigg\langle -2 \frac{\langle x, \tau_i\rangle}{|x|^2} DI_0(x)(g(x)) - 2\frac{\langle x, g(x)\rangle }{|x|^2} DI_0(x)(\tau_i) \\
&\quad\,- 2\frac{x}{|x|^4} \langle g(x), \tau_i\rangle  + DI_{0}(x)(Dg(x)(\tau_i)) ,  DI_{0}(x)(\tau_i) \Bigg\rangle
\end{align}
Using that $x = - |x|^2 DI_0(x)(x)$ and $\langle DI_0(x)(v), DI_0(x)(w)\rangle = \frac{1}{|x|^4} \langle v,w\rangle$ for any $v,\, w \in \R^n$, we deduce
\begin{align}
(\Div_\nu f)(I_{0}(x)) &= \sum_{i=1}^2 \Bigg\langle -2 \frac{\langle x, \tau_i\rangle}{|x|^2}  g(x)  - 2\frac{\langle x, g(x)\rangle }{|x|^2} \tau_i  + 2\frac{  x}{|x|^2} \langle g(x), \tau_i\rangle  +  Dg(x)(\tau_i)  , \tau_i \Bigg\rangle\notag\\
&=(\Div_\mu g)(x) -4 \frac{\langle x, g(x)\rangle}{|x|^2}\label{divergence for inversion}.
\end{align}
By \eqref{area formula} and the fact that $I_{0}$ is conformal, we have
\begin{align}
\int _{\R^n} (\Div_{\nu} f)(y) \dif \nu(y) &= \int _{\R^n} \frac{(\Div_\mu g)(x)}{|x|^4}   -4 \frac{\langle x, g(x)\rangle}{|x|^6} \dif \mu(x)\notag\\
&= \int _{\R^n} \Div_\mu \left ( \frac{g(x)}{|x|^4}\right ) +4 \sum_{i=1}^2\frac{\langle x,\tau_i\rangle \langle \tau_i, g(x)\rangle}{|x|^6}  -4 \frac{\langle x, g(x)\rangle}{|x|^6}\dif \mu(x)\notag\\
&= \int _{\R^n} -2 \frac{\langle g(x), \vec H(x)+2\frac{x^\perp}{|x|^2}\rangle}{|x|^4} \dif \mu(x)\label{divergence theorem for inversion}.
\end{align}
We see that $\nu$ is stationary if and only if $\vec H = -2\frac{x^\perp}{|x|^2}$ $\mu$-a.e., which is precisely the condition \eqref{mean curvature for equality in Li Yau}. Furthermore, $\Theta^2(\mu, \infty) = 0$ implies $\Theta^2(\nu,0) = 0$, so that $\nu$ is a stationary integral 2-varifold in $\R^n$, see \cite[(A.2)]{Kuwert2004removability}.

It remains to compare the densities. Define the map $\mu_r$ by $\mu_{r}(x) \cqq rx$. By \cite[3.4(1)(2),4.12(2),6.5]{Allard}, we find a sequence $r_k \to \infty$ such that
\[\mu_{r_k\#} \mu \wto C,\]
i.e.,
\begin{equation}
C(\phi) \cqq \lim_{k\to \infty} r_k^2 \int_{
M} \phi(r_k x) \dif \mc H^2(x)\quad\text{for all $\phi \in C_c(\R^n)$},\label{weak convergence to cone}
\end{equation}
where $C$ is a stationary integral 2-varifold satisfying $\mu_r C = C$ for all $r>0$ and $\Theta^2(C,0) = \Theta^2(\mu,0)$. The condition $\mu_{r\#} C=C$ implies that $\theta_C(rx)=\theta_C(x)$ for all $r>0$ and thus $I_{0\#} C = C$.
Using \eqref{area formula}, the blow-down of $\nu$ is given by
\begin{align}
(\mu _{1/r_k \#} \nu) (\phi) &= \int _{M}\frac{1}{r_k^2 |x|^4} \phi\left (\frac{x}{r_k |x|^2}\right ) \theta(x) \dif \mc H^2(x).
\intertext{Letting $\psi(x) \cqq \frac{1}{|x|^4} \phi\left (\frac{x}{|x|^2}\right )$, we see that this equals}
&=r_k^2 \int _M \psi(r_k x) \theta(x) \dif \mc H^2(x).
\intertext{In the case that $\phi \in C_c(\R^n \setminus \{0\})$, we get that $\psi \in C_c(\R^n\setminus \{0\})$ as well, so that \eqref{weak convergence to cone} yields}
&\to C(\psi) = I_{0\#} C(\psi) = C(\phi)\quad\text{as $k\to \infty$.}
\end{align}
Let $s>1$. Approximating $\chi_{B_s\setminus B_1}$ by continuous functions and using that $C(\partial (B_s\setminus B_1))=0$, we deduce
\[(\mu_{1/r_k \#}\nu)(B_s \setminus B_1) \to C(B_s \setminus B_1) = (s^2-1)\pi \Theta^2(C,0).\]
On the other hand, as $I_{0\#}\mu$ is minimal, the monotonicity formula \cite[(1.2)]{Simon}, implies that the quantity $f(r) \cqq \frac{1}{\pi r^2} \nu(B_r)$ is non-decreasing, and so
\[(\mu _{1/r_k\#} \nu)(B_s\setminus B_1) = s^2 \pi f(s r_k) - \pi f(r_k) \in ((s^2-1)\pi f(sr_k), s^2 \pi f(sr_k)).\]
In particular, $\Theta^2(\nu,\infty) = \lim _{r\to \infty} f(r) $ exists and satisfies 
\[\Theta^2(\nu,\infty) \leq \Theta^2(C,0) \leq \frac{s^2}{s^2-1} \Theta^2(I_{0\#}\mu,\infty).\]
Since this holds for any $s>1$, we conclude $\Theta^2(C,0) = \Theta^2(I_{0\#}\mu,\infty)$.
\end{proof}

\printbibliography
\end{document}